\documentclass[10pt]{amsart}
      \usepackage{amsmath,amsfonts}
\def\rg{\hbox to 30pt{\rightarrowfill}}
\def\lg{\hbox to 30pt{\leftarrowfill}}

      \parskip\smallskipamount
          \newtheorem{theorem}{Theorem}[section]
      \newtheorem{definition}[theorem]{Definition}
      \newtheorem{proposition}[theorem]{Proposition}
      \newtheorem{corollary}[theorem]{Corollary}
      \newtheorem{lemma}[theorem]{Lemma}

      \newtheorem{remark}[theorem]{Remark}
      \makeatletter
      \@addtoreset{equation}{section}
      \makeatother

      \newcommand{\BB}{{\mathbb B}}
      \newcommand{\CC}{{\mathbb C}}
      \newcommand{\NN}{{\mathbb N}}
      
      \newcommand{\ZZ}{{\mathbb Z}}
      \newcommand{\DD}{{\mathbb D}}
      
      \newcommand{\FF}{{\mathbb F}}

      \newcommand{\cA}{{\mathcal A}}
      
      \newcommand{\cC}{{\mathcal C}}
      \newcommand{\cD}{{\mathcal D}}
      \newcommand{\cE}{{\mathcal E}}
      
      \newcommand{\cG}{{\mathcal G}}
      \newcommand{\cH}{{\mathcal H}}
      \newcommand{\cK}{{\mathcal K}}
      \newcommand{\cL}{{\mathcal L}}
      \newcommand{\cM}{{\mathcal M}}
      \newcommand{\cN}{{\mathcal N}}
      \newcommand{\cQ}{{\mathcal Q}}
      
      \newcommand{\cR}{{\mathcal R}}
      
      \newcommand{\cT}{{\mathcal T}}
      
      \newcommand{\cV}{{\mathcal V}}

      \newcommand{\supp}{\hbox{\rm{supp}}\,}
      \newcommand{\rank}{\hbox{\rm{rank}}\,}

      \newdimen\expt
      \def\boxit#1{\setbox0\hbox{$\displaystyle{#1}$}
            \hbox{\lower.4\expt
       \hbox{\lower3\expt\hbox{\lower\dp0
            \hbox{\vbox{\hrule height.4\expt
       \hbox{\vrule width.4\expt\hskip3\expt
            \vbox{\vskip3\expt\box0\vskip2\expt}%
       \hskip3\expt\vrule width.4\expt}\hrule height.4\expt}}}}}}
      
\begin{document}
       \pagestyle{myheadings}
      \markboth{ Gelu Popescu}{  Curvature invariant on noncommutative  polyballs  }

      \title [  Curvature invariant on noncommutative   polyballs  ]
      { Curvature invariant on noncommutative   polyballs }
        \author{Gelu Popescu}
\date{November 7, 2014, revised version}
      \thanks{Research supported in part by an NSF grant}

      \address{Department of Mathematics, The University of Texas
      at San Antonio \\ San Antonio, TX 78249, USA}
      \email{\tt gelu.popescu@utsa.edu}

\begin{abstract}   In this paper we  develop a theory of curvature  (resp.~multiplicity) invariant
  for  tensor products of full
Fock spaces $F^2(H_{n_1})\otimes \cdots \otimes F^2(H_{n_k})$  and
also for   tensor products  of   symmetric Fock spaces
$F_s^2(H_{n_1})\otimes \cdots \otimes F_s^2(H_{n_k})$. This is an attempt to find a more general framework for these invariants and extend  some of the  results    obtained by Arveson for the symmetric Fock space, by the author and Kribs for the full Fock space,   and by Fang for the Hardy space $H^2(\DD^k)$ over the polydisc.
To prove the existence of the curvature and its basic properties in these  settings requires a new approach  based on noncommutative Berezin transforms and  multivariable operator theory on polyballs and varieties, as well as summability results for completely positive maps.
The results are presented in the more general setting of  regular polyballs.

\end{abstract}

      \maketitle

\section*{Contents}
{\it

\quad Introduction

\begin{enumerate}
   \item[1.]    Curvature invariant on noncommutative polyballs
   \item[2.]    The curvature operator and classification
   \item[3.]    Invariant subspaces and multiplicity invariant
   \item[4.]    Stability, continuity, and multiplicative properties
   \item[5.]  Commutative polyballs and curvature invariant
\end{enumerate}

\quad References

}

\bigskip

\section*{Introduction}

In \cite{Arv2} and \cite{Arv3},  Arveson introduced and studied a notion of
curvature for finite rank contractive Hilbert modules over
 $\CC[z_1,\dots, z_n]$, which is basically a numerical  invariant for
  commuting  $n$-tuples  $T:=(T_1,\ldots, T_n)$ in the unit ball
  $$
 [B(\cH)^{n}]_1^-:=\left\{
(X_{1},\ldots, X_{n})\in B(\cH)^{n}:\ I-X_{1} X_{1}^*-\cdots -X_{n}X_{n}^*\geq 0
 \right\},
$$
with
  rank\,$\Delta_T<\infty$, where  $\Delta_T:=I-T_1T_1^*-\cdots -T_nT_n^*$ and $B(\cH)$ is the algebra of bounded
linear operators on a Hilbert space $\cH$.
Subsequently, the author \cite{Po-curvature}, \cite{Po-varieties} and, independently,
Kribs \cite{Kr} defined and studied a notion of curvature for
arbitrary  elements in $[B(\cH)^n]_1^-$ and, in particular,  for the full Fock space  $F^2(H_n)$ with $n$ generators. Some of these results were extended by  Muhly and Solel \cite{MuSo4} to a class of completely positive maps on semifinite factors.
The theory of Arveson's curvature on the symmetric Fock space $F_s^2(H_n)$ with $n$ generators was significantly expanded  due to the work  by Greene, Richter, and Sundberg  \cite{GRS},  Fang  \cite{Fang1}, and   Gleason, Richter, and Sundberg  \cite{GlRS}. Engli\v s remarked in \cite{En}  that using Arveson's ideas one can extend the notion of curvature to complete Nevanlinna-Pick kernels. The extension of Arveson's theory to holomorphic spaces with non Nevanlinna-Pick kernels was first considered by
Fang \cite{Fang} who was able to show that the main results about the curvature invariant on the symmetric Fock space carry over, with different proofs and using commutative algebra techniques, to the Hardy space $H^2(\DD^k)$ over the polydisc, in spite of its extremely complicated lattice of invariant subspaces (see \cite{Ru}). Inspired by some results on the invariant subspaces of the Dirichlet shift obtained by Richter \cite{R},
the theory of  curvature invariant was extended to the Dirichlet space  by Fang \cite{Fang2}. In the noncommutative setting, a notion of curvature invariant  for  noncommutative domains generated by positive regular free polynomials  was considered in \cite{Po-domains}.

The goal of the present paper is to develop a theory of curvature invariant for    the  regular polyball
 ${\bf B_n}(\cH)$, which will be introduced below.
 In particular, our results allow one to formulate  a theory of  curvature invariant and multiplicity invariant  for the tensor product of full
Fock spaces $F^2(H_{n_1})\otimes \cdots \otimes F^2(H_{n_k})$  and
also for the  tensor product  of   symmetric Fock spaces
$F_s^2(H_{n_1})\otimes \cdots \otimes F_s^2(H_{n_k})$.
To prove the existence of the curvature and its basic properties in these  settings   requires a new approach  based on noncommutative Berezin transforms and  multivariable operator theory on polyballs and varieties (see \cite{Po-poisson}, \cite{Po-automorphism}, \cite{Po-Berezin-poly}, and \cite{Po-Berezin3}), and also certain summability results for completely positive maps which are trace contractive.
In particular, we obtain new proofs for the existence  of the curvature on  the full Fock space $F^2(H_{n})$,   the Hardy space $H^2(\DD^k)$ (which corresponds
    to $n_1=\cdots=n_k=1$), and the symmetric Fock space $F_s^2(H_{n})$.
We expect
our paper to be a step ahead towards finding a set of numerical (or operatorial) unitary invariants
which  completely classify  the elements of large classes of noncommutative domains
in $B(\cH)^m$  which
admit universal operator models.

  The results of the present paper  play a crucial role in \cite{Po-Euler-charact},
where
 we introduce  and study the Euler characteristic  associated with
  the elements of polyballs, and obtain an analogue of Arveson's  version of Gauss-Bonnet-Chern
   theorem from Riemannian  geometry, which connects the curvature to
   the Euler characteristic of some associated  algebraic modules.
   In particular, we prove that if $\cM$ is an  invariant subspace of
     $F^2(H_{n_1})\otimes \cdots \otimes F^2(H_{n_k})$, $n_i\geq 2$, which is graded
      (generated by multi-homogeneous polynomials),
    then the curvature and the Euler characteristic of the orthocomplement  of $\cM$ coincide.

To present our results, we need some notation and preliminaries.
Throughout this paper, we denote by  $B(\cH)^{n_1}\times_c\cdots \times_c B(\cH)^{n_k}$, where $n_i \in\NN:=\{1,2,\ldots\}$,
   the set of all tuples  ${\bf X}:=({ X}_1,\ldots, { X}_k)$ in $B(\cH)^{n_1}\times\cdots \times B(\cH)^{n_k}$
     with the property that the entries of ${X}_s:=(X_{s,1},\ldots, X_{s,n_s})$  are commuting with the entries of
      ${X}_t:=(X_{t,1},\ldots, X_{t,n_t})$  for any $s,t\in \{1,\ldots, k\}$, $s\neq t$.
  Note that  the operators $X_{s,1},\ldots, X_{s,n_s}$ are not necessarily commuting.
  Let ${\bf n}:=(n_1,\ldots, n_k)$ and define  the {\it regular polyball}
$$
{\bf B_n}(\cH):=\left\{ {\bf X}\in B(\cH)^{n_1}\times_c\cdots \times_c B(\cH)^{n_k}: \ {\bf \Delta_{X}^p}(I)\geq 0 \ \text{ for }\ {\bf 0}\leq {\bf p}\leq (1,\ldots,1)\right\},
$$
where
 the {\it defect mapping} ${\bf \Delta_{X}^p}:B(\cH)\to  B(\cH)$ is defined by
$$
{\bf \Delta_{X}^p}:=\left(id -\Phi_{X_1}\right)^{p_1}\circ \cdots \circ\left(id -\Phi_{ X_k}\right)^{p_k}, \qquad  {\bf p}:=(p_1,\ldots, p_k)\in \ZZ_+^k,
$$
 and
$\Phi_{X_i}:B(\cH)\to B(\cH)$  is the completely positive linear map defined by   $\Phi_{X_i}(Y):=\sum_{j=1}^{n_i}   X_{i,j} Y X_{i,j} ^*$ for $Y\in B(\cH)$.  We use the convention that $(id-\Phi_{f_i,X_i})^0=id$.
 For information on completely bounded (resp. positive) maps we refer the reader to \cite{Pa-book}.

Let $H_{n_i}$ be
an $n_i$-dimensional complex  Hilbert space with orthonormal basis $e^i_1,\ldots, e^i_{n_i}$.
  We consider the {\it full Fock space}  of $H_{n_i}$ defined by
$$F^2(H_{n_i}):=\CC 1 \oplus\bigoplus_{p\geq 1} H_{n_i}^{\otimes p},$$
where  $H_{n_i}^{\otimes p}$ is the
(Hilbert) tensor product of $p$ copies of $H_{n_i}$. Let $\FF_{n_i}^+$ be the unital free semigroup on $n_i$ generators
$g_{1}^i,\ldots, g_{n_i}^i$ and the identity $g_{0}^i$.
  Set $e_\alpha^i :=
e^i_{j_1}\otimes \cdots \otimes e^i_{j_p}$ if
$\alpha=g^i_{j_1}\cdots g^i_{j_p}\in \FF_{n_i}^+$
 and $e^i_{g^i_0}:= 1\in \CC$.
  The length of $\alpha\in
\FF_{n_i}^+$ is defined by $|\alpha|:=0$ if $\alpha=g_0^i$  and
$|\alpha|:=p$ if
 $\alpha=g_{j_1}^i\cdots g_{j_p}^i$, where $j_1,\ldots, j_p\in \{1,\ldots, n_i\}$.
 We  define
 the {\it left creation  operator} $S_{i,j}$ acting on the  Fock space $F^2(H_{n_i})$  by setting
$
S_{i,j} e_\alpha^i:=  e^i_{g_j^i \alpha}$, $\alpha\in \FF_{n_i}^+,
$
 and
 the operator ${\bf S}_{i,j}$ acting on the tensor  product
$F^2(H_{n_1})\otimes\cdots\otimes F^2(H_{n_k})$ by setting
$${\bf S}_{i,j}:=\underbrace{I\otimes\cdots\otimes I}_{\text{${i-1}$
times}}\otimes S_{i,j}\otimes \underbrace{I\otimes\cdots\otimes
I}_{\text{${k-i}$ times}},
$$
where  $i\in\{1,\ldots,k\}$ and  $j\in\{1,\ldots,n_i\}$.
 Let  ${\bf T}=({ T}_1,\ldots, { T}_k)\in {\bf B_n}(\cH)$ with $T_i:=(T_{i,1},\ldots, T_{i,n_i})$.
We  use the notation $T_{i,\alpha_i}:=T_{i,j_1}\cdots T_{i,j_p}$
  if  $\alpha_i=g_{j_1}^i\cdots g_{j_p}^i\in \FF_{n_i}^+$ and
   $T_{i,g_0^i}:=I$.
The {\it noncommutative Berezin kernel} associated with any element
   ${\bf T}$ in the noncommutative polyball ${\bf B_n}(\cH)$ is the operator
   $${\bf K_{T}}: \cH \to F^2(H_{n_1})\otimes \cdots \otimes  F^2(H_{n_k}) \otimes  \overline{{\bf \Delta_{T}}(I) (\cH)}$$
   defined by
   $$
   {\bf K_{T}}h:=\sum_{\beta_i\in \FF_{n_i}^+, i=1,\ldots,k}
   e^1_{\beta_1}\otimes \cdots \otimes  e^k_{\beta_k}\otimes {\bf \Delta_{T}}(I)^{1/2} T_{1,\beta_1}^*\cdots T_{k,\beta_k}^*h,
   $$
where the defect of ${\bf T}$ is defined by
$
{\bf \Delta_{T}}(I)  :=(id-\Phi_{T_1})\circ \cdots \circ(id-\Phi_{T_k})(I).
$
 A very  important property of the Berezin kernel is that
     $${\bf K_{T}} { T}^*_{i,j}= ({\bf S}_{i,j}^*\otimes I)  {\bf K_{T}}.
    $$
    This was used in \cite{Po-Berezin-poly}  to prove that
    ${\bf T}\in B(\cH)^{n_1}\times\cdots \times B(\cH)^{n_k}$
    is a {\it pure} element in the regular polyball  ${\bf B_n}(\cH)$, i.e.  $\lim_{q_i\to \infty}\Phi_{T_i}^{q_i}(I)=0$ in the weak operator topology,  if and only if
there is a Hilbert space $\cK$ and a subspace $\cM\subset F^2(H_{n_1})\otimes \cdots \otimes  F^2(H_{n_k})\otimes \cK$ invariant under each  operator ${\bf S}_{i,j}$ such that
$T_{i,j}^*=({\bf S}_{i,j}^*\otimes I)|_{\cM^\perp}$ under an appropriate idetification of $\cH$ with $\cM^\perp$.
The $k$-tuple ${\bf S}:=({\bf S}_1,\ldots, {\bf S}_k)$, where  ${\bf S}_i:=({\bf S}_{i,1},\ldots,{\bf S}_{i,n_i})$, is  an element  in the
regular polyball $ {\bf B_n}(\otimes_{i=1}^kF^2(H_{n_i}))$ and  plays the role of  {\it universal model} for
   the abstract
  polyball ${\bf B_n}:=\{{\bf B_n}(\cH):\ \cH \text{\ is a Hilbert space} \}$.
For more  results concerning  noncommutative Berezin transforms and multivariable operator theory on noncommutative balls and  polydomains, we refer the reader to \cite{Po-poisson}, \cite{Po-automorphism}, \cite{Po-Berezin-poly}, and \cite{Po-Berezin3}.

In   Section 1,
  we  introduce the {\it curvature}   of any   element  ${\bf T}\in  {\bf B_n}(\cH)$  with trace class defect, i.e. $\text{\rm trace\,}[{\bf \Delta_{T}}(I)]<\infty$,  by setting
\begin{equation*}
\text{\rm curv}\,({\bf T}):=
\lim_{m\to\infty}\frac{1}{\left(\begin{matrix} m+k\\ k\end{matrix}\right)}\sum_{{q_1\geq 0,\ldots, q_k\geq 0}\atop {q_1+\cdots +q_k\leq m}} \frac{\text{\rm trace\,}\left[ {\bf K_{T}^*} (P_{q_1}^{(1)}\otimes \cdots \otimes P_{q_k}^{(k)}\otimes I_\cH){\bf K_{T}}\right]}{\text{\rm trace\,}\left[P_{q_1}^{(1)}\otimes \cdots \otimes P_{q_k}^{(k)}\right]},
\end{equation*}
where ${\bf K_T}$ is the Berezin kernel of ${\bf T}$ and  $P_{q_i}^{(i)}$ is the orthogonal projection of the full Fock space $F^2(H_{n_i})$ onto the span of all vectors $e_{\alpha_i}^i$ with  $\alpha_i\in  \FF_{n_i}^+$ and $|\alpha_i|=q_i$.
The curvature $\text{\rm curv}({\bf T})$ is a unitary invariant for ${\bf T}$ that measures how far  ${\bf T}$ is  from being  ``free'', i.e. a  multiple of the universal model ${\bf S}$.
 We prove several summability results for completely positive maps which are trace contractive. These are used together with the theory of Berezin transforms on noncommutative polyballs to  prove the existence of the curvature
 $\text{\rm curv}({\bf T})$   and established several asymptotic  formulas for the curvature invariant which are very useful later on. It is also shown that the curvature  of  ${\bf T}:=(T_1,\ldots, T_k)$ can be expressed  in terms of
the associated completely positive maps $\Phi_{T_1}, \ldots, \Phi_{T_k}$. In particular, we prove that
$$
\text{\rm curv}\,({\bf T})=\lim_{(q_1,\ldots, q_k)\in \ZZ_+^k} \frac{1}{n_1^{q_1}\cdots n_k^{q_k}}\text{\rm trace\,}\left[ \Phi_{T_1}^{q_1}\circ \cdots \circ \Phi_{T_k}^{q_k}({\bf \Delta_{T}}(I))\right],
$$
which implies the inequalities  $0\leq\text{\rm curv}\,({\bf T})\leq \text{\rm trace\,} [{\bf \Delta_{T}}(I)]\leq \rank [{\bf \Delta_{T}}(I)]$.

 In Section 2, we  introduce the {\it curvature operator} ${\bf \Delta_{S\otimes {\it I}_\cH}}(\bf {K_T K_T^*})({\bf N}\otimes {\it I}_\cH)$ associated with each element ${\bf T}$ in the polyball ${\bf B_n}(\cH)$, which can be seen as a normalized ``differential'' of the Berezin transform.  We show that if ${\bf T}$ has characteristic function and finite rank, i.e.  $\rank {\bf \Delta_{T}}(I)<\infty$, then the curvature operator associated with ${\bf T}$  is trace class and \begin{equation*}\begin{split}
\text{\rm curv}({\bf T})=\text{\rm trace\,}\left[{\bf \Delta_{S\otimes {\it I}_\cH}}(\bf {K_T K_T^*})({\bf N}\otimes {\it I}_\cH)\right].
\end{split}
\end{equation*}
This is used to obtain an index type result for the curvature, namely,
$$
\text{\rm curv}({\bf T})=\rank [{\bf \Delta_{T}}(I)]-\text{\rm trace\,}\left[\Theta_{\bf T}({\bf P}_\CC\otimes I)\Theta_{\bf T}^* ({\bf N}\otimes I_\cH)\right],
$$
where $\Theta_{\bf T}$ is the characteristic function of ${\bf T}$.
As a consequence of these results, we show that  the curvature invariant can be used to detect the elements
${\bf T}\in B(\cH)^{n_1}\times \cdots\times B(\cH)^{n_k}$ which are unitarily equivalent  to ${\bf S}\otimes I_\cK$ for some Hilbert space   with $\dim \cK<\infty$.
      These are precisely  the  pure  elements ${\bf T}$ in the regular polyball  ${\bf B_n}(\cH)$ such that $\rank {\bf \Delta_{T}}(I)$ is finite,   ${\bf \Delta_{S\otimes {\it I}}}(I-{\bf K_{T}}{\bf K_{T}^*})\geq 0$, and
    $$  \text{\rm curv}({\bf T})=\rank [{\bf \Delta_{T}}(I)].
    $$
In this case, the Berezin kernel ${\bf K_T}$ is a unitary operator and
${\bf T}_{i,j}={\bf K_T^*} ({\bf S}_{i,j}\otimes I_{\overline{{\bf \Delta_{T}}(I) (\cH)}}){\bf K_T}.
$

We say that $\cM$
 is an invariant subspace of the tensor product $F^2(H_{n_1})\otimes \cdots \otimes F^2(H_{n_k})\otimes \cH$  or  that $\cM$  is    invariant  under  ${\bf S}\otimes I_\cH$ if it is invariant under each operator
  ${\bf S}_{i,j}\otimes I_\cH$.
In Section 2, we show that the  curvature invariant completely classifies the
 finite rank  Beurling type invariant subspaces  of ${\bf S}\otimes I_\cH$  which do not contain reducing subspaces. In particular, the curvature invariant classifies  the  finite rank  Beurling type invariant subspaces  of $ F^2(H_{n_1})\otimes \cdots \otimes  F^2(H_{n_k})$ (see Theorem \ref{classification}).

  Given  an invariant subspace $\cM$ of the tensor product $F^2(H_{n_1})\otimes\cdots\otimes F^2(H_{n_k})\otimes \cE$, where $\cE$ is a finite dimensional Hilbert space, we introduce its {\it multiplicity}, in Section 3,  by setting
\begin{equation*}
\begin{split}
m(\cM)&:=\lim_{m\to\infty}\frac{1}{\left(\begin{matrix} m+k\\ k\end{matrix}\right)}\sum_{{q_1\geq 0,\ldots, q_k\geq 0}\atop {q_1+\cdots +q_k\leq m}}\frac{\text{\rm trace\,}\left[P_{\cM}(P_{q_1}^{(1)}\otimes \cdots \otimes P_{q_k}^{(k)}\otimes I_\cE) \right]}{\text{\rm trace\,}\left[P_{q_1}^{(1)}\otimes \cdots \otimes P_{q_k}^{(k)}\right]}.
\end{split}
\end{equation*}
Analogously, we define $m(\cM^\perp)$ by using $P_{\cM^\perp}$ instead of $P_\cM$.
The multiplicity measures the size of the subspace. We  prove that the multiplicity of $\cM$ exists  and provide several asymptotic formulas for it and  also an important   connection with the curvature invariant, namely,
$$
m(\cM)=\dim \cE -\text{\rm curv}({\bf M}),
$$
where  ${\bf M}$ is the compression of ${\bf S}\otimes I_\cE$ to the orthocomplement of $\cM$.

Unlike the case of the symmetric Fock space and the Hardy space over the polydisc when the curvature  and the multiplicity turn out to be nonnegative  integers (see \cite{GRS} and \cite{Fang}), for the tensor product of full Fock spaces,  we prove that they can be any nonnegative real number, as long as at least one of the Fock spaces has  at least $2$  generators (see Corollary \ref{range}). Moreover, we show that if
${\bf n}=(n_1,\ldots, n_k)\in \NN^k$ is  such that $n_i\geq 2$ and $n_j\geq 2$  and  $i\neq j$, then, for each $t\in(0,1)$,  there exists an uncountable family $\{T^{(\omega)}(t)\}_{\omega\in \Omega}$ of  non-isomorphic pure elements  of rank  one defect  in the regular polyball  such that
$$
\text{\rm curv} (T^{(\omega)}(t))=t, \qquad  \text{ for all } \omega\in \Omega.
$$
We also show that  the curvature invariant detects  the {\it inner sequences} of multipliers of $\otimes_{i=1}^k F^2(H_{n_i})$ when it takes the extremal value zero.

In Section 4, we consider several properties concerning the stability, continuity, and multiplicative   properties for the curvature and multiplicity. In particular, we show that  the multiplicity invariant is lower semi-continuous, i.e., if
 $\cM$ and $\cM_m$ are  invariant subspaces of
$\otimes_{i=1}^k F^2(H_{n_i})\otimes \cE$ with $\dim \cE<\infty$
and  $\text{\rm WOT-}\lim_{m\to\infty}P_{\cM_m}=P_\cM$, then
$$
\liminf_{m\to\infty}m(\cM_m)\geq m(\cM).
$$

In Section 5,  we introduce a curvature invariant associated with the elements of the  commutative polyball ${\bf B_n^c}(\cH)$  with the property  that they  have finite rank defects and constrained characteristic functions. Since the case $n_1=\cdots=n_k=1$ is considered in the previous sections, we assume, throughout this section,  that at least one $n_i\geq 2$. The  commutative polyball ${\bf B_n^c}(\cH)$ is the set of all ${\bf X}=(X_1,\ldots, X_k)\in {\bf B_n}(\cH)$ with $X_i=(X_{i,1},\ldots, X_{i,n_i})$, where  the entries $X_{i,j}$ are commuting operators. According to \cite{Po-Berezin3}, the universal model associated with the abstract commutative polyball ${\bf B_n^c}$ is the $k$-tuple
 ${\bf B}:=({\bf B}_1,\ldots, {\bf B}_k)$   with ${\bf B}_i=({\bf B}_{i,1},\ldots, {\bf B}_{i, n_i})$, where
the operator ${\bf B}_{i,j}$ is acting on the tensor product of symmetric Fock spaces
$F_s^2(H_{n_1})\otimes\cdots\otimes F_s^2(H_{n_k})$  and is defined by setting
$${\bf B}_{i,j}:=\underbrace{I\otimes\cdots\otimes I}_{\text{${i-1}$
times}}\otimes B_{i,j}\otimes \underbrace{I\otimes\cdots\otimes
I}_{\text{${k-i}$ times}},
$$
where ${B}_{i,j}$ is the compression of the left creation operator  ${S}_{i,j}$ to the symmetric Fock space  $F_s^2(H_{n_i})$.   For basic results concerning constrained  Berezin transforms and   multivariable model theory on the commutative polyball ${\bf B_n^c}(\cH)$ we refer the reader to \cite{Po-Berezin3}.
  All the results of this section are under the assumption that  the elements ${\bf T} \in  {\bf B^c_n}(\cH)$  have characteristic  functions.

Given an element ${\bf T}$ in the  commutative polyball ${\bf B_n^c}(\cH)$  with the property  that  it has finite rank defect and constrained characteristic function, we introduce  its curvature
 by setting
$$
\text{\rm curv}_c({\bf T}):=
\lim_{m\to\infty}\frac{1}{\left(\begin{matrix} m+k\\ k\end{matrix}\right)}\sum_{{q_1\geq 0,\ldots, q_k\geq 0}\atop {q_1+\cdots +q_k\leq m}} \frac{\text{\rm trace\,}\left[ {\bf \widetilde K_{T}^*} (Q_{q_1}^{(1)}\otimes \cdots \otimes Q_{q_k}^{(k)}\otimes I_\cH){\bf \widetilde K_{T}}\right]}{\text{\rm trace\,}\left[Q_{q_1}^{(1)}\otimes \cdots \otimes Q_{q_k}^{(k)}\right]},
$$
where  $Q_{q_i}^{(i)}$ is the orthogonal projection of the symmetric Fock space ${F_s^2(H_{n_i})}$ onto its subspace of homogeneous polynomials of degree $q_i$. In spite of many similarities  with the noncommutative  case, the proof of the existence of $\text{\rm curv}_c({\bf T})$ is quite  different from the one for $\text{\rm curv}({\bf T})$ and depends on the theory of characteristic functions. We obtain commutative analogues of Theorem \ref{index}, Corollary \ref{curva-maps},  Theorem \ref{comp-inv},  and Theorem \ref{multiplicity}.
In particular,  we mention that the curvature  of  ${\bf T}:=(T_1,\ldots, T_k)$ can be expressed  in terms of
the associated completely positive maps $\Phi_{T_1}, \ldots, \Phi_{T_k}$ by the  asymptotic  formula
$$
\text{\rm curv}_c({\bf T})=n_1 !\cdots n_k !
 \lim_{q_1\to\infty}\cdots\lim_{q_k\to\infty}
\frac{\text{\rm trace\,}\left[(id-\Phi_{T_1}^{q_1+1})\circ\cdots\circ (id-\Phi_{T_k}^{q_k+1})(I)\right]}
{q_1^{n_1}\cdots q_k^{n_k}}.
$$
When $k=1$, we recover Arveson's asymptotic formula for the curvature.

Given a  Beurling type invariant subspace  $\cM$   of the  tensor product $F_s^2(H_{n_1})\otimes\cdots\otimes F_s^2(H_{n_k})\otimes \cE$, where $\cE$ is a finite dimensional Hilbert space,
 we introduce  its multiplicity   by setting
$$
m_c(\cM):=
\lim_{m\to\infty}\frac{1}{\left(\begin{matrix} m+k\\ k\end{matrix}\right)}\sum_{{q_1\geq 0,\ldots, q_k\geq 0}\atop {q_1+\cdots +q_k\leq m}} \frac{\text{\rm trace\,}\left[ P_\cM (Q_{q_1}^{(1)}\otimes \cdots \otimes Q_{q_k}^{(k)}\otimes I_\cE) \right]}{\text{\rm trace\,}\left[Q_{q_1}^{(1)}\otimes \cdots \otimes Q_{q_k}^{(k)}\right]}.
$$
 We show that the multiplicity invariant exists for Beurling type invariant subspace.
 It remains an open problem whether  the multiplicity invariant  exists for arbitrary invariant subspace
of the tensor product $F_s^2(H_{n_1})\otimes\cdots\otimes F_s^2(H_{n_k})\otimes \cE$.
This is true for the polydisc, when $n_1=\cdots =n_k=1$, and  for the symmetric Fock space, when $k=1$.
We remark that there are commutative analogues of all the results from Section 4, concerning the stability, continuity, and multiplicative properties of the curvature and multiplicity invariants.

Regarding the results of Section 5, when  $k\geq 2$ and at least one $n_i\geq 2$, an important problem remains open. Can one drop the condition that the elements of the commutative polyball  have constrained characteristic functions $?$  In case  of a positive answer, the multiplicity invariant would exist  for any invariant subspace
of $F_s^2(H_{n_1})\otimes\cdots\otimes F_s^2(H_{n_k})\otimes \cE$.
We should mention that   $F_s^2(H_{n_1})\otimes\cdots\otimes F_s^2(H_{n_k})$ can be seen as a reproducing kernel Hilbert space of holomorphic functions on $\BB_{n_1}\times\cdots \times \BB_{n_k}$ (see \cite{Po-Berezin3}), where $\BB_n:=\{z\in \CC^n: |z|<1\}$. Under this identification, one can obtain an integral type  formula for the curvature, similar to the one introduced by Arveson \cite{Arv2}. Having in mind the  results on the curvature  invariant on the symmetric Fock space (see \cite{Arv2}, \cite{Arv3}, \cite{GRS}, \cite{Fang1},  \cite{GlRS}), significant  problems still remain open in the setting of Section 5.

Finally, we remark that one can re-formulate the results of the paper in terms of Hilbert modules \cite{DoPa} over  the complex semigroup algebra $\CC[\FF_{n_1}^+\times \cdots \times \FF_{n_k}^+]$ generated by the direct product of the free semigroups $\FF_{n_1}^+, \ldots ,\FF_{n_k}^+$. In this setting, the Hilbert module associated with the universal model ${\bf S}$ acting on  the  tensor product  $F^2(H_{n_1})\otimes \cdots \otimes F^2(H_{n_k})$ plays the role of rank-one free module in the algebraic theory \cite{K}. The commutative case can be re-formulated in a similar manner.

\bigskip

\section{Curvature invariant on noncommutative polyballs}

In this section we  introduce the curvature invariant associated with   the   elements  of the polyball ${\bf B_n}(\cH)$  which have trace class defects.  We prove several summability results for completely positive maps which are trace contractive. These are used together with the theory of Berezin transforms on noncommutative polyballs to  prove the existence of the curvature  and establish several asymptotic  formulas which will be very useful in the coming sections. It is also shown that the curvature  of  ${\bf T}:=(T_1,\ldots, T_k)$ can be expressed  in terms of
the associated completely positive maps $\Phi_{T_1}, \ldots, \Phi_{T_k}$ .

 Given two $k$-tuples ${\bf q}=(q_1,\ldots, q_k)$ and ${\bf p}=(p_1,\ldots, p_k)$ in $\ZZ_+^k$, we consider the partial order ${\bf q}\leq {\bf p}$ defined by $q_i\leq p_i$ for any $i\in \{1,\ldots, k\}$. We consider $\ZZ_+^k$ as a directed set with respect to this partial order. Denote by $\cT(\cH)$ the ideal of trace class operators on the Hilbert space $\cH$ and let $\cT^+(\cH)$ be its positive cone.

\begin{lemma}\label{basic} Let $\phi_1,\ldots, \phi_k$ be commuting positive linear maps on $B(\cH)$ such that $\phi_i(\cT^+(\cH))\subset \cT^+(\cH)$ and
$$
\text{\rm trace\,} [\phi_i(X)]\leq \text{\rm trace\,} (X)
$$
for any $X\in \cT^+(\cH)$ and $i\in \{1,\ldots, k\}$.
Then the limit

\begin{equation*}
  \lim_{(q_1,\ldots, q_k)\in \ZZ_+^k} \text{\rm trace\,}\left[ \phi_1^{q_1}\circ \cdots \circ \phi_k^{q_k}(X)\right]
\end{equation*}
exists and is equal to
\begin{equation*}
\lim_{m\to\infty}\frac{1}{\left(\begin{matrix} m+k-1\\ k-1\end{matrix}\right)}\sum_{{q_1\geq 0,\ldots, q_k\geq 0}\atop {q_1+\cdots +q_k=m}} \text{\rm trace\,}\left[ \phi_1^{q_1}\circ \cdots \circ \phi_k^{q_k}(X)\right]
\end{equation*}
for any $X\in \cT^+(\cH)$.
\end{lemma}
\begin{proof} Due to the hypotheses,
if $Y\in \cT^+(\cH)$ and $q_i\in \ZZ_+:=\{0,1,\ldots\}$, then $\phi_i^{q_i}(Y)\in \cT^+(\cH)$ and
$
0\leq  \text{\rm trace\,}[\phi_i^{q_i+1}(Y)]\leq  \text{\rm trace\,}[\phi_i^{q_i}(Y)].
$
 Consequently, $\lim_{q_i\to\infty}\text{\rm trace\,}[\phi_i^{q_i}(Y)]$ exists. Assume that $1\leq p< k$ and
\begin{equation}
\label{lim}
 \lim_{q_1\to\infty}\cdots\lim_{q_p\to\infty} \text{\rm trace\,}\left[ \phi_1^{q_1}\circ \cdots \circ \phi_p^{q_p}(Y)\right]\quad \text{exists}
\end{equation}
for any $Y\in \cT^+(\cH)$. Using the fact that  $\phi_1,\ldots, \phi_k$ are commuting and
$$
\text{\rm trace\,}\left[\phi_{p+1}^{q_{p+1}+1}( \phi_1^{q_1}\circ \cdots \circ \phi_p^{q_p}(X))\right]
\leq \text{\rm trace\,}\left[\phi_{p+1}^{q_{p+1}}( \phi_1^{q_1}\circ \cdots \circ \phi_p^{q_p}(X))\right],
$$
we can apply  relation \eqref{lim} when $Y$ is equal to  $\phi_{p+1}^{q_{p+1}}(X)$ or $\phi_{p+1}^{q_{p+1}+1}(X)$ and deduce that
$$
0\leq  \lim_{q_p\to\infty}\cdots\lim_{q_1\to\infty}
\text{\rm trace\,}\left[  \phi_1^{q_1}\circ \cdots \circ \phi_p^{q_p}(\phi_{p+1}^{q_{p+1}+1}(X))\right]
\leq
\lim_{q_p\to\infty}\cdots\lim_{q_1\to\infty}
\text{\rm trace\,}\left[  \phi_1^{q_1}\circ \cdots \circ \phi_p^{q_p}(\phi_{p+1}^{q_{p+1}}(X))\right]
$$
for any $q_{p+1}\in \ZZ_+$. Consequently,
$
 \lim_{q_{p+1}\to\infty}\cdots\lim_{q_1\to\infty} \text{\rm trace\,}\left[ \phi_1^{q_1}\circ \cdots \circ \phi_p^{q_{p+1}}(X)\right]$ exists
for any $X\in \cT^+(\cH)$. Since the sequence
$\left\{ \text{\rm trace\,}\left[ \phi_1^{q_1}\circ \cdots \circ \phi_k^{q_k}(X)\right]\right\}_{q_1,\ldots,q_k}$ is decreasing with respect to each of the indices $q_1,\ldots, q_k$, it is clear that
\begin{equation*}
\begin{split}
\lim_{q_1\to\infty}\cdots\lim_{q_k\to\infty} \text{\rm trace\,}\left[ \phi_1^{q_1}\circ \cdots \circ \phi_k^{q_k}(X)\right]
&= \lim_{(q_1,\ldots, q_k)\in \ZZ_+^k} \text{\rm trace\,}\left[ \phi_1^{q_1}\circ \cdots \circ \phi_k^{q_k}(X)\right]\\
&= \inf_{(q_1,\ldots, q_k)\in \ZZ_+^k} \text{\rm trace\,}\left[ \phi_1^{q_1}\circ \cdots \circ \phi_k^{q_k}(X)\right]
\end{split}
\end{equation*}
and the order of the iterated limits does not matter. Set $L(X):=\inf_{(q_1,\ldots, q_k)\in \ZZ_+^k} \text{\rm trace\,}\left[ \phi_1^{q_1}\circ \cdots \circ \phi_k^{q_k}(X)\right]$.
Given $\epsilon >0$ let $N_0\in \NN$ be such that
$
\left|\text{\rm trace\,}\left[ \phi_1^{q_1}\circ \cdots \circ \phi_k^{q_k}(X)\right]-L(X)\right|< \epsilon$ for any $q_1\geq N_0,\ldots q_k\geq N_0$.
Consider the following sets:
\begin{equation*}
\begin{split}
A_1&:=\left\{(q_1,\ldots,q_k)\in \ZZ_+^k:\  q_1+\cdots + q_k=m, q_1<N_0\right\}\\
A_2&:=\left\{(q_1,\ldots,q_k)\in \ZZ_+^k:\  q_1+\cdots + q_k=m, q_1\geq N_0, q_2<N_0\right\}\\
&\  \cdots\\
A_k&:=\left\{(q_1,\ldots,q_k)\in \ZZ_+^k:\  q_1+\cdots + q_k=m, q_1\geq N_0, ,\ldots, q_{k-1}\geq N_0,  q_k<N_0\right\}\\
B_{N_0}&:=\left\{(q_1,\ldots,q_k)\in \ZZ_+^k:\  q_1+\cdots + q_k=m, q_1\geq N_0, ,\ldots, q_{k}\geq N_0\right\}.
\end{split}
\end{equation*}
Note that $A_1,\ldots, A_k, B_{N_0}$ are disjoint sets  and
$$\left\{(q_1,\ldots,q_k)\in \ZZ_+^k:\  q_1+\cdots + q_k=m\right\}=\left(\bigcup_{i=1}^k A_i\right)\cup B_{N_0}.
$$
Regarding the cardinality of these sets, a close look  reveals that
$$
\text{\rm card\,}(\left\{(q_1,\ldots,q_k)\in \ZZ_+^k:\  q_1+\cdots + q_k=m\right\})=\left(\begin{matrix} m+k-1\\ k-1\end{matrix}\right)
$$
and
$
\text{\rm card\,} (A_k)\leq \cdots \leq \text{\rm card\,} (A_1).
$
Moreover, we have
$$
 \text{\rm card\,} (A_1) =\sum_{j=0}^{N_0-1}\left(\begin{matrix} m-j+k-2\\ k-2\end{matrix}\right)\leq N_0\left(\begin{matrix} m+k-2\\ k-2\end{matrix}\right).
 $$
 Consequently, using the fact that
 $\text{\rm trace\,} [\phi_i(X)]\leq \text{\rm trace\,} (X)$
for any $X\in \cT^+(\cH)$ and $i\in \{1,\ldots, k\}$, we obtain
\begin{equation*}
\begin{split}
\sum_{(q_1,\ldots, q_k)\in \cup_{i=1}^k A_i}\text{\rm trace\,}\left[ \phi_1^{q_1}\circ \cdots \circ \phi_k^{q_k}(X)\right]
&=\sum_{i=1}^k\sum_{(q_1,\ldots, q_k)\in  A_i}\text{\rm trace\,}\left[ \phi_1^{q_1}\circ \cdots \circ \phi_k^{q_k}(X)\right]\\
&\leq \sum_{i=1}^k \text{\rm card\,} (A_i)  \text{\rm trace\,}(X)\leq kN_0 \left(\begin{matrix} m+k-2\\ k-2\end{matrix}\right) \text{\rm trace\,}(X).
\end{split}
\end{equation*}
Now, using these inequalities and  that $
\left|\text{\rm trace\,}\left[ \phi_1^{q_1}\circ \cdots \circ \phi_k^{q_k}(X)\right]-L(X)\right|< \epsilon$ for any $q_1\geq N_0,\ldots q_k\geq N_0$, we have
\begin{equation*}
\begin{split}
&\left|\frac{1}{\left(\begin{matrix} m+k-1\\ k-1\end{matrix}\right)}\sum_{{q_1\geq 0,\ldots, q_k\geq 0}\atop {q_1+\cdots +q_k=m}} \text{\rm trace\,}\left[ \phi_1^{q_1}\circ \cdots \circ \phi_k^{q_k}(X)\right]-L(X)\right|\\
&\qquad \leq
\frac{1}{\left(\begin{matrix} m+k-1\\ k-1\end{matrix}\right)}\sum_{(q_1,\ldots, q_k)\in \cup_{i=1}^k A_i}\text{\rm trace\,}\left[ \phi_1^{q_1}\circ \cdots \circ \phi_k^{q_k}(X)\right]\\
&\qquad+ \frac{1}{\left(\begin{matrix} m+k-1\\ k-1\end{matrix}\right)}
\sum_{(q_1,\ldots, q_k)\in  B_{N_0}}\left|\text{\rm trace\,}\left[ \phi_1^{q_1}\circ \cdots \circ \phi_k^{q_k}(X)\right]-L(X)\right|+\frac{1}{\left(\begin{matrix} m+k-1\\ k-1\end{matrix}\right)}L(X) \text{\rm card\,}\left(\cup_{i=1}^k A_i\right)\\
&\qquad \leq \frac{ kN_0( \text{\rm trace\,}(X) +L(X))\left(\begin{matrix} m+k-2\\ k-2\end{matrix}\right)}
{\left(\begin{matrix} m+k-1\\ k-1\end{matrix}\right)} + \frac{ \text{\rm card\,} (B_{N_0})}
{\left(\begin{matrix} m+k-1\\ k-1\end{matrix}\right)}\epsilon.
\end{split}
\end{equation*}
Since $\lim_{m\to\infty} \frac{ \left(\begin{matrix} m+k-2\\ k-2\end{matrix}\right)}
{\left(\begin{matrix} m+k-1\\ k-1\end{matrix}\right)} =0
$
and $\text{\rm card\,} (B_{N_0})\leq \left(\begin{matrix} m+k-1\\ k-1\end{matrix}\right)$, one can easily complete the proof.
\end{proof}

We remark that  Lemma \ref{basic} implies
$
\lim_{(q_1,\ldots, q_k)\in \ZZ_+^k} \text{\rm trace\,}\left[ \phi_1^{q_1}\circ \cdots \circ \phi_k^{q_k}(X)\right]\leq \text{\rm trace\,}(X)$ for  any $ X\in \cT^+(\cH).
$
Using the classical Stolz-Ces\` aro  convergence theorem  and Lemma \ref{basic}, one can easily prove the following result.
\begin{lemma}\label{basic2} Under the conditions of Lemma \ref{basic},
the limit
\begin{equation*}
\lim_{m\to\infty}\frac{1}{\left(\begin{matrix} m+k\\ k\end{matrix}\right)}\sum_{{q_1\geq 0,\ldots, q_k\geq 0}\atop {q_1+\cdots +q_k\leq m}} \text{\rm trace\,}\left[ \phi_1^{q_1}\circ \cdots \circ \phi_k^{q_k}(X)\right]
\end{equation*}
exists and is equal to
\begin{equation*}
  \lim_{q_1\to\infty}\cdots\lim_{q_k\to\infty} \frac{1}{q_1}\sum_{s_1=0}^{q_1}\cdots \frac{1}{q_k}\sum_{s_k=0}^{q_k}\text{\rm trace\,}\left[ \phi_1^{q_1}\circ \cdots \circ \phi_k^{q_k}(X)\right]
\end{equation*}
for any $X\in \cT^+(\cH)$. Moreover, these limits are equal to those from
Lemma \ref{basic}.
\end{lemma}

We remark that if we take $X=\Delta_\phi(I):=(id-\phi_1)\circ \cdots \circ (id -\phi_k)(I)$ and assume that it is  a positive trace class operator, then we can show that the limits in Lemma \ref{basic2} are equal to
$$
\text{\rm curv}({\bf \phi}):=\lim_{q_1\to\infty}\cdots\lim_{q_k\to\infty}
\frac{\text{\rm trace\,}\left[(id-\phi_{1}^{q_1+1})\circ \cdots \circ (id-\phi_{k}^{q_k+1})(I)\right]}
{q_1\cdots q_k}.
$$
 This can be seen as a curvature invariant associated  with the $k$-tuple of positive maps $\phi=(\phi_1,\ldots, \phi_k)$ satisfying the conditions of Lemma \ref{basic} and such that the defect $\Delta_\phi(I)$ is a positive trace class operator. In this case, we have
$0\leq \text{\rm curv}({\bf \phi})\leq \text{\rm trace\,}[\Delta_\phi(I)]$.

 Let ${\bf T}=({ T}_1,\ldots, { T}_k)$ be in  the regular polyball  $ {\bf B_n}(\cH)$. If    the
  defect of ${\bf T}$,
$$
{\bf \Delta_{T}}(I)  :=(id-\Phi_{T_1})\circ \cdots \circ (id-\Phi_{T_k})(I),
$$
is  a trace class operator,  we say that ${\bf T}$ has  trace class defect. When ${\bf \Delta_{T}}(I)$ has finite rank,  we say that ${\bf T}$ has finite rank and write  $\rank({\bf T}):=\rank [{\bf \Delta_{T}}(I)]$.
  We  introduce the curvature   of any   element  ${\bf T}\in  {\bf B_n}(\cH)$  with trace class defect by setting
\begin{equation}\label{curva}
\text{\rm curv}({\bf T}):=
\lim_{m\to\infty}\frac{1}{\left(\begin{matrix} m+k\\ k\end{matrix}\right)}\sum_{{q_1\geq 0,\ldots, q_k\geq 0}\atop {q_1+\cdots +q_k\leq m}} \frac{\text{\rm trace\,}\left[ {\bf K_{T}^*} (P_{q_1}^{(1)}\otimes \cdots \otimes P_{q_k}^{(k)}\otimes I_\cH){\bf K_{T}}\right]}{\text{\rm trace\,}\left[P_{q_1}^{(1)}\otimes \cdots \otimes P_{q_k}^{(k)}\right]},
\end{equation}
where ${\bf K_T}$ is the Berezin kernel of ${\bf T}$ and  $P_{q_i}^{(i)}$ is the orthogonal projection of the full Fock space $F^2(H_{n_i})$ onto the span of all vectors $e_{\alpha_i}^i$ with  $\alpha\in  \FF_{n_i}^+$ and $|\alpha_i|=q_i$.
 In what follows,  we show that $\text{\rm curv}({\bf T})$  exists and established several  asymptotic formulas for the curvature invariant in terms of the Berezin kernel.

For each $q_i\in \{0,1,\ldots\}$ and  $i\in \{1,\ldots, k\}$, let $P_{\leq(q_1,\ldots, q_k)}$ be the orthogonal projection of $\otimes_{i=1}^k F^2(H_{n_i})$ onto the the span of all vectors of the form $e_{\alpha_1}^1\otimes\cdots \otimes e_{\alpha_k}^k$, where $\alpha_i\in \FF_{n_i}^+$, $|\alpha_i|\leq q_i$.

\begin{theorem} \label{curva1} If  \, ${\bf T}=({ T}_1,\ldots, { T}_k)$ is  in the regular polyball  $ {\bf B_n}(\cH)$ and has  trace class defect, then the curvature of ${\bf T}$ exists and satisfies the  asymptotic formulas
\begin{equation*}\begin{split}
\text{\rm curv}({\bf T})&=\lim_{(q_1,\ldots, q_k)\in \ZZ_+^k}  \frac{\text{\rm trace\,}\left[ {\bf K_{T}^*} (P_{q_1}^{(1)}\otimes \cdots \otimes P_{q_k}^{(k)}\otimes I_\cH){\bf K_{T}}\right]}{\text{\rm trace\,}\left[P_{q_1}^{(1)}\otimes \cdots \otimes P_{q_k}^{(k)}\right]}\\
&=\lim_{m\to\infty}\frac{1}{\left(\begin{matrix} m+k-1\\ k-1\end{matrix}\right)}\sum_{{q_1\geq 0,\ldots, q_k\geq 0}\atop {q_1+\cdots +q_k= m}} \frac{\text{\rm trace\,}\left[ {\bf K_{T}^*} (P_{q_1}^{(1)}\otimes \cdots \otimes P_{q_k}^{(k)}\otimes I_\cH){\bf K_{T}}\right]}{\text{\rm trace\,}\left[P_{q_1}^{(1)}\otimes \cdots \otimes P_{q_k}^{(k)}\right]}\\
&=\lim_{q_1\to\infty}\cdots\lim_{q_k\to\infty}
\frac{\text{\rm trace\,}\left[ {\bf K_{T}^*}(P_{\leq(q_1,\ldots, q_k)}  \otimes I_\cH){\bf K_{T}}\right]}{\text{\rm trace\,}\left[P_{\leq(q_1,\ldots, q_k)}\right]}.
\end{split}
\end{equation*}
\end{theorem}
\begin{proof}  Let   ${\bf S}:=({\bf S}_1,\ldots, {\bf S}_k)\in {\bf B_n}(\otimes_{i=1}^kF^2(H_{n_i}))$ with  ${\bf S}_i:=({\bf S}_{i,1},\ldots,{\bf S}_{i,n_i})$  be  the  universal model associated
  with the abstract
  polyball ${\bf B_n}$.  It is easy to see that     ${\bf S}:=({\bf S}_1,\ldots, {\bf S}_k)$ is  a pure $k$-tuple   and
   $(id-\Phi_{{\bf S}_1})\circ \cdots \circ (id-\Phi_{{\bf S}_k})(I)={\bf P}_\CC,
   $
   where ${\bf P}_\CC$ is the
 orthogonal projection from $\otimes_{i=1}^k F^2(H_{n_i})$ onto $\CC 1\subset \otimes_{i=1}^k F^2(H_{n_i})$, where $\CC 1$ is identified with $\CC 1\otimes\cdots \otimes \CC 1$.
First, we  prove that
\begin{equation} \label{connection}
 {\bf K_{T}^*} (P_{q_1}^{(1)}\otimes \cdots \otimes P_{q_k}^{(k)}\otimes I_\cH){\bf K_{T}}=
  \Phi_{T_1}^{q_1}\circ \cdots \circ \Phi_{T_k}^{q_k}({\bf \Delta_{T}}(I))
  \end{equation}
  for any $q_1,\ldots,q_k\in \ZZ^+$.
  Taking into account that
   ${\bf K_{T}} { T}^*_{i,j}= ({\bf S}_{i,j}^*\otimes I)  {\bf K_{T}}
    $
for any $i\in \{1,\ldots, k\}$ and $j\in \{1,\ldots, n_i\}$, we deduce that
\begin{equation}\label{equ} \begin{split}
 {\bf K_{T}^*}&\left\{[ \Phi_{{\bf S}_1}^{q_1}\circ \cdots \circ \Phi_{{\bf S}_k}^{q_k}\circ(id-\Phi_{{\bf S}_k})\circ \cdots \circ (id-\Phi_{{\bf S}_1})(I)]\otimes
 I_\cH\right\}{\bf K_{T}}\\
 &\qquad \qquad=\Phi_{T_1}^{q_1}\circ \cdots \circ \Phi_{T_k}^{q_k}\circ(id-\Phi_{T_1})\circ \cdots \circ (id-\Phi_{T_k})({\bf K_{T}^*}{\bf K_{T}}).
 \end{split}
 \end{equation}
Since
$$
{\bf K_{T}^*}{\bf K_{T}}=
\lim_{q_k\to\infty}\ldots \lim_{q_1\to\infty}  (id-\Phi_{T_k}^{q_k})\circ \cdots \circ (id-\Phi_{T_1}^{q_1})(I),
$$
where the limits are in the weak  operator topology, we can prove that
\begin{equation} \label{id-Q}
 {\bf \Delta_{T}}({\bf K_{T}^*}{\bf K_{T}})={\bf \Delta_{T}}(I).
\end{equation}
Indeed,   note that   that  $\{(id-\Phi_{T_k}^{q_k})\circ \cdots \circ (id-\Phi_{T_1}^{q_1})(I)\}_{{\bf q}=(q_1,\ldots, q_k)\in \ZZ_+^k}$ is an  increasing sequence of positive operators and
 $$
 (id-\Phi_{T_k}^{q_k})\circ \cdots \circ (id-\Phi_{T_1}^{q_1})(I)
 =\sum_{s_k=0}^{q_k-1}\Phi_{T_k}^{s_k}\circ\cdots \sum_{s_1=0}^{q_1-1}\Phi_{T_1}^{s_1}\circ
(id-\Phi_{T_k})\circ \cdots \circ (id-\Phi_{T_1})(I).
$$
 Since $\Phi_{T_1},\ldots, \Phi_{T_k}$ are commuting WOT-continuous completely positive linear maps  and  $\lim_{q_i\to \infty}\Phi_{T_i}^{q_i}(I)$ exists  in the weak operator topology for each $i\in\{1,\ldots,k\}$, we have
\begin{equation*}
\begin{split}
(id-\Phi_{T_1})({\bf K_{T}^*}{\bf K_{T}})
&=\lim_{q_k\to\infty}\ldots \lim_{q_1\to\infty}  (id-\Phi_{T_k}^{q_k})\circ \cdots \circ (id-\Phi_{T_1}^{q_1})\circ(id-\Phi_{T_1})(I)\\
&=\lim_{q_k\to\infty}\ldots \lim_{q_2\to\infty}  (id-\Phi_{T_k}^{q_k})\circ \cdots \circ (id-\Phi_{T_2}^{q_2})
\left[\lim_{q_1\to\infty}(id-\Phi_{T_1}^{q_1})\circ(id-\Phi_{T_1})(I)\right]\\
&=\lim_{q_k\to\infty}\ldots \lim_{q_2\to\infty}  (id-\Phi_{T_k}^{q_k})\circ \cdots \circ (id-\Phi_{T_2}^{q_2})\circ(id-\Phi_{T_1})(I).
\end{split}
\end{equation*}
Applying now $id-\Phi_{T_2}$, a similar reasoning leads to
$$
(id-\Phi_{T_2})\circ(id-\Phi_{T_1})({\bf K_{T}^*}{\bf K_{T}})=\lim_{q_k\to\infty}\ldots \lim_{q_3\to\infty}  (id-\Phi_{T_k}^{q_k})\circ \cdots \circ (id-\Phi_{T_3}^{q_3})\circ(id-\Phi_{T_1})\circ(id-\Phi_{T_2})(I).
$$
Continuing this process, we obtain relation \eqref{id-Q}. Now, using  relations  \eqref{equ} and \eqref{id-Q},  we have
\begin{equation*}
\begin{split}
{\bf K_{T}^*} (P_{q_1}^{(1)}\otimes \cdots \otimes P_{q_k}^{(k)}\otimes I_\cH){\bf K_{T}}
&={\bf K_{T}^*}\left[ (\Phi_{{\bf S}_1}^{q_1}-  \Phi_{{\bf S}_1}^{q_1+1})\circ \cdots \circ (\Phi_{{\bf S}_k}^{q_k}-  \Phi_{{\bf S}_k}^{q_k+1})(I)\otimes I_\cH \right]{\bf K_{T}}\\
&={\bf K_{T}^*}\left[ \Phi_{{\bf S}_1}^{q_1} \circ \cdots \circ \Phi_{{\bf S}_k}^{q_k}\circ(id-\Phi_{{\bf S}_k})\circ \cdots \circ (id-\Phi_{{\bf S}_1})(I) \otimes I_\cH \right]{\bf K_{T}}\\
&=\Phi_{T_1}^{q_1}\circ \cdots \circ \Phi_{T_k}^{q_k}\circ(id-\Phi_{T_1})\circ \cdots \circ (id-\Phi_{T_k})({\bf K_{T}^*}{\bf K_{T}})\\
&=\Phi_{T_1}^{q_1}\circ \cdots \circ \Phi_{T_k}^{q_k} ({\bf \Delta_{T}}(I)).
\end{split}
\end{equation*}
This proves relation \eqref{connection}.
Now, note that for each $i\in\{1,\ldots, k\}$ and $X\in \cT^+(\cH)$,
$\Phi_{T_i}(X)$ is a positive trace class operator and
\begin{equation*}
\begin{split}
\text{\rm trace\,}[\Phi_{T_i}(X)]&=\sum_{j=1}^{n_i}\text{\rm trace\,}[T_{i,j} X T_{i,j}^*]
=\text{\rm trace\,}\left[ \sum_{j=1}^{n_i}T_{i,j}^*T_{i,j} X\right]\\
&\leq \left\| \sum_{j=1}^{n_i}T_{i,j}^*T_{i,j}\right\|\text{\rm trace\,}(X)
\leq n_i \text{\rm trace\,}(X).
\end{split}
\end{equation*}
Applying Lemma \ref{basic}  when $\phi_i:=\frac{1}{n_i} \Phi_{T_i}$, $i\in \{1,\ldots, k\}$, and $X={\bf \Delta_T}(I)$, we deduce that the limit
 $$
L:=\lim_{q_1\to\infty}\cdots\lim_{q_k\to\infty}  \frac{1}{n_1^{q_1}\cdots n_k^{q_k}}\text{\rm trace\,}\left[ \Phi_{T_1}^{q_1}\circ \cdots \circ \Phi_{T_k}^{q_k}({\bf \Delta_{T}}(I))\right]
$$
exists, which proves the existence of the curvature invariant defined by relation \eqref{curva}.
  Setting $$x_{q_1}:=\lim_{q_2\to\infty}\cdots\lim_{q_k\to\infty}  \frac{1}{n_1^{q_1}\cdots n_k^{q_k}}\text{\rm trace\,}\left[ \Phi_{T_1}^{q_1}\circ \cdots \circ \Phi_{T_k}^{q_k}({\bf \Delta_{T}}(I))\right],
 $$
 an application of  Stolz-Ces\` aro  convergence theorem to the sequence $\{x_{q_1}\}_{q_1=0}^\infty$ implies
$$
\lim_{q_1\to \infty}\frac{1}{1+n_1+\cdots + n_1^{q_1}}\sum_{s_1=0}^{q_1}\lim_{q_2\to\infty}\cdots\lim_{q_k\to\infty} \frac{1}{n_2^{q_2}\cdots n_k^{q_k}}\text{\rm trace\,}\left[ \Phi_{T_1}^{s_1}\circ\Phi_{T_2}^{q_2}\circ \cdots \circ \Phi_{T_k}^{q_k}({\bf \Delta_{T}}(I))\right]=\lim_{q_1\to\infty}x_{q_1}=L.
$$
Similarly, setting $y_{q_2}:=\lim_{q_3\to\infty}\cdots\lim_{q_k\to\infty} \frac{1}{n_2^{q_2}\cdots n_k^{q_k}}\text{\rm trace\,}\left[ \Phi_{T_1}^{s_1}\circ \Phi_{T_2}^{q_2}\circ \cdots \circ \Phi_{T_k}^{q_k}({\bf \Delta_{T}}(I))\right]$, we deduce that
$$
\lim_{q_2\to \infty}\frac{1}{1+n_2+\cdots + n_2^{q_2}}\sum_{s_2=0}^{q_2}\lim_{q_3\to\infty}\cdots\lim_{q_k\to\infty} \frac{1}{n_3^{q_3}\cdots n_k^{q_k}}\text{\rm trace\,}\left[ \Phi_{T_1}^{s_1}\circ \Phi_{T_2}^{s_2} \circ \Phi_{T_3}^{q_3}\cdots \circ \Phi_{T_k}^{q_k}({\bf \Delta_{T}}(I))\right]=\lim_{q_2\to\infty}y_{q_2}.
$$
Continuing this process and putting together all these relations, we obtain
\begin{equation*}\label{formulas2}
\begin{split}
L&=\lim_{q_1\to\infty}\cdots\lim_{q_k\to\infty}
\frac{\text{\rm trace\,}\left[ \sum_{s_1=0}^{q_1}\cdots \sum_{s_k=0}^{q_k} \Phi_{T_1}^{s_1}\circ \cdots \circ \Phi_{T_k}^{s_k}({\bf \Delta_{T}}(I))\right]}
{\prod_{i=1}^k(1+n_i+\cdots + n_i^{q_i})}\\
&=\lim_{q_1\to\infty}\cdots\lim_{q_k\to\infty}
\frac{\text{\rm trace\,}\left[(id-\Phi_{T_1}^{q_1+1})\circ \cdots \circ (id-\Phi_{T_k}^{q_k+1})(I)\right]}
{\prod_{i=1}^k(1+n_i+\cdots + n_i^{q_i})}.
\end{split}
\end{equation*}
Consequently, using relation \eqref{connection}, we obtain
\begin{equation*}
\begin{split}
\text{\rm curv}({\bf T})&=\lim_{q_1\to\infty}\cdots\lim_{q_k\to\infty}
\frac{\text{\rm trace\,}\left[ \sum_{s_1=0}^{q_1}\cdots \sum_{s_k=0}^{q_k} \Phi_{T_1}^{s_1}\circ \cdots \circ \Phi_{T_k}^{s_k}({\bf \Delta_{T}}(I))\right]}
{\prod_{i=1}^k(1+n_i+\cdots + n_i^{q_i})}\\
&=
\lim_{q_1\to\infty}\cdots\lim_{q_k\to\infty}
\frac{\text{\rm trace\,}\left[ {\bf K_{T}^*}(P_{\leq(q_1,\ldots, q_k)}  \otimes I_\cH){\bf K_{T}}\right]}{\text{\rm trace\,}\left[P_{\leq(q_1,\ldots, q_k)}\right]},
\end{split}
\end{equation*}
where the order of the iterated limits is irrelevant.
 The proof is  complete.
\end{proof}

If $m\in \ZZ_+$, we denote by $P_{\leq m}$ the orthogonal projection  of
$\otimes_{i=1}^k F^2(H_{n_i})$ onto the the span of all vectors   $e_{\alpha_1}\otimes\cdots \otimes e_{\alpha_k}$, where $\alpha_i\in \FF_{n_i}^+$ and $|\alpha_1|+\cdots+ |\alpha_k|\leq m$.
In the particular case of the regular polydisk, i.e.  $n_1=\cdots =n_k=1$,  we have $\text{\rm trace\,}\left[P_{q_1}^{(1)}\otimes \cdots \otimes P_{q_k}^{(k)}\right]=1$ and, therefore,
$$
\text{\rm curv}({\bf T})= \lim_{m\to\infty}\frac{
\text{\rm trace\,}\left[ {\bf K_{T}^*} (P_{\leq m}\otimes I_\cH){\bf K_{T}}\right]}{\text{\rm trace\,} P_{\leq m}}.
$$
 Theorem \ref{curva1}  implies  several other asymptotic formulas for the curvature in the regular polydisc.
 Using Theorem \ref{curva1}  and Lemma \ref{basic2} when $\phi_i:=\frac{1}{n_i} \Phi_{T_i}$, $i\in \{1,\ldots, k\}$, and $X={\bf \Delta_T}(I)$, we obtain the following

\begin{corollary}\label{curva-maps}  Let  ${\bf T}=({ T}_1,\ldots, { T}_k)$ be an   element   in the regular polyball  $ {\bf B_n}(\cH)$, ${\bf n}=(n_1,\ldots, n_k)\in \NN^k$, with  trace class defect and let $\Phi_{T_1}, \ldots, \Phi_{T_k}$ be the associated completely positive maps. Then the curvature of ${\bf T}$ satisfies the  asymptotic formulas
\begin{equation*}\begin{split}
\text{\rm curv}({\bf T})&=\lim_{m\to\infty}\frac{1}{\left(\begin{matrix} m+k\\ k\end{matrix}\right)}\sum_{{q_1\geq 0,\ldots, q_k\geq 0}\atop {q_1+\cdots +q_k\leq m}}\frac {1}{n_1^{q_1}\cdots n_k^{q_k}}{\text{\rm trace\,}\left[ \Phi_{T_1}^{q_1}\circ \cdots \circ \Phi_{T_k}^{q_k}({\bf \Delta_{T}}(I))\right]}\\
&=\lim_{m\to\infty}\frac{1}{\left(\begin{matrix} m+k-1\\ k-1\end{matrix}\right)}\sum_{{q_1\geq 0,\ldots, q_k\geq 0}\atop {q_1+\cdots +q_k=m}} \frac{1}{n_1^{q_1}\cdots n_k^{q_k}}\text{\rm trace\,}\left[ \Phi_{T_1}^{q_1}\circ \cdots \circ \Phi_{T_k}^{q_k}({\bf \Delta_{T}}(I))\right]\\
&= \lim_{q_1\to\infty}\cdots\lim_{q_k\to\infty}
\frac{\text{\rm trace\,}\left[(id-\Phi_{T_1}^{q_1+1})\circ \cdots \circ (id-\Phi_{T_k}^{q_k+1})(I)\right]}
{\prod_{i=1}^k(1+n_i+\cdots + n_i^{q_i})}\\
&=\lim_{(q_1,\ldots, q_k)\in \ZZ_+^k} \frac{1}{n_1^{q_1}\cdots n_k^{q_k}}\text{\rm trace\,}\left[ \Phi_{T_1}^{q_1}\circ \cdots \circ \Phi_{T_k}^{q_k}({\bf \Delta_{T}}(I))\right].
\end{split}
\end{equation*}
\end{corollary}

\begin{remark}  A closer look at the proof of Thoorem \ref{curva1} reveals that, if ${\bf T}=({ T}_1,\ldots, { T}_k)$ is an element in $B(\cH)^{n_1}\times_c\cdots \times_c B(\cH)^{n_k}$    such that ${\bf \Delta_{T}}(I)$ is a  trace class positive operator, the limits in Corollary \ref{curva-maps} exist and can be used as the definition of the curvature  of   $T$, extending  the one on the regular polyball.
\end{remark}

\begin{corollary} If ${\bf T}\in  {\bf B_n}(\cH)$ has   trace class defect, then
$$
0\leq \text{\rm curv}({\bf T})\leq \text{\rm trace\,}\left[{\bf \Delta_{T}}(I)\right]\leq \rank [{\bf \Delta_{T}}(I)].
$$
\end{corollary}

\begin{corollary} If ${\bf T}\in  {\bf B_n}(\cH)$  and ${\bf T}'\in {\bf B_n}(\cH')$  have trace class defects,   then ${\bf T}\oplus {\bf T}'\in {\bf B_n}(\cH\oplus\cH')$ has trace class defect and
$$
\text{\rm curv}({\bf T}\oplus {\bf T}')=\text{\rm curv}({\bf T})+ \text{\rm curv}({\bf T}').
$$
If, in addition, $\dim\cH'<\infty$, then $\text{\rm curv}({\bf T}\oplus {\bf T}')=\text{\rm curv}({\bf T})$.
\end{corollary}
More properties for the curvature will be presented in the coming sections.

\section{The curvature operator and classification}

In this section, we show that if ${\bf T}\in {\bf B_n}(\cH)$ has characteristic function and finite rank, then its curvature is the trace of the curvature operator, to be introduced, and satisfies an index type formula which is used to obtain certain classification  results.
In particular, we show that the curvature invariant classifies  the  finite rank  Beurling type invariant subspaces  of the  tensor product of full
Fock spaces $F^2(H_{n_1})\otimes \cdots \otimes F^2(H_{n_k})$.

\begin{lemma}\label{reproducing} Let $\phi_1,\ldots, \phi_k$ be positive linear maps on $B(\cH)$ such that each $\phi_i$ is pure., i.e. $\phi_i^n(I) \to 0$ weakly  as $n\to \infty$. Then
${\bf\Delta_\phi}:=(id-\phi_1)\circ \cdots \circ (id -\phi_k)$ is a one-to-one map and each $X\in B(\cH)$ has the representation
$$
X=\sum_{s_k=0}^\infty \phi_{k}^{s_k}\left(\sum_{s_{k-1}=0}^\infty \phi_{k-1}^{s_{k-1}}\left(\cdots \sum_{s_1=0}^\infty \phi_{1}^{s_1} \left({\bf \Delta_{\phi}}(X)\right)\cdots \right)\right),
$$
where the iterated series converge   in the weak operator topology. If, in addition, ${\bf \Delta_{\phi}}(X)\geq 0$, then
$$
X=\sum_{(s_1,\ldots, s_k)\in \ZZ_+^k} \phi_k^{s_k}\cdots \phi_1^{s_1}({\bf \Delta_{\phi}}(X)).
$$
\end{lemma}
\begin{proof} We use the notation ${\bf \Delta}_\phi^{(p_1,\ldots, p_k)}:=(id-\phi_1)^{p_1}\circ \cdots \circ (id -\phi_k)^{p_k}$ for $p_i\in \{0,1\}$. Note that
\begin{equation}
\label{rec}
\sum_{s_1=0}^{q_1} \phi_1^{s_1}({\bf \Delta_{\phi}}(X))
= {\bf \Delta}_\phi^{(0,1,\ldots, 1)}(X)-\phi_1^{q_1+1}({\bf \Delta}_\phi^{(0,1,\ldots, 1)}(X)).
\end{equation}
If $Z\in B(\cH)$ is a positive operator and $x,y\in \cH$, the Cauchy-Schwarz inequality implies
$$
\left|\left<\phi_i^{q_i}(Z)x, y\right>\right|\leq \|Z\|\left<\phi_i^{q_i}(I)x, x\right>^{1/2} \left<\phi_i^{q_i}(I)y, y\right>^{1/2}.
$$
Since  $\phi_i^{q_i}(I) \to 0$ weakly as $q_i\to \infty$, we deduce that $\left<\phi_i^{q_i}(Z)x, y\right>\to 0$ as $q_i\to\infty$. Taking into account that any bounded linear operator  is a linear combination of positive operators, we conclude that the convergence above holds for any $Z\in B(\cH)$. Passing to the limit in  relation \eqref{rec} as $q_i\to\infty$, we obtain
$
\sum_{s_1=0}^{\infty} \phi_1^{s_1}({\bf \Delta_{\phi}}(X))
= {\bf \Delta}_\phi^{(0,1,\ldots, 1)}(X).
$
Similarly, we obtain
$
\sum_{s_2=0}^{\infty} \phi_2^{s_2}({\bf \Delta}_{\phi}^{(0,1,\ldots, 1)}(X))
= {\bf \Delta}_\phi^{(0,0,1\ldots, 1)}(X)  $
and, continuing this process,
$$\sum_{s_k=0}^{\infty} \phi_k^{s_k}({\bf \Delta}_{\phi}^{(0,\ldots,0, 1)}(X))
= {\bf \Delta}_\phi^{(0,\ldots, 0)}(X)=X.
$$
Putting together these relations, we deduce the first equality of the lemma.
Now, one can see that ${\bf \Delta}_\phi$ is one-to-one.
If we assume that ${\bf \Delta_{\phi}}(X)\geq 0$, then  the multi-sequence
$$\sum_{s_k=0}^{q_k} \phi_{k}^{s_k}\left(\sum_{s_{k-1}=0}^{q_{k-1}} \phi_{k-1}^{s_{k-1}}\left(\cdots \sum_{s_1=0}^{q_1} \phi_{1}^{s_1} \left({\bf \Delta_{\phi}}(X)\right)\cdots \right)\right)
$$ is increasing with respect to each indexes $q_1,\ldots, q_k$.
The last part of the lemma follows.
\end{proof}
 A simple consequence of Lemma \ref{reproducing} is the following result which will be used later.
\begin{corollary}\label{Taylor-rep}
 If ${\bf T}=({ T}_1,\ldots, { T}_k)$  is  a pure element  in the regular polyball  ${\bf B_n}(\cH)$, then the  defect map ${\bf \Delta_{T}}$ is one-to-one and any $X\in B(\cH)$ has the following Taylor type representation around its defect
$$
X=\sum_{s_1=0}^\infty \Phi_{T_1}^{s_1}\left(\sum_{s_2=0}^\infty \Phi_{T_2}^{s_2}\left(\cdots \sum_{s_k=0}^\infty \Phi_{T_k}^{s_k} \left({\bf \Delta_{T}}(X)\right)\cdots \right)\right),
$$
where the iterated series converge in the weak operator topology. If, in addition, ${\bf \Delta_{T}}(X)\geq 0$, then
$$
X=\sum_{(s_1,\ldots, s_k)\in \ZZ_+^k} \Phi_{T_1}^{s_1}\circ \cdots \circ \Phi_{T_k}^{s_k}({\bf \Delta_{T}}(X)).
$$
\end{corollary}

For each ${\bf q}=(q_1,\ldots, q_k)\in \ZZ_+^k$, define  the operator ${\bf N}_{\leq {\bf q}}$  acting on    $F^2(H_{n_1})\otimes\cdots\otimes F^2(H_{n_k})$ by setting
 \begin{equation} \label{NQ}{\bf N}_{\leq {\bf q}}:=
\sum_{{0\leq s_i\leq q_i}\atop{i\in\{1,\ldots, k\}}} \frac{1}{\text{\rm trace\,}\left[P_{s_1}^{(1)}\otimes \cdots \otimes P_{s_k}^{(k)}\right]}\,
P_{s_1}^{(1)}\otimes \cdots \otimes P_{s_k}^{(k)}.
\end{equation}

\begin{lemma}\label{identity}
 Let $Y$ be a bounded operator acting on $\otimes_{i=1}^k F^2(H_{n_i})\otimes \cH$  with $\dim(\cH)<\infty$. Then
 $$
 \frac{\text{\rm trace\,}\left[(P_{q_1}^{(1)}\otimes \cdots \otimes P_{q_k}^{(k)}\otimes I_\cH)Y\right]}{\text{\rm trace\,}\left[P_{q_1}^{(1)}\otimes \cdots \otimes P_{q_k}^{(k)}\right]}
 =\text{\rm trace\,}\left[{\bf \Delta_{S\otimes{\it I}_\cH}}(Y)({\bf N}_{\leq {\bf q}}\otimes I_\cH)\right],
$$
where ${\bf S}$ is the  universal model associated
  with the abstract
  polyball ${\bf B_n}$.
\end{lemma}
\begin{proof}
Let  ${\bf S}:=({\bf S}_1,\ldots, {\bf S}_k)$ with  ${\bf S}_i:=({\bf S}_{i,1},\ldots,{\bf S}_{i,n_i})$.
    It is easy to see that     ${\bf S}:=({\bf S}_1,\ldots, {\bf S}_k)$ is  a pure $k$-tuple  in the
noncommutative polyball $ {\bf B_n}(\otimes_{i=1}^kF^2(H_{n_i}))$.
 Applying Corollary \ref{Taylor-rep} to  ${\bf S}\otimes I_\cH$, we have

$$
Y=\sum_{s_1=0}^\infty \Phi_{{\bf S}_1\otimes I_\cH}^{s_1}\left(\sum_{s_2=0}^\infty \Phi_{{\bf S}_2\otimes I_\cH}^{s_2}\left(\cdots \sum_{s_k=0}^\infty \Phi_{{\bf S}_k\otimes I_\cH}^{s_k} \left({\bf \Delta}_{{\bf S}\otimes I_\cH}(Y)\right)\cdots \right)\right),
$$
where the iterated series converge  in the weak operator topology.
 Denoting $P_{\bf q}:=P_{q_1}^{(1)}\otimes \cdots \otimes P_{q_k}^{(k)}$ for ${\bf q}=(q_1,\ldots, q_k)\in \ZZ_+^k$, the relation above implies

\begin{equation*}
\begin{split}
\left(P_{\bf q}\otimes I_\cH\right) Y
 =
\left(P_{\bf q}\otimes I_\cH\right)
\left(\sum_{s_1=0}^{q_1} \Phi_{{\bf S}_1\otimes I_\cH}^{s_1}\left(\sum_{s_2=0}^{q_2} \Phi_{{\bf S}_2\otimes I_\cH}^{s_2}\left(\cdots \sum_{s_k=0}^{q_k} \Phi_{{\bf S}_k\otimes I_\cH}^{s_k} \left({\bf \Delta}_{{\bf S}\otimes I_\cH}(Y)\right)\cdots \right)\right)\right).
\end{split}
\end{equation*}
Consequently, we have
\begin{equation*}
\begin{split}
&\text{\rm trace\,}\left[\left(P_{\bf q}\otimes I_\cH\right) Y\right] \\
&=
\sum_{s_1=0}^{q_1}\cdots \sum_{s_k=0}^{q_k} \text{\rm trace\,}
\left[\left(P_{\bf q}\otimes I_\cH\right)\Phi_{{\bf S}_1\otimes I_\cH}^{s_1}\circ \cdots \circ \Phi_{{\bf S}_k\otimes I_\cH}^{s_k}({\bf \Delta}_{{\bf S}\otimes I_\cH}(Y))\right]\\
&=
\sum_{s_1=0}^{q_1}\cdots \sum_{s_k=0}^{q_k} \text{\rm trace\,}
\left[\left(P_{\bf q}\otimes I_\cH\right)\sum_{{\alpha_1\in \FF_{n_1}^+, \cdots
\alpha_k\in \FF_{n_k}^+} \atop{|\alpha_1|=s_1\cdots,  |\alpha_k|=s_k}}
({\bf S}_{1,\alpha_1}\cdots {\bf S}_{k,\alpha_k}\otimes I_\cH){\bf \Delta}_{{\bf S}\otimes I_\cH}(Y)
({\bf S}_{k,\alpha_k}^*\cdots {\bf S}_{1,\alpha_1}^*\otimes I_\cH)
\right]\\
&=
\sum_{s_1=0}^{q_1}\cdots \sum_{s_k=0}^{q_k} \text{\rm trace\,}
\left[{\bf \Delta}_{{\bf S}\otimes I_\cH}(Y)\sum_{{\alpha_1\in \FF_{n_1}^+, \cdots
\alpha_k\in \FF_{n_k}^+} \atop{|\alpha_1|=s_1\cdots,  |\alpha_k|=s_k}}
({\bf S}_{k,\alpha_k}^*\cdots {\bf S}_{1,\alpha_1}^* P_{\bf q} {\bf S}_{1,\alpha_1}\cdots {\bf S}_{k,\alpha_k})\otimes I_\cH
\right].
\end{split}
\end{equation*}
Note that, for each $i\in \{1,\ldots, k\}$ and $\alpha_i\in \FF_{n_i}^+$ with $|\alpha_i|=s_i\leq q_i$, we have
$$
{\bf S}_{i,\alpha_i}^*(I\otimes \cdots \otimes P_{q_i}^{(i)}\otimes\cdots \otimes I)=(I\otimes \cdots \otimes P_{q_i-s_i}^{(i)}\otimes\cdots\otimes I){\bf S}_{i,\alpha_i}^*.
$$
Consequently,
\begin{equation*}
\begin{split}
\sum_{{\alpha_1\in \FF_{n_1}^+, \cdots
\alpha_k\in \FF_{n_k}^+} \atop{|\alpha_1|=s_1\cdots,  |\alpha_k|=s_k}}
&{\bf S}_{k,\alpha_k}^*\cdots {\bf S}_{1,\alpha_1}^* P_{\bf q} {\bf S}_{1,\alpha_1}\cdots {\bf S}_{k,\alpha_k}\\
&=
P_{q_1-s_1}^{(1)}\otimes \cdots \otimes P_{q_k-s_k}^{(k)}\sum_{{\alpha_1\in \FF_{n_1}^+, \cdots
\alpha_k\in \FF_{n_k}^+} \atop{|\alpha_1|=s_1\cdots,  |\alpha_k|=s_k}}
{\bf S}_{k,\alpha_k}^*\cdots {\bf S}_{1,\alpha_1}^*  {\bf S}_{1,\alpha_1}\cdots {\bf S}_{k,\alpha_k}\\
&=
n_1^{s_1}\cdots n_k^{s_k} P_{q_1-s_1}^{(1)}\otimes \cdots \otimes P_{q_k-s_k}^{(k)}.
\end{split}
\end{equation*}
Using the relations above, we obtain
\begin{equation*}
\begin{split}
\text{\rm trace\,}\left[\left(P_{\bf q}\otimes I_\cH\right) Y\right]&\\
&=
\sum_{s_1=0}^{q_1}\cdots \sum_{s_k=0}^{q_k} \text{\rm trace\,}
\left[{\bf \Delta}_{{\bf S}\otimes I_\cH}(Y)n_1^{s_1}\cdots n_k^{s_k} \left(P_{q_1-s_1}^{(1)}\otimes \cdots \otimes P_{q_k-s_k}^{(k)}\otimes I_\cH\right)\right]\\
&=
\text{\rm trace\,}
\left[{\bf \Delta}_{{\bf S}\otimes I_\cH}(Y)\sum_{s_1=0}^{q_1}\cdots \sum_{s_k=0}^{q_k}
n_1^{s_1}\cdots n_k^{s_k} \left(P_{q_1-s_1}^{(1)}\otimes \cdots \otimes P_{q_k-s_k}^{(k)}\otimes I_\cH
\right)\right]\\
&=
\text{\rm trace\,}
\left\{{\bf \Delta}_{{\bf S}\otimes I_\cH}(Y)\left[\left(\sum_{s_1=0}^{q_1} n_1^{s_1}P_{q_1-s_1}^{(1)}\right)\otimes \cdots\otimes \left(\sum_{s_k=0}^{q_k} n_k^{s_k}P_{q_k-s_k}^{(k)}\right)\otimes I_\cH\right]\right\}.
\end{split}
\end{equation*}
Hence, we deduce that
\begin{equation*}
\begin{split}
\frac{\text{\rm trace\,}\left[\left(P_{\bf q}\otimes I_\cH\right) Y\right]}
{\text{\rm trace\,}\left[P_{\bf q}\right]}&=
\text{\rm trace\,}
\left\{{\bf \Delta}_{{\bf S}\otimes I_\cH}(Y)\left[\left(\sum_{s_1=0}^{q_1} \frac{1}{n_1^{s_1}}P_{s_1}^{(1)}\right)\otimes \cdots\otimes \left(\sum_{s_k=0}^{q_k} \frac{1}{n_k^{s_k}}P_{s_k}^{(k)}\right)\otimes I_\cH\right]\right\}\\
&=\text{\rm trace\,}\left[{\bf \Delta}_{{\bf S}\otimes I_\cH}(Y)
\sum_{{0\leq s_i\leq q_i}\atop{i\in\{1,\ldots, k\}}} \frac{1}{\text{\rm trace\,}\left[P_{\bf s}\right]}\,
P_{\bf s}\otimes I_\cH\right]=
\text{\rm trace\,}\left[{\bf \Delta}_{{\bf S}\otimes I_\cH}(Y)
({\bf N}_{\leq \bf q}\otimes I_\cH)\right].
\end{split}
\end{equation*}
This completes the proof.
\end{proof}

In spite of the fact that  the  operator ${\bf \Delta_{S\otimes {\it I}_\cH}}({\bf K_{T}}{\bf K_{T}^*})$ is not  positive  in general, using Theorem \ref{curva1} and Lemma \ref{identity}, in the particular case when $Y={\bf K_T} {\bf K_T^*}$,  we obtain the following result.

\begin{corollary} If  \, ${\bf T}$ is  in the regular polyball  $ {\bf B_n}(\cH)$ and has finite rank, then
\begin{equation*}
\text{\rm curv}({\bf T})=\lim_{{\bf q}\in \ZZ_+^k} \text{\rm trace\,}\left[{\bf \Delta_{S\otimes {\it I}_\cH}}({\bf K_{T}}{\bf K_{T}^*})({\bf N}_{\leq \bf q}\otimes I_\cH)\right],
\end{equation*}
where ${\bf K_T}$ is the Berezin kernel of ${\bf T}$.
\end{corollary}

Define the bounded linear operator ${\bf N}$  acting on $F^2(H_{n_1})\otimes\cdots\otimes F^2(H_{n_k})$ by setting
$${\bf N}:=\sum_{{\bf q}=(q_1,\ldots, q_k)\in \ZZ_+^k} \frac{1}{\text{\rm trace\,}\left[P_{q_1}^{(1)}\otimes \cdots \otimes P_{q_k}^{(k)}\right]}\,
P_{q_1}^{(1)}\otimes \cdots \otimes P_{q_k}^{(k)}
$$
and note that ${\bf N}$ is not a trace class operator.
\begin{theorem}\label{trace-N} Let $Y$ be a bounded operator acting on $\otimes_{i=1}^k F^2(H_{n_i})\otimes \cH$ such that $\dim \cH<\infty$ and
$${\bf \Delta_{S\otimes {\it I}}}(Y)  :=(id-\Phi_{{\bf S}_1\otimes I})\circ\cdots\circ (id-\Phi_{{\bf S}_k\otimes I})(Y)\geq 0.
$$
Then  ${\bf \Delta_{S\otimes {\it I}}}(Y)({\bf N}\otimes I_\cH)$ is a trace class operator and
$$
\text{\rm trace\,}\left[{\bf \Delta_{S\otimes {\it I}_\cH}}(Y)({\bf N}\otimes {\it I}_\cH)\right]
=
 \lim_{q_1\to\infty}\cdots\lim_{q_k\to\infty}
 \frac{\text{\rm trace\,}\left[(P_{q_1}^{(1)}\otimes \cdots \otimes P_{q_k}^{(k)}\otimes I_\cH)Y\right]}{\text{\rm trace\,}\left[P_{q_1}^{(1)}\otimes \cdots \otimes P_{q_k}^{(k)}\right]}.
 $$
\end{theorem}
\begin{proof}
First, note that, under the given conditions, $Y$ should be  a positive operator and
$$
 0\leq \frac{\text{\rm trace\,}\left[(P_{q_1}^{(1)}\otimes \cdots \otimes P_{q_k}^{(k)}\otimes I_\cH)Y\right]}{\text{\rm trace\,}\left[P_{q_1}^{(1)}\otimes \cdots \otimes P_{q_k}^{(k)}\right]}
 \leq \|Y\|\dim \cH
$$
for any $q_i\geq 0$ and  $i\in \{1,\ldots, k\}$. Consequently, due to Lemma \ref{identity}, we have
$$
0\leq \text{\rm trace\,}\left[{\bf \Delta_{S\otimes {\it I}_\cH}}(Y)({\bf N}_{\leq {\bf q}}\otimes I_\cH)\right]\leq \|Y\|\dim \cH.
$$
Since $\{{\bf N}_{\leq {\bf q}}\}_{{\bf q}\in \ZZ_+^k}$ is an increasing multi-sequence
of positive operators convergent to ${\bf N}$ and  we have ${\bf \Delta_{S\otimes {\it I}_\cH}}(Y)\geq 0$,
we deduce that
\begin{equation*}
\begin{split}
\text{\rm trace\,}\left[{\bf \Delta_{S\otimes {\it  I}_\cH}}(Y)({\bf N}\otimes I_\cH)\right]&=\lim_{{\bf q}\in \ZZ_+^k}\text{\rm trace\,}\left[{\bf \Delta_{S\otimes {\it I}_\cH}}(Y)({\bf N}_{\leq {\bf q}}\otimes I_\cH)\right]=\sup_{{\bf q}\in \ZZ_+^k}\text{\rm trace\,}\left[{\bf \Delta_{S\otimes {\it I}_\cH}}(Y)({\bf N}_{\leq {\bf q}}\otimes I_\cH)\right].
\end{split}
\end{equation*}
Therefore, ${\bf \Delta_{S\otimes {\it I}_\cH}}(Y)({\bf N}\otimes I_\cH)$ is a trace class operator. Using again   Lemma \ref{identity}, we can complete the proof.
\end{proof}

Let ${\bf S}:=({\bf S}_1,\ldots, {\bf S}_k)$ be the universal model
 associated to the abstract  noncommutative polyball ${\bf B_n}$.
  An operator $M:\otimes_{i=1}^k F^2(H_{n_i})\otimes \cH\to
\otimes_{i=1}^k F^2(H_{n_i})\otimes \cK$  is called {\it  multi-analytic}  with respect to ${\bf S}$
 if $M({\bf S}_{i,j}\otimes I_\cH)=({\bf S}_{i,j}\otimes I_\cK)M$
for any $i\in \{1,\ldots,k\}$ and  $j\in \{1,\ldots, n_i\}$. In case $M$ is a partial isometry, we call
it {\it inner multi-analytic} operator.
We introduce the {\it curvature operator} ${\bf \Delta_{S\otimes {\it I}_\cH}}(\bf {K_T K_T^*})({\bf N}\otimes {\it I}_\cH)$ associated with each element ${\bf T}$ in the polyball ${\bf B_n}(\cH)$.

An element ${\bf T} \in  {\bf B_n}(\cH)$    is said to have
 {\it characteristic function} if there is a multi-analytic operator ${\bf  \Theta_T}:\otimes_{i=1}^k F^2(H_{n_i})\otimes \cE\to \otimes_{i=1}^k F^2(H_{n_i})\otimes \overline{{\bf \Delta_T}(I)(\cH)}$ with respect the universal model ${\bf S}$   such that $\bf { K_T  K_T^*} +{\bf \Theta_T}{\bf  \Theta_T^*}={\it I}$. The characteristic function is essentially unique if we restrict it to its support.
 We proved in \cite{Po-Berezin-poly} that   ${\bf T}$ has characteristic function if and only if
${\bf \Delta}_{{\bf S}\otimes I}(I-\bf { K_T  K_T^*})\geq 0.
$

\begin{theorem} \label{index} If  ${\bf T} \in  {\bf B_n}(\cH)$   has finite rank and characteristic function, then the curvature operator of ${\bf T}$  is trace class and \begin{equation*}\begin{split}
\text{\rm curv}({\bf T})&=\text{\rm trace\,}\left[{\bf \Delta_{S\otimes {\it I}_\cH}}(\bf {K_T K_T^*})({\bf N}\otimes {\it I}_\cH)\right]=\rank {\bf \Delta_{T}}(I)-\text{\rm trace\,}\left[\Theta_{\bf T}({\bf P}_\CC\otimes I)\Theta_{\bf T}^* ({\bf N}\otimes I_\cH)\right],
\end{split}
\end{equation*}
where $\Theta_{\bf T}$ is the characteristic function of ${\bf T}$.
\end{theorem}
\begin{proof}
 Applying   Theorem \ref{curva1} and Theorem \ref{trace-N}, we obtain
 \begin{equation*}\begin{split}
\text{\rm curv}({\bf T})&=\lim_{(q_1,\ldots, q_k)\in \ZZ_+^k}  \frac{\text{\rm trace\,}\left[ (P_{q_1}^{(1)}\otimes \cdots \otimes P_{q_k}^{(k)}\otimes I_\cH){\bf K_{T}} {\bf K_{T}^*}\right]}{\text{\rm trace\,}\left[P_{q_1}^{(1)}\otimes \cdots \otimes P_{q_k}^{(k)}\right]}\\
&=\rank {\bf \Delta_{T}}(I)-\text{\rm trace\,}\left[{\bf \Delta_{S\otimes {\it I}}}(\bf {\Theta_T \Theta_T^*})({\bf N}\otimes {\it  I}_\cH)\right]=\text{\rm trace\,}\left[{\bf \Delta_{S\otimes {\it I}}}(\bf {K_T K_T^*})({\bf N}\otimes {\it I}_\cH)\right]\\
&=\rank {\bf \Delta_{T}}(I)-\text{\rm trace\,}\left[{\bf \Theta}_{\bf T}({\bf P}_\CC\otimes I)\Theta_{\bf T}^* ({\bf N}\otimes I_\cH)\right].
\end{split}
\end{equation*}
The latter equality is due to the fact that ${\bf \Theta_T}$ is a multi-analytic operator and ${\bf \Delta_{S\otimes {\it I}}}(I)={\bf P}_\CC\otimes I$.
The proof is complete.
\end{proof}

 We remark that   ${\bf T}$ is pure,  i.e. $\lim_{q_i\to \infty}\Phi_{T_i}^{q_i}(I)=0$ in the weak operator topology for each $i\in\{1,\ldots,k\}$,
  if and only if
$ {\bf K_{T}}$ is an isometry. We also mention  that the set of all pure elements of the polyball ${\bf B_n}(\cH)$ coincides with the pure elements
${\bf X}=(X_1,\ldots, X_k)\in B(\cH)^{n_1}\times_c\cdots \times_c B(\cH)^{n_k}$ such that ${\bf \Delta_{X}}(I)\geq 0$. In what follows, we show that the curvature invariant can be used to detect the elements
${\bf T}\in B(\cH)^{n_1}\times \cdots\times B(\cH)^{n_k}$ which are unitarily equivalent  to ${\bf S}\otimes I_\cK$ for some finite dimensional  Hilbert space  $\cK$.

\begin{theorem}\label{comp-inv} If ${\bf T}=({ T}_1,\ldots, { T}_k)\in B(\cH)^{n_1}\times \cdots\times B(\cH)^{n_k}$, then the following statements are equivalent:
\begin{enumerate}
\item[(i)] ${\bf T}$ is unitarily equivalent  to ${\bf S}\otimes I_\cK$ for some finite dimensional Hilbert space $\cK$;
    \item[(ii)] ${\bf T}$  is a pure  finite rank element in the regular polyball  ${\bf B_n}(\cH)$ such that  ${\bf \Delta_{S\otimes {\it I}}}(I-{\bf K_{T}}{\bf K_{T}^*})\geq 0$ and
    $$\text{\rm curv}({\bf T})=\rank [{\bf \Delta_{T}}(I)].
    $$
\end{enumerate}
In this case, the Berezin kernel ${\bf K_T}$ is a unitary operator and
$${\bf T}_{i,j}={\bf K_T^*} ({\bf S}_{i,j}\otimes I_{\overline{{\bf \Delta_{T}}(I) (\cH)}}){\bf K_T}, \qquad i\in\{1,\ldots,k\},  j\in\{1,\ldots, n_i\}.
$$
\end{theorem}
\begin{proof}
 Assume that item (i) holds. Then there is unitary operator $U:\cH\to \otimes_{i=1}^k F^2(H_{n_i})\otimes \cK$ such that
$T_{i,j}=U^*({\bf S}_{i,j}\otimes I_\cK) U$ for all $i\in\{1,\ldots,k\}$ and   $j\in\{1,\ldots, n_i\}$. Consequently, using Corollary \ref{curva-maps}, we have
\begin{equation*}
\begin{split}
\text{\rm curv}({\bf T})&= \lim_{q_1\to\infty}\cdots\lim_{q_k\to\infty}
\frac{\text{\rm trace\,}\left[(id-\Phi_{{\bf S}_1\otimes I_\cK}^{q_1+1})\circ\cdots\circ (id-\Phi_{{\bf S}_k\otimes I_\cK}^{q_k+1})(I)\right]}
{\prod_{i=1}^k(1+n_i+\cdots + n_i^{q_i})}\\
&=\text{\rm curv}({\bf S}\otimes I_\cK)=\dim \cK.
\end{split}
\end{equation*}
On the other hand, $\rank [{\bf \Delta_{T}}(I)]=\rank [{\bf \Delta_{S\otimes {\it I}_\cK}}(I)]=\dim\cK$. Therefore,
 $\text{\rm curv}({\bf T})=\rank [{\bf \Delta_{T}}(I)]$.  The fact that ${\bf T}$ is pure is obvious. On the other hand, one can easily see that the Berezin kernel of ${\bf S}\otimes I_\cK$  is a unitary operator and, consequently, so is the Berezin kernel of ${\bf T}$. This completes the proof of the implication  $(i)\to (ii)$.

Assume that item (ii) holds. Since ${\bf \Delta_{S\otimes {\it I}}}(I-{\bf K_{T}}{\bf K_{T}^*})\geq 0$,  the $k$-tuple ${\bf T}$   has characteristic function ${\bf \Theta_T}:\otimes_{i=1}^k F^2(H_{n_i})\otimes \cE\to \otimes_{i=1}^k F^2(H_{n_i})\otimes \overline{{\bf \Delta_T}(I)(\cH)}$ which is  a multi-analytic operator with respect to the universal model ${\bf S}$ and
$\bf {K_T K_T^*} +{\bf \Theta_T}{\bf \Theta_T^*}={\it I}$.
The  support of ${\bf \Theta_T}$  is the  smallest reducing subspace
 $\supp ({\bf \Theta_T})\subset \otimes_{i=1}^k F^2(H_{n_i})\otimes \cE$
 under  each operator ${\bf S}_{i,j}$, containing   the co-invariant
  subspace $\cM:=\overline{{\bf \Theta_T^*}(\otimes_{i=1}^k F^2(H_{n_i})\otimes \overline{{\bf \Delta_T}(I)(\cH)})}$.
  Due to  Theorem 5.1 from \cite{Po-Berezin-poly}, we deduce that
 $$
 \supp ({\bf \Theta_T})=\bigvee_{(\alpha)\in \FF_{n_1}^+\times\cdots
  \times \FF_{n_k}^+}({\bf S}_{(\alpha)}\otimes I_\cH) (\cM)
 =\otimes_{i=1}^k F^2(H_{n_i})\otimes \cL,
 $$
 where $\cL:=({\bf P}_\CC\otimes I_\cH)\overline{{\bf \Theta^*}
 (\otimes_{i=1}^k F^2(H_{n_i})\otimes \overline{{\bf \Delta_T}(I)(\cH)})}$.
 Since ${\bf T}$ is pure, Theorem 6.3 from \cite{Po-Berezin-poly} implies that ${\bf \Theta_T}$ is a partial isometry.
Due to the fact that ${\bf S}_{i,j}$ are isometries, the initial space of ${\bf \Theta_T}$, i.e.,
$${\bf \Theta_T^*}(\otimes_{i=1}^k F^2(H_{n_i})\otimes \overline{{\bf \Delta_T}(I)(\cH)})
 =\{x\in \otimes_{i=1}^k F^2(H_{n_i})\otimes \cE: \ \|{\bf \Phi_T} x\|=\|x\|\}$$
  is
 reducing under each ${\bf S}_{i,j}$. Since
 the  support of ${\bf \Theta_T}$  is the  smallest reducing subspace
 $\supp ({\bf \Theta_T})\subset \otimes_{i=1}^k F^2(H_{n_i}))\otimes \cE$
 under  each operator ${\bf S}_{i,j}$, containing   the co-invariant
  subspace ${\bf \Theta_T^*}(\otimes_{i=1}^k F^2(H_{n_i})\otimes \overline{{\bf \Delta_T}(I)(\cH)})$,  we must have
   $\supp ({\bf \Theta_T})={\bf \Theta_T^*}(\otimes_{i=1}^k F^2(H_{n_i})\otimes \overline{{\bf \Delta_T}(I)(\cH)})$. Note that
   $\Phi:={\bf \Theta_T}|_{\supp ({\bf \Theta_T})}$ is an isometric multi-analytic operator and $\Phi \Phi^*={\bf \Theta_T}{\bf \Theta_T^*}$.
Due to Theorem \ref{index}, we  have
\begin{equation*}
\begin{split}
\text{\rm curv}({\bf T})&=\rank [{\bf \Delta_{T}}(I)]-\text{\rm trace\,}\left[{\bf \Delta_{S\otimes {\it  I}}}(\bf {\Theta_T \Theta_T^*})({\bf N}\otimes {\it I}_\cH)\right]\\
&\rank [{\bf \Delta_{T}}(I)]-\text{\rm trace\,}\left[\Phi({\bf P}_\CC\otimes I_\cL) \Phi^*({\bf N}\otimes I_\cH)\right].
\end{split}
\end{equation*}
Since $\text{\rm curv}({\bf T})=\rank {\bf \Delta_{T}}(I)$, we deduce that
$\text{\rm trace\,}\left[\Phi({\bf P}_\CC\otimes I_\cL) \Phi^*({\bf N}\otimes I_\cH)\right]=0$. Taking into account that the trace is faithful, we obtain
$\Phi({\bf P}_\CC\otimes I_\cL) \Phi^*({\bf N}\otimes I_\cH)=0$. The latter relation implies the equality
$\Phi({\bf P}_\CC\otimes I_\cL) \Phi^*(P_{q_1}^{(1)}\otimes \cdots \otimes P_{q_k}^{(k)}\otimes I_\cH)=0$ for any $(q_1,\ldots, q_k)\in \ZZ_+^k$. Hence, we have
$\Phi({\bf P}_\CC\otimes I_\cL) \Phi^*=0$. Since $\Phi$ is an isometry, we have
$\Phi({\bf P}_\CC\otimes I_\cL)=0$. Taking into account that $\Phi$ is a multi-analytic operator with respect to ${\bf S}$, we deduce that $\Phi=0$ and, consequently ${\bf \Theta_T}=0$. Now relation  ${\bf {K_T K_T^*}} +{\bf \Theta_T}{\bf \Theta_T^*}= I$ shows that ${\bf K_T}$ is a co-isometry. On the other hand, since ${\bf T}$ is pure, the Berezin kernel  associated with ${\bf T}$ is an isometry. Therefore ${\bf K_T}$ is a unitary operator. Since ${\bf T}_{i,j}={\bf K_T^*} ({\bf S}_{i,j}\otimes I_{\overline{{\bf \Delta_{T}}(I) (\cH)}}){\bf K_T}$ for $i\in\{1,\ldots,k\}$ and $j\in\{1,\ldots, n_i\}$,
the proof is complete.
\end{proof}

Let ${\bf S}:=({\bf S}_1,\ldots, {\bf S}_k)$, where ${\bf S}_i:=({\bf S}_{i,1},\ldots, {\bf S}_{i, n_i})$, be the universal model for the abstract polyball ${\bf B_n}$, and let $\cH$ be a Hilbert space. We say that $\cM$
 is an invariant subspace of $\otimes_{i=1}^k F^2(H_{n_i})\otimes
  \cH$ or  that $\cM$  is    invariant  under  ${\bf S}\otimes I_\cH$ if it is invariant under each operator
  ${\bf S}_{i,j}\otimes I_\cH$ for $i\in \{1,\ldots, k\}$ and $j\in \{1,\ldots, n_j\}$.

 \begin{definition} Given two invariant subspaces $\cM$ and $\cN$ under ${\bf S}\otimes I_\cH$,  we say that they are unitarily equivalent if there is a unitary operator $U:\cM\to \cN$ such that $U({\bf S}_{i,j}\otimes I_\cH)|_\cM=({\bf S}_{i,j}\otimes I_\cH)|_\cN U$ for any $i\in \{1,\ldots, k\}$  and  $j\in \{1,\ldots, n_i\}$.
 \end{definition}

According to Theorem 5.1 from \cite{Po-Berezin-poly}, if a subspace $\cM\subset \otimes_{i=1}^k F^2(H_{n_i})\otimes \cH$ is  co-invariant  under each operator ${\bf S}_{i,j}\otimes I_\cH$, then
\begin{equation*}
\overline{\text{\rm span}}\,\left\{\left({\bf S}_{1,\beta_1}\cdots {\bf S}_{k,\beta_k}\otimes
I_\cE\right)\cM:\ \beta_1\in \FF_{n_1}^+,\ldots, \beta_k\in \FF_{n_k}^+\right\}=\otimes_{i=1}^k F^2(H_{n_i})\otimes \cE,
\end{equation*}
where $\cE:=({\bf P}_\CC\otimes I_\cH)(\cM)$.
Consequently,
 a subspace
$\cR\subseteq \otimes_{i=1}^k F^2(H_{n_i})\otimes \cH$ is
 reducing under
${\bf S}\otimes I_\cH$
if and only if   there exists a subspace $\cG\subseteq \cH$ such
that
$
 \cR=\otimes_{i=1}^k F^2(H_{n_i})\otimes \cG.
$

We recall that Beurling \cite{Be} obtain a characterization of the invariant subspaces of the Hardy space $H^2(\DD)$ in terms of inner functions. This result was extended to the full Fock space $F^2(H_n)$ in \cite{Po-charact} and \cite{Po-analytic}.  On the other hand, it is well known \cite{Ru} that the lattice of the invariant subspaces for the Hardy space $H^2(\DD^k)$ is very complicated and contains many  invariant subspaces which are not of Beurling type. The same complicated situation occurs in the case of the tensor product $\otimes_{i=1}^k F^2(H_{n_i})$.
 Following the classical case, we say that $\cM$ is a {\it Beurling type invariant subspace} for ${\bf S}\otimes I_\cH$ if there is  an inner  multi-analytic operator
 $\Psi:\otimes_{i=1}^k F^2(H_{n_i})\otimes \cE \to
 \otimes_{i=1}^k F^2(H_{n_i})\otimes \cH$ with respect to ${\bf S}$,
 where  $\cE$ is a Hilbert space, such that
$\cM=\Psi\left[\otimes_{i=1}^k F^2(H_{n_i})\otimes \cE\right]$. In this case, $\Psi$ can be chosen to be isometric.
In \cite{Po-Berezin-poly}, we proved that $\cM$ is a Beurling type invariant subspace  for ${\bf S}\otimes I_\cH$
if and only if   $$
   (id-\Phi_{{\bf S}_1\otimes I_\cH})\circ \cdots \circ
    (id-\Phi_{{\bf S}_k\otimes I_\cH})(P_\cM)\geq 0,
   $$
   where $P_\cM$ is the orthogonal projection on $\cM$.
   If $\cM$  is a  Beurling type invariant subspace  of  ${\bf S}\otimes I_\cH$, then $({\bf S}\otimes I_\cH)|_\cM:=(({\bf S}_1\otimes I_\cH)|_\cM,\ldots, ({\bf S}_k\otimes I_\cH)|_\cM)$ is in the polyball ${\bf B_n}(\cM)$, where  $({\bf S}_i\otimes I_\cH)|_\cM:=(({\bf S}_{i,1}\otimes I_\cH)|_\cM,\ldots, ({\bf S}_{i, n_i}\otimes I_\cH)|_\cM)$. We say that $\cM$ has finite rank if $({\bf S}\otimes I_\cH)|_\cM$ has finite rank.

The next result shows that the curvature invariant completely classifies the
 finite rank  Beurling type invariant subspaces  of ${\bf S}\otimes I_\cH$  which do not contain reducing subspaces. In particular, the curvature invariant classifies  the  finite rank  Beurling type invariant subspaces  of $ F^2(H_{n_1})\otimes\cdots \otimes F^2(H_{n_k})$.
\begin{theorem}\label{classification} Let $\cM$ and $\cN$ be finite rank  Beurling type invariant subspaces  of $\otimes_{i=1}^k F^2(H_{n_i})\otimes \cH$ which do not contain reducing subspaces of ${\bf S}\otimes I_\cH$. Then $\cM$ and $\cN$ are unitarily equivalent if and only if
$$
\text{\rm curv}(({\bf S}\otimes I_\cH)|_\cM)=\text{\rm curv}(({\bf S}\otimes I_\cH)|_\cN).
$$
\end{theorem}
\begin{proof}
If $\cM$ is a   Beurling type invariant subspaces  of $\otimes_{i=1}^k F^2(H_{n_i})\otimes \cH$, then there is a Hilbert space $\cL$ and an isometric multi-analytic operator $\Psi:\otimes_{i=1}^k F^2(H_{n_i})\otimes \cL\to \otimes_{i=1}^k F^2(H_{n_i})\otimes \cH$ such that
$\cM=\Psi[\otimes_{i=1}^k F^2(H_{n_i})\otimes \cL]$. Since $P_\cM=\Psi \Psi^*$, we have
\begin{equation*}
\begin{split}
{\bf \Delta}_{({\bf S}\otimes I)|_\cM}(I_\cM)&= \left(id -\Phi_{({\bf S}_1\otimes I)|_\cM}\right)\circ \cdots \circ\left(id -\Phi_{ ({\bf S}_k\otimes I)|_\cM}\right)(P_\cM)\\
&=\Psi \left(id -\Phi_{{\bf S}_1\otimes I}\right)\circ \cdots \circ\left(id -\Phi_{ {\bf S}_k\otimes I}\right)(I)\Psi^*|_\cM=\Psi({\bf P}_\CC\otimes I_\cL)\Psi^*|_\cM.
\end{split}
\end{equation*}
Let $\{\ell_\omega\}_{\omega\in \Omega}$ be an orthonormal basis for $\cL$. Note that $\{v_\omega:=\Psi(1\otimes \ell_\omega):\ \omega\in \Omega\}$ is an orthonormal  set
and
$$
\{
\Psi( e^1_{\beta_1}\otimes \cdots \otimes  e^k_{\beta_k}\otimes \ell_\omega): \ \beta_i\in\FF_{n_i}^+, i\in \{1,\ldots, k\},  \omega\in \Omega\}
$$
is an orthonormal basis for $\cM$.  Note also that
$\overline{\Psi({\bf P}_\CC\otimes I_\cL)\Psi^*(\cM)}$ coincides with the  closure of the range of the defect operator
${\bf \Delta}_{({\bf S}\otimes I)|_\cM}(I_\cM)$ and also to  the closed linear span of $\{v_\omega:=\Psi(1\otimes \ell_\omega):\ \omega\in \Omega\}$.
Consequently,
$
\rank [({\bf S}\otimes I)|_\cM]=\text{\rm card\,} \Omega=\dim \cL.
$
Now,  assume that $\rank ({\bf S}\otimes I)|_\cM)=p=\dim\cL$. Since $\Psi$ is a multi-analytic operator and $P_\cM=\Psi\Psi^*$, we have
\begin{equation*}
\begin{split}
\left(id -\Phi_{({\bf S}_1\otimes I)|_\cM}^{q_1+1}\right)\circ \cdots \circ\left(id -\Phi_{ ({\bf S}_k\otimes I)|_\cM}^{q_k+1}\right)(I_\cM)
&=\left(id -\Phi_{{\bf S}_1\otimes I}^{q_1+1}\right)\circ \cdots \circ\left(id -\Phi_{ {\bf S}_k\otimes I}^{q_k+1}\right)(P_\cM)\\
&=\Psi \left(id -\Phi_{{\bf S}_1\otimes I_\cL}^{q_1+1}\right)\circ \cdots \circ\left(id -\Phi_{ {\bf S}_k\otimes I_\cL}^{q_k+1}\right)(I)\Psi^*|_\cM.
\end{split}
\end{equation*}
Hence, we deduce that
\begin{equation*}
\begin{split}
\text{\rm trace\,}&\left[\left(id -\Phi_{({\bf S}_1\otimes I)|_\cM}^{q_1+1}\right)\circ \cdots \circ\left(id -\Phi_{ ({\bf S}_k\otimes I)|_\cM}^{q_k+1}\right)(I_\cM)\right]= \text{\rm trace\,}\left[\left(id -\Phi_{{\bf S}_1\otimes I_\cL}^{q_1+1}\right)\circ \cdots \circ\left(id -\Phi_{ {\bf S}_k\otimes I_\cL}^{q_k+1}\right)(I)\right]\\
&\qquad
=\text{\rm trace\,}\left[\left(id -\Phi_{{\bf S}_1}^{q_1+1}\right)\circ \cdots \circ\left(id -\Phi_{{\bf S}_k}^{q_k+1}\right)(I)\right]\dim\cL.
\end{split}
\end{equation*}
Using Corollary \ref{curva-maps} and the fact that
$\text{\rm curv}({\bf S})=1$, we deduce that
$\text{\rm curv}(({\bf S}\otimes I)|_\cM)=\dim\cL=p=\rank [({\bf S}\otimes I)|_\cM].
$
Now, if  $\cM$ and $\cN$ are  finite rank  Beurling type invariant subspaces  of $\otimes_{i=1}^k F^2(H_{n_i})\otimes \cH$ and $({\bf S}\otimes I)|_\cM$ is unitarily equivalent to $({\bf S}\otimes I)|_\cN$, then it is clear that
$
\text{\rm curv}(({\bf S}\otimes I)|_\cM)=\text{\rm curv}(({\bf S}\otimes I)|_\cN).
$
To prove the converse, assume that the later equality holds. As we saw above, we must have
$\rank(({\bf S}\otimes I)|_\cM)=\rank(({\bf S}\otimes I)|_\cN)$. Consequently, the defect spaces associated with $({\bf S}\otimes I)|_\cM$ and $({\bf S}\otimes I)|_\cN$ have the same dimension. Using the Wold decomposition from \cite{Po-Berezin-poly}, we conclude that $({\bf S}\otimes I)|_\cM$ is unitarily equivalent to $({\bf S}\otimes I)|_\cN$.
The proof is complete.
\end{proof}

\section{Invariant subspaces and multiplicity invariant}

In this section, we prove the existence of the multiplicity  of any
invariant subspace of the tensor product $F^2(H_{n_1})\otimes\cdots\otimes F^2(H_{n_k})\otimes \cE$, where $\cE$ is a finite dimensional Hilbert space.
We provide several asymptotic  formulas for the multiplicity  and an important  connection with the curvature invariant. We show that  the range of both the curvature and the multiplicity  invariant  is the interval $[0,\infty)$. We  prove that there is an  uncountable family of non-unitarily  equivalent  pure elements in the noncommutative ball with rank one and  the same prescribed curvature. We also show that  the curvature invariant detects  the {\it inner sequences} of multipliers of $\otimes_{i=1}^k F^2(H_{n_i})$ when it takes the extremal value zero.

Let $\cM$ be an invariant subspace of the tensor product $F^2(H_{n_1})\otimes\cdots\otimes F^2(H_{n_k})\otimes \cE$, where $\cE$ is a finite dimensional Hilbert space. We introduce  the multiplicity of $\cM$ by setting
\begin{equation*}
\begin{split}
m(\cM)&:=\lim_{m\to\infty}\frac{1}{\left(\begin{matrix} m+k\\ k\end{matrix}\right)}\sum_{{q_1\geq 0,\ldots, q_k\geq 0}\atop {q_1+\cdots +q_k\leq m}}\frac{\text{\rm trace\,}\left[P_{\cM}(P_{q_1}^{(1)}\otimes \cdots \otimes P_{q_k}^{(k)}\otimes I_\cE) \right]}{\text{\rm trace\,}\left[P_{q_1}^{(1)}\otimes \cdots \otimes P_{q_k}^{(k)}\right]}.
\end{split}
\end{equation*}
Similarly, we define $m(\cM^\perp)$ by replacing $P_\cM$ with $P_{\cM^\perp}$.
Note that $m(\cM)+m(\cM^\perp)=\dim \cE$.

\begin{theorem}\label{multiplicity} Let $\cM$ be an invariant subspace of the tensor product $F^2(H_{n_1})\otimes\cdots\otimes F^2(H_{n_k})\otimes \cE$, where $\cE$ is a finite dimensional Hilbert space. Then the multiplicity of $\cM$ exists and satisfies the equations
 \begin{equation*}
\begin{split}
m(\cM)&= \lim_{(q_1,\ldots, q_k)\in \ZZ_+^k}
 \frac{\text{\rm trace\,}\left[P_{\cM}(P_{q_1}^{(1)}\otimes \cdots \otimes P_{q_k}^{(k)}\otimes I_\cE) \right]}{\text{\rm trace\,}\left[P_{q_1}^{(1)}\otimes \cdots \otimes P_{q_k}^{(k)}\right]}=\lim_{q_1\to\infty}\cdots\lim_{q_k\to\infty}
\frac{\text{\rm trace\,}\left[P_{\cM}(P_{\leq(q_1,\ldots, q_k)}\otimes I_\cE)  \right]}{\text{\rm trace\,}\left[P_{\leq(q_1,\ldots, q_k)}\right]}\\
&=\lim_{m\to\infty}\frac{1}{\left(\begin{matrix} m+k-1\\ k-1\end{matrix}\right)}\sum_{{q_1\geq 0,\ldots, q_k\geq 0}\atop {q_1+\cdots +q_k= m}} \frac{\text{\rm trace\,}\left[ P_\cM (P_{q_1}^{(1)}\otimes \cdots \otimes P_{q_k}^{(k)}\otimes I_\cE)\right]}{\text{\rm trace\,}\left[P_{q_1}^{(1)}\otimes \cdots \otimes P_{q_k}^{(k)}\right]}\\
&=\dim \cE -\text{\rm curv}({\bf M}),
\end{split}
\end{equation*}
where  ${\bf M}:=(M_1,\ldots, M_k)$ with $M_i:=(M_{i,1},\ldots, M_{i,n_i})$ and $M_{i,j}:=P_{\cM^\perp}({\bf S}_{i,j}\otimes I_\cE)|_{\cM^\perp}$.
\end{theorem}
\begin{proof}
 Since $\cM$ is an invariant subspace under each operator  ${\bf S}_{i,j}\otimes I_\cE$ for $i\in \{1,\ldots, k\}$, $j\in\{1,\ldots, n_i\}$, we have   $M_{i,j}^* M_{r,s}^*=({\bf S}_{i,j}^*\otimes I_\cE) ({\bf S}_{r,s}^*\otimes I_\cE)|_{\cM^\perp}$ and  deduce that
\begin{equation}
\label{De}
{\bf \Delta_{M}^p}(I_{\cM^\perp})=P_{\cM^\perp}{\bf \Delta_{S\otimes {\it I}}^p}(I)|_{\cM^\perp}\geq 0
\end{equation}
for any ${\bf p}=(p_1,\ldots, p_k)$ with $p_i\in \{0,1\}$. Therefore, ${\bf M}$ is in the polyball ${\bf B_n}(\cM^\perp)$ and has finite rank. Since $\cM^\perp$ is invariant under  ${\bf S}_{i,j}^*\otimes I_\cE$ and using  relation \eqref{De}, we obtain
\begin{equation*}
\begin{split}
\text{\rm trace\,}[\Phi_{M_1}^{q_1}\circ \cdots \circ \Phi_{M_k}^{q_k} ({\bf \Delta_{M}}(I_{\cM^\perp}))]
&=\text{\rm trace\,}[P_{\cM^\perp}\Phi_{\bf {S}_1\otimes {\it I}}^{q_1}\circ \cdots \circ \Phi_{{\bf S}_k\otimes {\it I}}^{q_k} ({\bf \Delta_{S\otimes {\it I}}}(I))|_{\cM^\perp}]\\
&=\text{\rm trace\,}[P_{\cM^\perp}(P_{q_1}^{(1)}\otimes \cdots \otimes P_{q_k}^{(k)}\otimes I_\cE)]
\end{split}
\end{equation*}
for any $q_i\in \ZZ_+$. Therefore, due to Theorem \ref{curva1} and Corollary \ref{curva-maps},  \text{\rm curv}({\bf M}) exists and  one can easily complete the proof.
\end{proof}

Using Theorem \ref{multiplicity} and Lemma \ref{identity}, we deduce the following result.
\begin{corollary}
Let $\cM$ be an invariant subspace of tensor product $F^2(H_{n_1})\otimes\cdots\otimes F^2(H_{n_k})\otimes \cE$, where $\cE$ is a finite dimensional Hilbert space. Then the    multiplicity of $\cM$ satisfies the formula
 \begin{equation*}
m(\cM)= \lim_{{\bf q}\in \ZZ_+^k}\text{\rm trace\,}\left[{\bf \Delta_{S\otimes {\it I}_\cE}}(P_\cM)({\bf N}_{\leq {\bf q}}\otimes I_\cE)\right],
\end{equation*}
where ${\bf N}_{\leq {\bf q}}$ is defined by relation \eqref{NQ}.
\end{corollary}

\begin{corollary}\label{Fang} If $\cM$ is an invariant subspace for the Hardy space  $H^2(\DD^k)$ of the polydisc,   then
$$
m(\cM):= \lim_{m\to\infty}\frac{
\text{\rm trace\,}\left[ P_{\cM}(P_{\leq m}\otimes I_\cE) \right]}{\text{\rm trace\,}[ P_{\leq m}]}
$$
exists, where $P_{\leq m}$ is the orthogonal projection onto the polynomials of degree $\leq m$.
\end{corollary}
   The existence of the limit in Corollary \ref{Fang}  was first proved  by Fang (see \cite{Fang}) using different methods.  According to Theorem \ref{multiplicity}, we have several other asymptotic formulas  for the multiplicity.

\begin{theorem}
Given a function $\kappa :\NN\to \NN$ and ${\bf n}^{(i)}\in \NN^{\kappa(i)}$ for  $i\in \{1,\ldots,p\}$, let ${\bf S}^{({\bf n}^{(i)})}$ and ${\bf S}^{({\bf n}^{(1)},\ldots, {\bf n}^{(p)})}$ be the universal models of the polyballs ${\bf B}_{{\bf n}^{(i)}}$ and
${\bf B}_{({\bf n}^{(1)},\ldots, {\bf n}^{(p)})}$, respectively.
For each $i\in \{1,\ldots, p\}$, assume that
\begin{enumerate}
\item[(i)] $\cE_i$ is a finite dimensional Hilbert space with $\dim\cE_i=q_i\in \NN$;
\item[(ii)] $\cM_i$ is an invariant subspace under ${\bf S}^{({\bf n}^{(i)})}\otimes I_{\cE_i}$;
\item[(iii)] $c_i:=\text{\rm curv} \left(P_{\cM_i^\perp}\left({\bf S}^{({\bf n}^{(i)})}\otimes I_{\cE_i}\right)|_{\cM_i^\perp}\right)$;
\item[(iv)]  $c:=\text{\rm curv} \left(P_{\cM^\perp}\left({\bf S}^{({\bf n}^{(1)},\ldots, {\bf n}^{(p)})}\otimes I_{\cE_1\otimes \cdots \otimes \cE_p}\right)|_{\cM^\perp}\right)$,
\end{enumerate}
where, under the appropriate  identification, $\cM:=\cM_1\otimes \cdots \otimes \cM_p$ is viewed as an invariant subspace for ${\bf S}^{({\bf n}^{(1)},\ldots, {\bf n}^{(p)})}\otimes I_{\cE_1\otimes\cdots \otimes \cE_p}$.
Then the curvature invariant satisfies the equation
$$
c=q_1\cdots q_p-\prod_{i=1}^p(q_i-c_i)
$$
 and the
 multiplicity invariant satisfies the equation
 $
 m(\cM_1\otimes \cdots \otimes \cM_p)=\prod_{i=1}^pm(\cM_i).
 $
\end{theorem}
\begin{proof} For each $i\in\{1,\ldots, p\}$ let ${\bf n}^{(i)}:=(n_1^{(i)},\ldots, n_{\kappa(i)}^{(i)})\in \NN^{\kappa(i)}$.
If $j\in \{1,\ldots, \kappa(i)\}$, let $F^2(H_{n_j}^{(i)})$ be the full Fock space with $n_j^{(i)}$ generators, and denote by $P_{q_j^{(i)}}^{(n_j^{(i)})}$
the orthogonal projection  of $F^2(H_{n_j^{(i)}})$ onto the span of all homogeneous polynomials of $F^2(H_{n_j^{(i)}})$ of degree equal to $q_j^{(i)}\in \ZZ_+$.
Set
 $$
 Q_1:=P_{q_1^{(1)}}^{(n_1^{(1)})}\otimes \cdots \otimes P_{q_{\kappa(1)}^{(1)}}^{(n_{\kappa(1)}^{(1)})}, \quad \cdots, \quad
 Q_p:=P_{q_1^{(p)}}^{(n_1^{(p)})}\otimes \cdots \otimes P_{q_{\kappa(p)}^{(p)}}^{(n_{\kappa(p)}^{(p)})}
 $$
and let $U$ be the unitary operator which provides the  canonical  identification of the Hilbert tensor product
$\otimes_{i=1}^p\left[ F^2(H_{n_1}^{(i)})\otimes \cdots \otimes F^2(H_{n_{\kappa(i)}}^{(i)})\otimes \cE_i\right]$ with
$\left\{\otimes_{i=1}^p\left[ F^2(H_{n_1}^{(i)})\otimes \cdots \otimes F^2(H_{n_{\kappa(i)}}^{(i)})\right]\right\}\otimes (\cE_1\otimes \cdots\otimes \cE_p)$. Note that
$$
U[(Q_1\otimes I_{\cE_1})\otimes \cdots \otimes (Q_p\otimes I_{\cE_p})]
=[Q_1\otimes \cdots \otimes Q_p\otimes I_{\cE_1\otimes\cdots\otimes \cE_p}]U
$$
and $U(\cM_1\otimes \cdots \otimes \cM_p)$ is an invariant subspace under
${\bf S}^{({\bf n}^{(1)},\ldots, {\bf n}^{(p)})}\otimes I_{\cE_1\otimes\cdots \otimes \cE_p}$.
Using   Theorem \ref{multiplicity}, we obtain
 \begin{equation*}
 \begin{split}
m\left(\otimes_{i=1}^k \cM_i\right)&=  \lim   \frac{\text{\rm trace\,}\left[P_{U(\cM_1\otimes \cdots \otimes \cM_p)}
 \left(Q_1
 \otimes \cdots \otimes Q_p\otimes I_{\cE_1\otimes\cdots\otimes \cE_p}
 \right)\right]}{\text{\rm trace\,}\left[Q_1\otimes \cdots \otimes Q_p \right]}\\
 &=
  \lim   \frac{\text{\rm trace\,}\left[U^*P_{U(\cM_1\otimes \cdots \otimes \cM_p)} UU^*
 \left(Q_1
 \otimes \cdots \otimes Q_p\otimes I_{\cE_1\otimes\cdots\otimes \cE_p}
 \right)U\right]}{\text{\rm trace\,}\left[Q_1\otimes \cdots \otimes Q_p \right]}\\
 &=
 \lim   \frac{\text{\rm trace\,}\left\{P_{\cM_1\otimes \cdots \otimes \cM_p} \left[
 \left(Q_1\otimes I_{\cE_1}\right)
 \otimes \cdots \otimes \left(Q_p\otimes I_{\cE_p}\right)\right]
  \right\}}{\text{\rm trace\,}\left[Q_1\otimes \cdots \otimes Q_p \right]}
 \end{split}
 \end{equation*}
 where the limit is taken over ${(q_1^{(1)},\ldots, q_{\kappa(1)}^{(1)},\ldots, q_1^{(p)},\ldots, q_{\kappa(p)}^{(p)})\in \ZZ_+^{\kappa(1)+\cdots +\kappa(p)}}$.
 Note that  the latter limit is equal to the product
$$
 \lim_{(q_1^{(1)},\ldots, q_{\kappa(1)}^{(1)})\in \ZZ_+^{\kappa(1)}}   \frac{\text{\rm trace\,}\left[P_{\cM_1}
 \left(Q_1\otimes I_{\cE_1} \right)\right]}
 {\text{\rm trace\,}\left[Q_1 \right]}\quad
  \cdots \quad  \lim_{(q_1^{(p)},\ldots, q_{\kappa(p)}^{(p)})\in \ZZ_+^{\kappa(p)}}
  \frac{\text{\rm trace\,}\left[P_{\cM_p}\left(Q_p \otimes I_{\cE_p}\right)
 \right]}
 {\text{\rm trace\,}\left[ Q_p\right]},
 $$
which, due to Theorem \ref{multiplicity}, is equal  to
$\prod_{i=1}^pm\left( \cM_i\right)$. Therefore,
$
m\left(\otimes_{i=1}^p \cM_i\right)=\prod_{i=1}^pm\left( \cM_i\right).
$
For each $i\in \{1,\ldots, p\}$, set
$${\bf M}^{(i)}:=P_{\cM_i^\perp}\left({\bf S}^{({\bf n}^{(i)})}\otimes I_{\cE_i}\right)|_{\cM_i^\perp}\quad \text{ and } \quad
{\bf M}:=P_{\cM^\perp}\left({\bf S}^{({\bf n}^{(1)},\ldots, {\bf n}^{(p)})}\otimes I_{\cE_1\otimes \cdots \otimes \cE_p}\right)|_{\cM^\perp}.
$$
Since, due to Thorem \ref{multiplicity},
$\text{\rm curv}({\bf M})=\dim (\cE_1\otimes\cdots \otimes  \cE_p)-m(\cM)$  and  $ \text{\rm curv}({\bf M}^{(i)})=\dim (\cE_i)-m(\cM_i),
$ we deduce the corresponding identity for the curvature.
  The proof is complete.
\end{proof}

Note that, in  particular,  if  $\cM\subset H^2(\DD^n)\otimes \CC^r$ and $\cN\subset H^2(\DD^p)\otimes \CC^q$ are invariant subspaces, then so is
$\cM\otimes \cN\subset H^2(\DD^{n+p})\otimes \CC^{rq}$ and the multiplicity invariant has the multiplicative property
$m(\cM\otimes \cN)=m(\cM)m(\cN)$.

We recall that a Beurling type characterization of the invariant subspaces of the full Fock space was obtained in \cite{Po-charact}.
\begin{corollary} \label{for-perp} For each  $i\in \{1,\ldots, k\}$, let  $\cM_i$  be an invariant subspace  of $F^2(H_{n_i})$. Then  the tensor product $\cM:=\otimes_{i=1}^k \cM_i$ is an invariant subspace  of $\otimes_{i=1}^k F^2(H_{n_i})$ and
\begin{equation*}
\begin{split}
\text{\rm curv}(P_{\cM^\perp} {\bf S}|_{\cM^\perp})&=1-\prod_{i=1}^k\left(1-\text{\rm curv}_i(P_{\cM_i^\perp} {\bf S}_i|_{\cM_i^\perp})\right)\quad \text{ and } \quad m\left(\otimes_{i=1}^k \cM_i\right)=\prod_{i=1}^km_i\left( \cM_i\right),
\end{split}
\end{equation*}
where $\text{\rm curv}_i$   and  $m_i$ are  the curvature  and  the  multiplicity   with respect to $F^2(H_{n_i})$, respectively.

\end{corollary}

\begin{theorem}\label{quotients-inv} Let $\cM$, $\cN$ be   invariant subspaces of
$\otimes_{i=1}^k F^2(H_{n_i})\otimes \cE$ such that they do not contain nontrivial reducing subspaces for the universal model ${\bf S}\otimes I_\cE$.
 If $P_{\cM^\perp}({\bf S}\otimes I_\cE)|_{\cM^\perp}$ is unitarily equivalent to $P_{\cN^\perp}({\bf S}\otimes I_\cE)|_{\cN^\perp}$, then  there is a unitary operator $U\in B(\cE)$ such that
$$(I\otimes U)P_\cM=P_\cN(I\otimes U).
$$
Conversely, if $\cM$, $\cN$ are  Beurling type  invariant subspaces  and
 there is a unitary operator $U\in B(\cE)$ such that
$(I\otimes U)P_\cM=P_\cN(I\otimes U)$,
then $P_{\cM^\perp}({\bf S}\otimes I_\cE)|_{\cM^\perp}$ is unitarily equivalent to $P_{\cN^\perp}({\bf S}\otimes I_\cE)|_{\cN^\perp}$.
\end{theorem}
\begin{proof}  Let  $\cM$, $\cN$ be   invariant subspaces of
$\otimes_{i=1}^k F^2(H_{n_i})\otimes \cE$ such that they do not contain nontrivial reducing subspaces for  ${\bf S}\otimes I$. Denote  ${\bf S}_{(\alpha)}:= {\bf S}_{1,\alpha_1}\cdots {\bf S}_{k,\alpha_k}$ if  $(\alpha):=(\alpha_1,\ldots, \alpha_k)$ is  in $ \FF_{n_1}^+\times \cdots \times\FF_{n_k}^+$, and note that the Cuntz-Toeplitz algebra \cite{Cu} generated by the shifts ${\bf S}_{i,j}$ satisfies the relation
$$C^*({\bf S}_{i,j})=\overline{\text{\rm span}}\{{\bf S}_{(\alpha)}{\bf S}_{(\beta)}^*: \ (\alpha), (\beta)\in \FF_{n_1}^+\times \cdots \times \FF_{n_k}^+\}.
$$
Let $\Psi_1:C^*({\bf S}_{i,j})\to B(\cM^\perp)$ and $\Psi_2:C^*({\bf S}_{i,j})\to B(\cN^\perp)$ be the unital completely positive linear maps defined by
$\Psi_1({\bf S}_{(\alpha)}{\bf S}_{(\beta)}^*):=P_{\cM^\perp}{\bf S}_{(\alpha)}{\bf S}_{(\beta)}^*|_{\cM^\perp}$ and
$\Psi_2({\bf S}_{(\alpha)}{\bf S}_{(\beta)}^*):=P_{\cN^\perp}{\bf S}_{(\alpha)}{\bf S}_{(\beta)}^*|_{\cN^\perp}$ for any $(\alpha), (\beta)\in \FF_{n_1}^+\times \cdots \times \FF_{n_k}^+$, respectively.
Consider the $*$-representations $\pi_1:C^*({\bf S}_{i,j})\to B(\otimes_{i=1}^kF^2(H_{n_i})\otimes \cE)$  and  $\pi_2:C^*({\bf S}_{i,j})\to B(\otimes_{i=1}^kF^2(H_{n_i})\otimes \cE)$ defined by
$\pi_1(X):=X\otimes I_{\cE}$ and $\pi_2(X):=X\otimes I_{\cE}$ for any
$X\in C^*({\bf S}_{i,j})$, respectively.
Note that $\Psi_s(X)=V_s^*\pi_s(X)V_s$ for $s=1,2$, where
$V_1:\cM^\perp\to \otimes_{i=1}^k F^2(H_{n_i})\otimes \cE$  and
$V_2:\cN^\perp\to \otimes_{i=1}^k F^2(H_{n_i})\otimes \cE$  are  the injection maps.
Note that if $\cM$ is an   invariant subspace of
$\otimes_{i=1}^k F^2(H_{n_i})\otimes \cH$, then $\cM$  does not contain nontrivial reducing subspaces for  ${\bf S}\otimes I_\cH$ if and only if $\cM^\perp$ is a cyclic subspace  for ${\bf S}\otimes I_\cH$, i.e.
\begin{equation*}
\overline{\text{\rm span}}\,\left\{\left({\bf S}_{1,\beta_1}\cdots {\bf S}_{k,\beta_k}\otimes
I_\cE\right)\cM^\perp:\ \beta_1\in \FF_{n_1}^+,\ldots, \beta_k\in \FF_{n_k}^+\right\}=\otimes_{i=1}^k F^2(H_{n_i})\otimes \cH.
\end{equation*}
Consequently, $\pi_1$ and $\pi_2$  are  minimal Stinespring dilations of the unital completely positive linear maps $\Psi_1$ and $\Psi_2$, respectively.
  Now, assume that there is a unitary operator $Z:\cM^\perp\to \cN^\perp$ such that  $ZP_{\cM^\perp}({\bf S}_{i,j}\otimes I_\cE)|_{\cM^\perp}=P_{\cN^\perp}({\bf S}_{i,j}\otimes I_\cE)|_{\cN^\perp}Z$ for any $i\in\{1,\ldots, k\}$ and $j\in \{1,\ldots, n_i\}$.
 It is easy to see that $Z\Psi_1(X)=\Psi_2(X)Z$ for any $X\in C^*({\bf S}_{i,j})$.
 Now, using standard arguments  concerning the uniqueness of minimal dilations of completely positive maps of $C^*$-algebras (see \cite{Arv}), we deduce that there is a unique unitary operator $W\in B(\otimes_{i=1}^k F^2(H_{n_i})\otimes \cE)$ such that $W\pi_1(X)=\pi_2(X)W$ for any $X\in C^*({\bf S}_{i,j})$, and $WV_1=V_2Z$.
 Consequently, $WV_1V_1^*W^*=V_2UU^*V_2^*=V_2V_2^*$, which is equivalent to
 $WP_{\cM}=P_{\cN}W$.
Since $C^*({\bf S}_{i,j})$ is irreducible (see Lemma 5.5 from \cite{Po-Berezin-poly})
we must have $W=I\otimes U$ where $U\in B(\cE)$ is a unitary operator.

Now, we prove the converse. Assume that $U\in B(\cE)$  is a unitary operator such that
$(I\otimes U)P_\cM=P_\cN(I\otimes U)$, where $\cM$, $\cN$ are Beurling type  invariant subspaces of
$\otimes_{i=1}^k F^2(H_{n_i})\otimes \cE$ such that they do not contain nontrivial reducing subspaces for the universal model ${\bf S}\otimes I_\cE$. Therefore, we can find isometric multi-analytic operators
$\psi_s:\otimes_{i=1}^k F^2(H_{n_i})\otimes \cE_s\to \otimes_{i=1}^k F^2(H_{n_i})\otimes \cE$, $s=1,2$, such that $\psi_1 \psi_1^*=P_\cM$ and $\psi_2 \psi_2^*=P_\cN$. Consequently, setting $\psi:=(I\otimes U^*)\psi_2$, the relations above imply $\psi_1\psi_1^*=\psi\psi^*$. Since $\psi_1, \psi$ are multi-analytic operators and
${\bf P}_\CC= (id-\Phi_{{\bf S}_1})\cdots (id-\Phi_{{\bf S}_k})(I)$, we obtain
$$\|({\bf P}_\CC\otimes I_{\cE_1})\psi_1^*x\|=\|({\bf P}_\CC\otimes I_{\cE_2})\psi^*x\|,\qquad x\in \otimes_{i=1}^k F^2(H_{n_i})\otimes \cE.
$$
Define the operator $A\in B(\cE_1,\cE_2)$ by setting
$A(({\bf P}_\CC\otimes I_{\cE_1})\psi_1^*x):=({\bf P}_\CC\otimes I_{\cE_2})\psi^*x$ for $x\in \otimes_{i=1}^k F^2(H_{n_i})\otimes \cE$.
Since $\psi_1$ is an isometric multi-analytic operator and $\cM$ has no nontrivial reducing subspaces for the universal model ${\bf S}\otimes I_\cE$, we must have
 $\cE_1=({\bf P}_\CC\otimes I_{\cE_1})\psi_1^*(\otimes_{i=1}^k F^2(H_{n_i})\otimes \cE)$.
 A similar result holds for $\cE_2$.  Therefore, the operator $A$ is unitary and
 $\psi({\bf P}_\CC\otimes I_{\cE_2})=\psi_1({\bf P}_\CC\otimes I_{\cE_1})A^*$.
 Hence, $\psi A|_{\CC\otimes \cE_1}=\psi_1|_{\CC\otimes \cE_1}$ and, due to the analyticity of $\psi $ and $\psi_1$,  we have $\psi(I\otimes A)=\psi_1$.
Since $\psi:=(I\otimes U^*)\psi_2$, we deduce that
$(I\otimes U)\psi_1=\psi_2(I\otimes A)$. Now, note that the unitary operator $I\otimes U$ takes $\cM$ onto $\cN$ and $\cM^\perp$ onto $\cN^\perp$.
 Since $({\bf S}_{i,j}^*\otimes I_\cE)(I\otimes U^*)=(I\otimes U^*)({\bf S}_{i,j}^*\otimes I_\cE)$ and $\cM^\perp$, $\cN^\perp$ are invariant subspaces under
 ${\bf S}_{i,j}^*\otimes I_\cE$, we deduce that
 $({\bf S}_{i,j}^*\otimes I_\cE)|_{\cM^\perp}\Lambda^*=\Lambda^*({\bf S}_{i,j}^*\otimes I_\cE)|_{\cN^\perp}$, where
  $\Lambda:=(I\otimes U)|_{\cM^\perp}$. Consequently,
  $$
  \Lambda P_{\cM^\perp}({\bf S}_{i,j}\otimes I_\cE)|_{\cM^\perp}=P_{\cN^\perp}({\bf S}_{i,j}\otimes I_\cE)|_{\cN^\perp}\Lambda
  $$
  for any $i\in\{1,\ldots, k\}$ and $j\in \{1,\ldots, n_i\}$.
This completes the proof.
\end{proof}

We remark that  when $\cM$, $\cN$ are nontrivial invariant subspaces of
$\otimes_{i=1}^k F^2(H_{n_i})$, Theorem \ref{quotients-inv} shows that
 $P_{\cM^\perp}{\bf S} |_{\cM^\perp}$ is unitarily equivalent to $P_{\cN^\perp}{\bf S}|_{\cN^\perp}$ if and only $\cM=\cN$.

\begin{theorem} \label{value}Let ${\bf n}=(n_1,\ldots, n_k)\in \NN^k$ be such that $n_i\geq 2$ and $n_j\geq 2$ for some $i,j\in \{1,\ldots, k\}$, $i\neq j$. Then, for each $t\in(0,1)$,  there exists an uncountable family $\{T^{(\omega)}(t)\}_{\omega\in \Omega}$ of pure elements in the regular polyball with the following properties:
\begin{enumerate}
\item[(i)] $T^{(\omega)}(t)$ in not unitarily equivalent to $T^{(\sigma)}(t)$ for any $\omega, \sigma\in \Omega$, $\omega\neq \sigma$.
\item[(ii)] $\rank [T^{(\omega)}(t)]=1$ and
$\text{\rm curv}\, [T^{(\omega)}(t)]=t$ for all $\omega\in \Omega.
$
\end{enumerate}
\end{theorem}

\begin{proof}    Assume that $n_i\geq 2$. If $t\in[0,1)$, there exists a subsequence  of natural numbers $\{k_p\}_{p=1}^N$, $1\leq k_1<k_2<\cdots$, where $N\in \NN$ or $N=\infty$,
and $d_p\in\{1,2,\ldots, n_i-1\}$,  such that $1-t=\sum_{p=1}^N \frac{d_p}{n_i^{k_p}}$. Recall that $\FF_{n_i}^+$ is the unital free semigroup on $n_i$ generators
$g_{1}^i,\ldots, g_{n_i}^i$ and the identity $g_{0}^i$.  Define the following subsets of $\FF_{n_i}^+$:
\begin{equation*}
\begin{split}
J_1^i&:=\left\{(g_1^i)^{k_1},\ldots, (g_{d_1}^i)^{k_1}\right\},\\
J_p^1&:=\left\{(g_1^i)^{k_p-k_{p-1}}(g_{n_i}^1)^{k_{p-1}}, (g_2^i)^{k_p-k_{p-1}}(g_{n_i}^i)^{k_{p-1}},\ldots, (g_{d_p}^i)^{k_p-k_{p-1}}(g_{n_i}^i)^{k_{p-1}}\right\}, \quad p=2,3,\ldots, N.
\end{split}
\end{equation*}
It is clear  that
\begin{equation}
\label{Mi}
\cM_i(t):=\bigoplus_{\beta\in \cup_{p=1}^N J_p^i} F^2(H_{n_i})\otimes e_\beta^i
\end{equation}
is an  invariant  subspace of $F^2(H_{n_i})$.
 If $k_p\leq q_i<k_{p+1}$, then we have
\begin{equation*}
\begin{split}
\frac{\text{\rm trace\,}\left[P_{\cM_i(t)^\perp}P_{q_i}^{(i)}  \right]}{\text{\rm trace\,}\left[P_{q_i}^{(i)}\right]}&=1-\frac{\text{\rm trace\,}\left[P_{\cM_i(t)}P_{q_i}^{(i)}  \right]}{\text{\rm trace\,}\left[P_{q_i}^{(i)}\right]} \\
&=1-\frac{1}{\text{\rm trace\,}\left[P_{q_i}^{(i)} \right]}
\sum_{\beta_i\in\FF_{n_i}^+}
\left<P_{\cM_i(t)}P_{q_i}^{(i)}e^i_{\beta_i},
e^i_{\beta_i} \right>\\
&=
1-\sum_{\beta_i\in\FF_{n_i}^+}
\frac{\left<P_{q_i}^{(i)} P_{\cM_i(t)} P_{q_i}^{(i)} e_{\beta_i}^i, e_{\beta_i}^i\right>}{n_i^{q_i}}
=1-\frac{1}{n_i^{q_i}}\sum_{\beta_i\in \FF_{n_i}^+, |\beta_i|=q_i}\|P_{\cM_i(t)}e_{\beta_i}^i\|^2
\\
&=1-\frac{d_1 n_i^{q_i-k_1}+\cdots + d_p n_i^{q_i-k_p}}{n_i^{q_i}}
=1- \left(\frac{d_1}{n_i^{k_1}}+\cdots + \frac{d_p}{n_i^{k_p}}\right).
\end{split}
\end{equation*}
Hence and using  Theorem \ref{multiplicity}, we deduce that
\begin{equation*}
\begin{split}
\text{\rm curv}_i[P_{\cM_i(t)^\perp} {\bf S}_i|_{\cM_i(t)^\perp}]=1-\lim_{q_i\to\infty}
  \frac{\text{\rm trace\,}\left[P_{\cM_i(t)}P_{q_i}^{(i)}  \right]}{\text{\rm trace\,}\left[P_{q_i}^{(i)}\right]}
 = 1-\sum_{p=1}^N \frac{d_p}{n_1^{k_p}}=t.
\end{split}
\end{equation*}
Fix $t\in (0,1)$, assume that $n_1, n_2\geq 2$, and let $\omega\in (1-t, 1)$.
Set
$$\cN_1(\omega):=\cM_1(1-\omega)\quad \text{ and }\quad \cN_2(\omega,t):=\cM_2\left(1-\frac{1-t}{\omega}\right)
$$ and  define
$$
\cM^{(\omega)}:=\cN_1(\omega)\otimes \cN_2(\omega, t)\otimes F^2(H_{n_3})\otimes \cdots \otimes F^2(H_{n_k}),
$$
which is an invariant subspace under ${\bf S}_{i,j}$ for any $i\in\{1,\ldots, k\}$ and $j\in \{1,\ldots, n_i\}$.
Defining
$T^{(\omega)}(t):=P_{{\cM^{(\omega)}}^\perp} {\bf S}|_{{\cM^{(\omega)}}^\perp}$
and using Corollary \ref{for-perp}, we deduce that
\begin{equation*}
\begin{split}
\text{\rm curv}\, [T^{(\omega)}(t)]&=1-\left[1-   \text{\rm curv}_1\left(P_{\cN_1(\omega)^\perp} {\bf S}_1|_{\cN_1(\omega)^\perp}\right)\right]          \left[1-\text{\rm curv}_2\left(P_{\cN_2(\omega, t)^\perp} {\bf S}_2|_{\cN_2(\omega, t)^\perp}\right)\right]\\
&=1-[1-(1-\omega)]\left[1-\left(1-\frac{1-t}{\omega}\right)\right]=t.
\end{split}
\end{equation*}
Since  the curvature is a unitary invariant and $\text{\rm curv}_2 [P_{\cN_2(\omega, t)^\perp} {\bf S}_2|_{\cN_2(\omega, t)^\perp}]=1-\frac{1-t}{\omega}$, we deduce that $P_{\cN_2(\omega, t)^\perp} {\bf S}_2|_{\cN_2(\omega, t)^\perp}$ is not unitary equivalent to
$P_{\cN_2(\sigma, t)^\perp} {\bf S}_2|_{\cN_2(\sigma, t)^\perp}$ if $\omega, \sigma\in (1-t,1)$ and $\omega\neq \sigma$.
Due to Theorem \ref{quotients-inv}, when $\cE=\CC$, we deduce that
$\cN_2(\omega, t)\neq \cN_2(\sigma, t)$, which  implies  $\cM^{(\omega)}(t)\neq \cM^{(\sigma)}(t)$. Using again
Theorem \ref{quotients-inv}, we conclude that $T^{(\omega)}(t)$ in not unitarily equivalent to $T^{(\sigma)}(t)$ for any $\omega, \sigma\in (1-t, 1)$, $\omega\neq \sigma$. The fact that $\rank [T^{(\omega)}(t)]=1$ is obvious.
The proof is complete.
\end{proof}

 We remark that Theorem \ref{value} shows, in particular,  that  the curvature is not a complete invariant.

\begin{corollary} Let ${\bf n}=(n_1,\ldots, n_k)\in \NN^k$ be such that $n_i\geq 2$ and $n_j\geq 2$ for some $i,j\in \{1,\ldots, k\}$, $i\neq j$. For each $s\in (0,1)$, there is an uncountable family of distinct invariant subspaces $\{\cN^{(\omega)}\}_{\omega\in \Omega}$ of $F^2(H_{n_1})\otimes \cdots \otimes F^2(H_{n_k})$  such that
$
m(\cN^{(\omega)})=s$ for $ \omega\in \Omega.
$
\end{corollary}
\begin{proof}   A closer look at the proof of Theorem \ref{value} and using  Corollary \ref{for-perp},  reveals that
  $m_i(\cM_i(t))=1-t$ for any  $t\in[0,1)$,  where $\cM_i(t)$ is defined by  relation \eqref{Mi}.
Fix $s\in (0,1)$, assume that $n_1, n_2\geq 2$, and let $\omega\in (s, 1)$. We define the subspace
$$
\cN^{(\omega)}=\cM_1(1-\omega)\otimes \cM_2\left(1-\frac{s}{\omega}\right)\otimes F^2(H_{n_3})\cdots \otimes F^2(H_{n_k}),
  $$
  where $\cM_1(1-\omega)\subset F^2(H_{n_1})$ and  $\cM_2\left(1-\frac{s}{\omega}\right)\subset F^2(H_{n_2})$  are invariant subspaces defined by relation \eqref{Mi}.
   Using  again  Corollary \ref{for-perp}, we obtain $m(\cN^{(\omega)})=s$ for all $\omega\in (s,1)$. As in the proof of Theorem \ref{value}, one can show that $\{\cN^{(\omega)}\}_{\omega\in (s,1)}$ are distinct  invariant subspaces.
\end{proof}

We remark that in the particular case when $n_1=\cdots=n_k=1$, Fang proved  in \cite{Fang} that  the multiplicity of any invariant subspace of $H(\DD^k)$ is  always a non-negative integer. When at least one $n_i$ is $\geq 2$, we have the following result.

\begin{corollary} \label{range} Let ${\bf n}=(n_1,\ldots, n_k)\in \NN^k$ be such that  there exists at least one number $n_i\geq2$ and let  $t\in [0,1]$. Then  there exists a pure element  ${\bf T}$  in the polyball ${\bf B_n}(\cH)$ such that $\rank ({\bf T})=1$ and
$$
\text{\rm curv}({\bf T})=t.
$$
Moreover, for each $s\in [0,1]$ there is an invariant subspace $\cN$ of  $ F^2(H_{n_1})\otimes\cdots \otimes F^2(H_{n_k})$  with multiplicity $$m(\cN)=s.$$
\end{corollary}
\begin{proof}
If $t=1$, we have $\text{\rm curv}({\bf S})=1$. We can assume that $n_1\geq 2$. If $t\in [0,1)$, consider the
subspace $\cM(t):=\cM_1(t)\otimes F^2(H_{n_2})\otimes \cdots \otimes F^2(H_{n_k})$, where $\cM_1(t)$ is defined by relation \eqref{Mi}.
As in the proof of Theorem \ref{value}, one can show that
$\text{\rm curv}(P_{\cM(t)^\perp} {\bf S}|_{\cM(t)^\perp})=t$. This implies
$m(\cM(t))=1-t$ and completes the proof.
\end{proof}

Using Corollary \ref{range} and tensoring  with the identity on $\CC^m$, one can see that for any $t\in [0,m]$, there exists a pure element ${\bf T}$  in the polyball with
$\rank ({\bf T})=m$ and $\text{\rm curv}({\bf T})=t$.

We remark that due to Theorem \ref{multiplicity},  if  $\cM$ is a proper invariant subspace of $\otimes_{i=1}^k F^2(H_{n_i})$ of finite codimension, then $$\text{\rm curv}(P_{\cM^\perp} {\bf S}|_{\cM^\perp})=0\quad \text{ and } \quad m(\cM)=1.$$
However, we have the following result.
\begin{proposition} \label{cur0} If $n_i\geq 2$ for at least one $i\in \{1,\ldots, k\}$, then there exist invariant subspaces $\cM$ of $\otimes_{i=1}^k F^2(H_{n_i})$ of infinite codimension such that
$
\text{\rm curv}(P_{\cM^\perp} {\bf S}|_{\cM^\perp})=0.
$
\end{proposition}
\begin{proof} We can assume that $n_1\geq 2$. Let $\cM_1$  be the  invariant subspace of $F^2(H_{n_1})$ such that  $\cM_1^\perp$ is the closed span of the vectors $(e^1_1)^s$, where $s=0,1,\ldots.$ Note that
\begin{equation*}
\begin{split}
\text{\rm curv}_1(P_{\cM_1^\perp} {\bf S}|_{\cM_1^\perp})
&=1-\lim_{q_1\to\infty}\frac{1}{n_1^{q_1}}\sum_{\beta\in \FF_{n_1}^+, |\beta|=q_1}\|P^{(1)}_{\cM_1} e^1_\beta\|^2=1-\lim_{q_1\to\infty}\frac{n_1^{q_1}-1}{n_1^{q_1}}=0.
\end{split}
\end{equation*}
For each $i\in \{2,\ldots, k\}$, let
$\cM_i$ be a proper invariant subspace of $\otimes_{i=1}^k F^2(H_{n_i})$ of finite codimension.  Then $\text{\rm curv}(P_{\cM_i^\perp} {\bf S}|_{\cM_i^\perp})=0$ for $i\in \{2,\ldots, k\}$ and, using Corollary \ref{for-perp}, we deduce that $\text{\rm curv}(P_{\cM^\perp} {\bf S}|_{\cM^\perp})=0$, where $\cM=\otimes_{i=1}^k \cM_i$.
The proof is complete.
\end{proof}

We remark that if ${\bf T}$  is an  element of  the regular polyball  ${\bf B_n}(\cH)$ and $\cM$ is an invariant subspace under ${\bf T}$,  then ${\bf T}|_\cM$ is not necessarily in the polyball. On the other hand,
if  ${\bf T}$  has finite rank, ${\bf T}|_\cM$ could have infinite rank.
Indeed,  if at least one $n_i\geq 2$, then  there are invariant subspaces  $\cM$ of $\otimes_{i=1}^k F^2(H_{n_i})$ of infinite codimension such that such that $\rank ({\bf S}|_\cM)=\infty$. To see this, assume that $n_1\geq 2$,  take $\cM_1:=\oplus_{p=1}^\infty\left[F^2(H_{n_1})\otimes e^1_{g_2}\otimes (e_{g_1}^1)^p\right]$
and $\cM:=\cM_1\otimes F^2(H_{n_2})\otimes\cdots \otimes F^2(H_{n_k})$.
We also mention that if $n_i\geq 2$ for some $i\in \{1,\ldots, k\}$, then there are  Beurling type invariant subspaces $\cM$  of $\otimes_{i=1}^k F^2(H_{n_i})$ of infinite codimension with
$\rank ({\bf S}|_\cM)= \text{\rm curv}({\bf S}|_\cM)=1$. Indeed, assume that $n_1\geq 2$ and take $\cM:=[F^2(H_{n_1})\otimes e_1^1]\otimes F^2(H_{n_2})\otimes \cdots \otimes F^2(H_{n_k})$.

\begin{lemma}\label{sub-poly} If ${\bf T}\in {\bf B_n}(\cH)$  and $\cM$ is an invariant subspace under ${\bf T}$, then ${\bf T}|_\cM\in {\bf B_n}(\cM)$   if and only if
$$\left(id -\Phi_{T_1}\right)^{p_1}\circ \cdots \circ\left(id -\Phi_{ T_k}\right)^{p_k}(P_\cM)\geq 0$$
for any $p_i\in\{0,1\}$.  If, in addition, ${\bf T}$ is pure, then the condition above is equivalent to
$$\left(id -\Phi_{T_1}\right)\circ \cdots \circ\left(id -\Phi_{ T_k}\right)(P_\cM)\geq 0.$$
In particular, if $\cM\subset \otimes_{k=1}^k F^2(H_{n_i})\otimes \cK$ is an invariant subspace under the universal model ${\bf S}\otimes I$, then  $({\bf S}\otimes I)|_\cM$ is in  the polyball if and only if $\cM$ is a Beurling type  invariant subspace
for  ${\bf S}\otimes I$.
\end{lemma}
\begin{proof}
Since $\cM$ is an invariant subspace under ${\bf T}$, we have
$$
\left(id -\Phi_{T_1|_\cM}\right)^{p_1}\circ \cdots \circ\left(id -\Phi_{ T_k|_\cM}\right)^{p_k}(I_\cM)=[\left(id -\Phi_{T_1}\right)^{p_1}\circ \cdots \circ\left(id -\Phi_{ T_k}\right)^{p_k}(P_\cM)]|_\cM.
$$
Since the direct implication is obvious, we prove the converse.  Assume that
$$
\left(id -\Phi_{T_1|_\cM}\right)^{p_1}\circ \cdots \circ\left(id -\Phi_{ T_k|_\cM}\right)^{p_k}(I_\cM)\geq 0.
$$
Let $h\in \cH$ and consider the orthogonal decomposition $h=x+y$ with $x\in \cM$ and $y\in \cM^\perp$. Using the fact that $\cM^\perp$ is invariant subspace under each operator $T_{i,j}^*$, we have
\begin{equation*}
\begin{split}
&\left<\left(id -\Phi_{T_1}\right)^{p_1}\circ \cdots \circ\left(id -\Phi_{ T_k}\right)^{p_k}(P_\cM)(x+y), x+y\right>\\
&=
\left<\left(id -\Phi_{T_1}\right)^{p_1}\circ \cdots \circ\left(id -\Phi_{ T_k}\right)^{p_k}(P_\cM)x, x+y\right>
+ \left<\left(id -\Phi_{T_1}\right)^{p_1}\circ \cdots \circ\left(id -\Phi_{ T_k}\right)^{p_k}(P_\cM)y, x+y\right>\\
&=\left<\left(id -\Phi_{T_1}\right)^{p_1}\circ \cdots \circ\left(id -\Phi_{ T_k}\right)^{p_k}(P_\cM)x, x\right> +\|y\|^2\geq 0.
\end{split}
\end{equation*}
Consequently, $\left(id -\Phi_{T_1}\right)^{p_1}\circ \cdots \circ\left(id -\Phi_{ T_k}\right)^{p_k}(P_\cM)\geq 0$. If, in addition, ${\bf T}$ is pure and
$\left(id -\Phi_{T_1}\right)\circ \cdots \circ\left(id -\Phi_{ T_k}\right)(P_\cM)\geq 0,$
then, since $\Phi_{T_1}$ is positive linear map, we deduce that
$$\Phi_{T_1}^m \left(id -\Phi_{T_2}\right)\circ \cdots \circ\left(id -\Phi_{ T_k}\right)(P_\cM)\leq \left(id -\Phi_{T_2}\right)\circ \cdots \circ\left(id -\Phi_{ T_k}\right)(P_\cM)$$
for any $m\in \NN$. Taking into account that $\Phi_{T_1}^m (I)\to 0$ as $m\to \infty$, we obtain $\left(id -\Phi_{T_2}\right)\circ \cdots \circ\left(id -\Phi_{ T_k}\right)(P_\cM)\geq 0$. Similarly, one can show that
$\left(id -\Phi_{T_1}\right)^{p_1}\circ \cdots \circ\left(id -\Phi_{ T_k}\right)^{p_k}(P_\cM)\geq 0$
for any $p_i\in\{0,1\}$. Now, using  Theorem 5.2 from \cite{Po-Berezin-poly}, the last part of this lemma follows.
\end{proof}

\begin{proposition} Let $\cM$ be a proper invariant subspace of $\otimes_{i=1}^k F^2(H_{n_i})$.  Then  ${\bf S}|_\cM \in {\bf B_n}(\cM)$ and
$\text{\rm curv}(P_{\cM^\perp} {\bf S}|_{\cM^\perp})=0$
 if and only if there is an inner sequence $\{\psi_s\}_{s=1}^\infty$ for $\cM$, i.e.,
$\psi_s$ are multipliers of $\otimes_{i=1}^k F^2(H_{n_i})$,
$ P_\cM=\sum_{s=1}^\infty \psi_s\psi_s^*
$
where the convergence is in the  strong operator topology, and
$$
\lim_{(q_1,\ldots, q_k)\in \ZZ_+^k} \frac{1}{n_1^{q_1}\cdots n_k^{q_k}}
\sum_{|\alpha_i|=q_i, i\in \{1,\ldots,k\}}\sum_{s=1}^\infty \|\psi_s(e^1_{\alpha_1}\otimes \cdots \otimes  e^k_{\alpha_k}\|^2=1.
$$
\end{proposition}
\begin{proof}  According to Lemma \ref{sub-poly},  condition   ${\bf S}|_\cM \in {\bf B_n}(\cM)$  holds if and only if  $\cM$ is a Beurling type invariant subspace. Therefore, there exist
multi-analytic operators $\psi_p: \otimes_{i=1}^k F^2(H_{n_i})\to \otimes_{i=1}^k F^2(H_{n_i})$  such that $P_\cM=\sum_{p=1}^\infty \psi_p\psi_p^*$.
Due to Theorem \ref{multiplicity},
condition $\text{\rm curv}(P_{\cM^\perp} {\bf S}|_{\cM^\perp})=0$ is equivalent to
\begin{equation*}
\begin{split} \lim_{(q_1,\ldots, q_k)\in \ZZ_+^k}
 \frac{\text{\rm trace\,}\left[\left(\sum_{p=1}^\infty \psi_p\psi_p^*\right)(P_{q_1}^{(1)}\otimes \cdots \otimes P_{q_k}^{(k)}) \right]}{\text{\rm trace\,}\left[P_{q_1}^{(1)}\otimes \cdots \otimes P_{q_k}^{(k)}\right]}=1.
 \end{split}
 \end{equation*}
Now, one can easily complete the proof.
\end{proof}

\begin{lemma}\label{coincidence} Let $\cM$ be  an invariant subspace of
$\otimes_{i=1}^k F^2(H_{n_i})\otimes \cE$ such that it does not contain nontrivial reducing subspaces for the universal model ${\bf S}\otimes I_\cE$, and let
 ${\bf T}:=P_{\cM^\perp} ({\bf S}\otimes I_\cE)|_{\cM^\perp}$. Then there is a unitary operator $Z:\overline{{\bf \Delta_T}(I)(\cH)}\to \cE$ such that
 $(I\otimes Z){\bf K_T}=V$, where ${\bf K_T}$ is the Berezin kernel associated with ${\bf T}$ and  $V$ is the injection of $\cM^\perp$ into
 $\otimes_{i=1}^k F^2(H_{n_i})\otimes \cE$.

 Moreover, ${\bf T}$ has characteristic function if and only if $\cM$ is a Beurling type invariant subspace.
\end{lemma}
\begin{proof} Note that $\cM^\perp$ is a cyclic subspace  for ${\bf S}\otimes I_\cE$.
Therefore, ${\bf S}\otimes I_\cE$ is a minimal isometric dilation of ${\bf T}:=P_{\cM^\perp} ({\bf S}\otimes I_\cE)|_{\cM^\perp}$.
On the other hand,
according to  \cite{Po-Berezin-poly} (see Theorem 5.6 and its proof) if
  $${\bf K_{T}}: \cH \to F^2(H_{n_1})\otimes \cdots \otimes  F^2(H_{n_k}) \otimes  \overline{{\bf \Delta_{T}}(I)(\cH)}$$
  is the noncommutative Berezin kernel,
 then the  subspace ${\bf K_{f,T}}\cH$ is   co-invariant  under  each operator
${\bf S}_{i,j}\otimes  I_{\overline{{\bf \Delta_{T}}\cH}}$ for   any $i\in \{1,\ldots, k\}$,  $  j\in \{1,\ldots, n_i\}$ and
    the dilation
  provided by   the relation
    $$ { \bf T}_{(\alpha)}= {\bf K_{T}^*}({\bf S}_{(\alpha)}\otimes  I_{\overline{{\bf\Delta_{T}}(I)(\cH)}})  {\bf K_{T}}, \qquad (\alpha) \in \FF_{n_1}^+\times \cdots \times \FF_{n_k}^+,
    $$
  is minimal.
Moreover, since the Cuntz-Toeplitz algebra \cite{Cu} satisfies the relation
$$C^*({\bf S}_{i,j})=\overline{\text{\rm span}}\{{\bf S}_{(\alpha)}{\bf S}_{(\beta)}^*: \ (\alpha), (\beta)\in \FF_{n_1}^+\times \cdots \times \FF_{n_k}^+\},
$$
the two dilations mentioned  above coincide   up to an isomorphism. Consequently,  there is a unitary operator $Z:\overline{{\bf \Delta_T}(I)(\cH)}\to \cE$ such that
 $(I\otimes Z){\bf K_T}=V$, where ${\bf K_T}$ is the Berezin kernel associated with ${\bf T}$ and  $V$ is the injection of $\cM^\perp$ into
 $\otimes_{i=1}^k F^2(H_{n_i})\otimes \cE$.
 Note that ${\bf K_{T}} {\bf K_{T}^*}=(I\otimes Z^*)P_{\cM^\perp} (I\otimes Z)$ and
 $${\bf \Delta_{S\otimes {\it I_\cE}}}(I-{\bf K_{T}}{\bf K_{T}^*})
 =(I\otimes Z^*){\bf \Delta_{S\otimes {\it I_\cE}}}(P_\cM) (I\otimes Z).
 $$
  As in the proof of Lemma \ref{sub-poly},  $\cM$ is a Beurling type invariant subspace if and only if ${\bf \Delta_{S\otimes {\it I_\cE}}}(P_\cM)\geq 0$. Using the identity above, we deduce that ${\bf \Delta_{S}}(I-{\bf K_{T}}{\bf K_{T}^*})\geq 0$, which, due to Theorem 6.2 from \cite{Po-Berezin-poly}, is equivalent to ${\bf T}$ having characteristic function. The proof is complete.
\end{proof}

\begin{proposition}\label{strict}  The following statements hold.
\begin{enumerate}
\item[(i)] If $\cM$ is a proper Beurling type invariant subspace of
$\otimes_{i=1}^k F^2(H_{n_i})$, then
$$
0\leq \text{\rm curv}(P_{\cM^\perp} {\bf S}|_{\cM^\perp})<1\quad \text{ and }\quad 0<m(\cM)\leq 1.
$$
\item[(ii)]
  If $\cM_i\neq \{0\}$, $i\in \{1,\ldots, k\}$,  is an invariant subspace  of $F^2(H_{n_i})$, then $\cM:=\otimes_{i=1}^k \cM_i$ is an invariant subspace  of $\otimes_{i=1}^k F^2(H_{n_i})$  and  $\text{\rm curv}(P_{\cM^\perp} {\bf S}|_{\cM^\perp})<1$.
 \end{enumerate}
\end{proposition}
\begin{proof}  To prove item (i), note that  Lemma \ref{coincidence} shows  that there is a unitary operator $Z:\overline{{\bf \Delta_T}(I)(\cM^\perp)}\to \CC$ such that
 $(I\otimes Z){\bf K_T}=V$, where ${\bf K_T}$ is the Berezin kernel associated with ${\bf T}:=P_{\cM^\perp} {\bf S}|_{\cM^\perp}$ and  $V$ is the injection of $\cM^\perp$ into
 $\otimes_{i=1}^k F^2(H_{n_i})$. Hence, we deduce that
 ${\bf K_{T}} {\bf K_{T}^*}=(I\otimes Z^*)P_{\cM^\perp} (I\otimes Z)$ and
 $${\bf \Delta_{S\otimes {\it I_\CC}}}(I-{\bf K_{T}}{\bf K_{T}^*})
 =(I\otimes Z^*){\bf \Delta_{S}}(P_\cM) (I\otimes Z).
 $$
   Since $\cM$ is a Beurling type invariant subspace, we have   ${\bf \Delta_{S}}(P_\cM)\geq 0$, which  implies ${\bf \Delta_{S}}(I-{\bf K_{T}}{\bf K_{T}^*})\geq 0$.
 If we assume that $\text{\rm curv}({\bf T})=1$, then $\text{\rm curv}({\bf T})=\rank({\bf T})$ and, due to
 Theorem \ref{comp-inv}, ${\bf T}$ is unitarily equivalent to  ${\bf S}$. Consequently, $\cM^\perp$ is an invariant subspace for ${\bf S}_{i,j}$ and, therefore, reducing for ${\bf S}_{i,j}$. Since the $C^*$-algebra $C^*({\bf S}_{i,j})$ is irreducible, we get a contradiction.
Now, we prove part (ii).  Since  any  invariant subspace  of $F^2(H_{n_i})$  is of Beurling type (see \cite{Po-charact}), we can apply  part (i), when $k=1$,  and deduce that $\text{\rm curv}_i(P_{\cM_i^\perp} {\bf S}_i|_{\cM_i^\perp})<1$.
Now, using Proposition \ref{for-perp}, one can  complete the proof.
\end{proof}

 We remark that Proposition \ref{strict} implies that $P_{\cM^\perp} {\bf S}|_{\cM^\perp}$ is not unitarily equivalent to the universal model ${\bf S}$. It remains an open question whether Proposition \ref{strict} part (i) is true for arbitrary  proper invariant subspace of
$\otimes_{i=1}^k F^2(H_{n_i})$, which is the case when $k=1$ (see \cite{Po-curvature}).

\section{Stability, continuity, and multiplicative   properties}

 In this section, we provide several properties concerning the stability, continuity, and multiplicative   properties for the curvature and the multiplicity invariants.

\begin{theorem} Let  ${\bf T}\in {\bf B_n}(\cH)$  have finite rank and let $\cM$ be an invariant subspace under ${\bf T}$ such that  ${\bf T}|_\cM\in {\bf B_n}(\cM)$ and $\dim \cM^\perp<\infty$. Then ${\bf T}|_\cM$ has finite rank and
$$
\left|\text{\rm curv}({\bf T})-\text{\rm curv}({\bf T}|_\cM)\right|
\leq \dim \cM^\perp\prod_{i=1}^k (n_i-1).
$$
\end{theorem}
\begin{proof} Note that $\rank ({\bf T}|_\cM)=\rank {\bf \Delta_T}(P_\cM)$. Since
${\bf \Delta_T}(P_\cM)={\bf \Delta_T}(I_\cH)-{\bf \Delta_T}(P_{\cM^\perp})$, we deduce that $\rank ({\bf T}|_\cM)<\infty$.
Since $\cM$ is an invariant subspace under ${\bf T}$, we have
\begin{equation*}
\begin{split}
\text{\rm trace}&\left[\left(id -\Phi_{T_1|_\cM}^{q_1+1}\right)\circ \cdots \circ\left(id -\Phi_{ T_k|_\cM}^{q_k+1}\right)(I_\cM)\right]\\
&= \text{\rm trace}\left[\left(id -\Phi_{T_1}^{q_1+1}\right)\circ \cdots \circ\left(id -\Phi_{ T_k}^{q_k+1}\right)(P_\cM)\right]\\
&\leq \text{\rm trace}\left[\left(id -\Phi_{T_1}^{q_1+1}\right)\circ \cdots \circ\left(id -\Phi_{ T_k}^{q_k+1}\right)(I_\cH)\right]+\text{\rm trace}\left[\left(id -\Phi_{T_1}^{q_1+1}\right)\circ \cdots \circ\left(id -\Phi_{ T_k}^{q_k+1}\right)(P_{\cM^\perp})\right].
\end{split}
\end{equation*}
Note also that
\begin{equation*}
\begin{split}
\text{\rm trace}&\left[\left(id -\Phi_{T_1}^{q_1+1}\right)\circ \cdots \circ\left(id -\Phi_{ T_k}^{q_k+1}\right)(P_{\cM^\perp})\right]\leq (1+n_1^{q_1+1})\cdots (1+n_k^{q_k+1})
\text{\rm trace}\,[P_{\cM^\perp}].
\end{split}
\end{equation*}
Consequently, we have
\begin{equation*}
\begin{split}
&\left|\frac{\text{\rm trace}\left[\left(id -\Phi_{T_1}^{q_1+1}\right)\circ \cdots \circ\left(id -\Phi_{ T_k}^{q_k+1}\right)(I_\cH)\right]}{{\prod_{i=1}^k(1+n_i+\cdots + n_i^{q_i})}}-\frac{\text{\rm trace}\left[\left(id -\Phi_{T_1|_\cM}^{q_1+1}\right)\circ \cdots \circ\left(id -\Phi_{ T_k|_\cM}^{q_k+1}\right)(I_\cM)\right]}{{\prod_{i=1}^k(1+n_i+\cdots + n_i^{q_i})}}
 \right|\\
 &\qquad\qquad\leq \frac{(1+n_1^{q_1+1})\cdots (1+n_k^{q_k+1})}{{\prod_{i=1}^k(1+n_i+\cdots + n_i^{q_i})}}\text{\rm trace}\,[P_{\cM^\perp}].
\end{split}
\end{equation*}
Using Corollary \ref{curva-maps}, we deduce that $\text{\rm curv}({\bf T})=\text{\rm curv}({\bf T}|_\cM)$ if $n_i=1$ for  at least one $i\in\{1,\ldots, k\}$.
When $n_i\geq 2$, we obtain the desired inequality.
The proof is complete.
\end{proof}

\begin{theorem} Let  ${\bf T}\in {\bf B_n}(\cH)$  have finite rank and let $\cM$ be a co-invariant subspace under ${\bf T}$   with $\dim \cM^\perp<\infty$. Then $P_\cM{\bf T}|_\cM$ has finite rank and
$$
\text{\rm curv}({\bf T})=\text{\rm curv}(P_\cM{\bf T}|_\cM).
$$
\end{theorem}
\begin{proof} Set $A:=(A_1,\ldots, A_k)$ and  $A_i:=(A_{i,1},\ldots, A_{i,n_i})$, where $A_{i,j}:=P_\cM T_{i,j}|_\cM$ for $i\in \{1,\ldots, k\}$ and $j\in \{1,\ldots, n_i\}$.
Since $\Phi_{A_i}^{q_i}(I_{\cM})=P_{\cM} \Phi_{T_i}^{q_i}(I_\cH)|_\cM$, we have
\begin{equation*}
\begin{split}
&\text{\rm trace}\left[\left(id -\Phi_{A_1}^{q_1+1}\right)\circ \cdots \circ\left(id -\Phi_{ A_k}^{q_k+1}\right)(I_\cM)\right]\\
&\leq \text{\rm trace}\left[P_\cM \left(id -\Phi_{T_1}^{q_1+1}\right)\circ \cdots \circ\left(id -\Phi_{ T_k}^{q_k+1}\right)(I_\cH)\right]+
\text{\rm trace}\left[P_{\cM^\perp}\left(id -\Phi_{T_1}^{q_1+1}\right)\circ \cdots \circ\left(id -\Phi_{ T_k}^{q_k+1}\right)(I_\cH)\right]\\
&\leq \text{\rm trace}\left[ \left(id -\Phi_{A_1}^{q_1+1}\right)\circ \cdots \circ\left(id -\Phi_{ A_k}^{q_k+1}\right)(I_\cM)\right]+
\text{\rm trace}\left[P_{\cM^\perp}\left(id -\Phi_{T_1}^{q_1+1}\right)\circ \cdots \circ\left(id -\Phi_{ T_k}^{q_k+1}\right)(I_\cH)\right].
\end{split}
\end{equation*}
On the other hand, since  $\left(id -\Phi_{T_1}^{q_1+1}\right)\circ \cdots \circ\left(id -\Phi_{ T_k}^{q_k+1}\right)(I_\cH)\leq I$, we deduce that
$$
\text{\rm trace}\left[P_{\cM^\perp}\left(id -\Phi_{T_1}^{q_1+1}\right)\circ \cdots \circ\left(id -\Phi_{ T_k}^{q_k+1}\right)(I_\cH)\right]
\leq \dim\cM^\perp.
$$
Consequently, we have
\begin{equation*}
\begin{split}
&\left|\frac{\text{\rm trace}\left[\left(id -\Phi_{T_1}^{q_1+1}\right)\circ \cdots \circ\left(id -\Phi_{ T_k}^{q_k+1}\right)(I_\cH)\right]}{{\prod_{i=1}^k(1+n_i+\cdots + n_i^{q_i})}}-\frac{\text{\rm trace}\left[\left(id -\Phi_{A_1}^{q_1+1}\right)\circ \cdots \circ\left(id -\Phi_{ A_k}^{q_k+1}\right)(I_\cM)\right]}{{\prod_{i=1}^k(1+n_i+\cdots + n_i^{q_i})}}
 \right|\\
 &\qquad\qquad\leq \frac{ 1}{{\prod_{i=1}^k(1+n_i+\cdots + n_i^{q_i})}}\dim\cM^\perp.
\end{split}
\end{equation*}
Using Corollary \ref{curva-maps}, we deduce that $\text{\rm curv}({\bf T})=\text{\rm curv}({\bf A})$.
The proof is complete.
\end{proof}

We denote by $\cA$ the set of all $k$-tuples $\phi=(\phi_1,\ldots, \phi_k)$ of commuting positive linear maps on $B(\cH)$  such that $\phi_i(\cT^+(\cH))\subset \cT^+(\cH)$ and
$
\text{\rm trace\,} [\phi_i(X)]\leq \text{\rm trace\,} (X)
$
for any $X\in \cT^+(\cH)$ and $i\in \{1,\ldots, k\}$.

\begin{lemma} \label{semi-cont}
Let $\phi=(\phi_1,\ldots, \phi_k)$ and  $\phi_{(m)}=(\phi_{1,m},\ldots, \phi_{k,m})$, $m\in \NN$, be in $\cA$ and let $X$ and $X_m$ be in $ \cT^+(\cH)$.
If
\begin{equation*}
   \lim_{m\to\infty}
\text{\rm trace\,}\left[ \phi_{1,m}^{q_1}\circ \cdots \circ \phi_{k,m}^{q_k}(X_m)\right]= \text{\rm trace\,}\left[ \phi_{1}^{q_1}\circ \cdots \circ \phi_{k}^{q_k}(X)\right]
\end{equation*}
for each ${\bf q}=(q_1,\ldots, q_k)\in \ZZ_+^k$, then
$$
\limsup_{m\to \infty}\left\{\lim_{{\bf q}\in \ZZ_+^k}\text{\rm trace\,}\left[ \phi_{1,m}^{q_1}\circ \cdots \circ \phi_{k,m}^{q_k}(X_m)\right]\right\}\leq \lim_{{\bf q}\in \ZZ_+^k}\text{\rm trace\,}\left[ \phi_{1}^{q_1}\circ \cdots \circ \phi_{k}^{q_k}(X)\right].
$$
\end{lemma}
\begin{proof}
For  each ${\bf q}=(q_1,\ldots, q_k)\in \ZZ_+^k$, $m\in \NN$, let
$$
x_{\bf q}:=\text{\rm trace\,}\left[ \phi_{1}^{q_1}\circ \cdots \circ \phi_{k}^{q_k}(X)\right]\quad \text{ and } \quad x_{\bf q}^{(m)}:=
\text{\rm trace\,}\left[ \phi_{1,m}^{q_1}\circ \cdots \circ \phi_{k,m}^{q_k}(X_m)\right].
$$
According to Lemma \ref{basic}, $\{x_{\bf q}\}_{{\bf q}\in \ZZ_+^k}$ and
$\{x_{\bf q}^{(m)}\}_{{\bf q}\in \ZZ_+^k}$ are decreasing multi-sequences with respect to each of the indices $q_1,\ldots, q_k$. Set $x:=\lim_{{\bf q}\in \ZZ_+^k} x_{\bf q}$ and $x^{(m)}:=\lim_{{\bf q}\in \ZZ_+^k} x_{\bf q}^{(m)}$.
We prove that $\limsup_{m\to\infty} x^{(m)}\leq x$, by contradiction. Passing to a subsequence we can assume that there is $\epsilon>0$ such that $x^{(m)}-x\geq 2\epsilon>0$ for any $m\in \NN$. Since  $x:=\lim_{{\bf q}\in \ZZ_+^k} x_{\bf q}$, we can choose $N\geq 1$ such that $|x_{\bf q}-x|<\epsilon$ for any ${\bf q}=(q_1,\ldots, q_k)\in \ZZ_+^k$ with $q_i\geq N$. Fix such a ${\bf q}$ and note that  $x_{\bf q}^{(m)}\geq x^{(m)}$ and
$$
|x_{\bf q}^{(m)}-x_{\bf q}|\geq
|(x_{\bf q}^{(m)}-x^{(m)})+(x^{(m)}-x)|-|x_{\bf q} -x |\geq 2\epsilon-\epsilon=\epsilon.
$$
This contradicts the hypothesis that $\lim_{m\to\infty} x_{\bf q}^{(m)}=x_{\bf q}$. The proof is complete.
\end{proof}

\begin{theorem}\label{KK} Let ${\bf T}$ and $\{{\bf T}^{(m)}\}_{m\in \NN}$ be elements in the polyball $ {\bf B_n}(\cH)$   such that they have finite ranks which are  uniformly bounded and
$\text{\rm WOT-}\lim_{m\to\infty}{\bf K}_{{\bf T}^{(m)}}{\bf K}^*_{{\bf T}^{(m)}}= {\bf K_T} {\bf K_T^*}.
$
Then $$
\limsup_{m\to\infty}\text{\rm curv}({\bf T}^{(m)})\leq \text{\rm curv}({\bf T}).
$$
\end{theorem}
\begin{proof} Due to the hypothesis, we can assume that
${\bf K}_{{\bf T}^{(m)}}$ and ${\bf K_T}$ are taking values in $\otimes_{i=1}^k F^2(H_{n_i})\otimes \cK$, where $\cK$ is a Hilbert space  with $\dim \cK<\infty$. Consequently, we have
$$
\lim_{m\to\infty}\text{\rm trace\,}\left[ (P_{q_1}^{(1)}\otimes \cdots \otimes P_{q_k}^{(k)}\otimes I_\cK){\bf K_{T^{(\it m)}}}{\bf K_{T^{(\it m)}}^*}\right]\\
=\text{\rm trace\,}\left[ (P_{q_1}^{(1)}\otimes \cdots \otimes P_{q_k}^{(k)}\otimes I_\cK){\bf K_{T }}{\bf K_{T }^*}\right]
$$
for any ${\bf q}=(q_1,\ldots, q_k)\in \ZZ_+^k$.
Due to Theorem \ref{curva1} and its proof, we have
\begin{equation*}
\begin{split}
\lim_{m\to\infty}\text{\rm trace\,}\left[\Phi_{T^{(m)}_1}^{q_1}\circ \cdots \circ \Phi_{T^{(m)}_k}^{q_k} ({\bf \Delta_{T^{(\it m)}}}(I))\right]
&=\lim_{m\to\infty}\text{\rm trace\,}\left[{\bf K_{T^{(\it m)}}^*} (P_{q_1}^{(1)}\otimes \cdots \otimes P_{q_k}^{(k)}\otimes I_\cK){\bf K_{T^{(\it m)}}}\right]\\
&=\text{\rm trace\,}\left[ (P_{q_1}^{(1)}\otimes \cdots \otimes P_{q_k}^{(k)}\otimes I_\cK){\bf K_{T }}{\bf K_{T }^*}\right]\\
&=\text{\rm trace\,}\left[\Phi_{T_1}^{q_1}\circ \cdots \circ \Phi_{T_k}^{q_k} ({\bf \Delta_{T }}(I))\right]
\end{split}
\end{equation*}
for any ${\bf q}=(q_1,\ldots, q_k)\in \ZZ_+^k$.
Applying Lemma \ref{semi-cont} when
\begin{equation*}
\label{partic}
\begin{split}
\phi&=\left(\frac{1}{n_1}\Phi_{T_1},\ldots, \frac{1}{n_k}\Phi_{T_k}\right),\quad X={\bf \Delta_T}(I),\\
\Phi_{(m)}&=\left(\frac{1}{n_1}\Phi_{T^{(m)}_1},\ldots, \frac{1}{n_k}\Phi_{T^{(m)}_k}\right),\quad X_m={\bf \Delta_{T^{(\it m)}}}(I),
\end{split}
\end{equation*}
we obtain
$$
\limsup_{m\to \infty}\left\{\lim_{{\bf q}\in \ZZ_+^k}\frac{\text{\rm trace\,}\left[\Phi_{T^{(m)}_1}^{q_1}\circ \cdots \circ \Phi_{T^{(m)}_k}^{q_k} ({\bf \Delta_{T^{(\it m)}}}(I))\right]}{n_1^{q_1}\cdots n_k^{q_k}}\right\}\leq \lim_{{\bf q}\in \ZZ_+^k}\frac{\text{\rm trace\,}\left[\Phi_{T_1}^{q_1}\circ \cdots \circ \Phi_{T_k}^{q_k} ({\bf \Delta_{T }}(I))\right]}{n_1^{q_1}\cdots n_k^{q_k}}.
$$
Using again Theorem \ref{curva1}, we conclude that
$\limsup_{m\to\infty}\text{\rm curv}({\bf T}^{(m)})\leq \text{\rm curv}({\bf T})$. The proof is complete.
\end{proof}

The next result shows that the multiplicity invariant is lower semi-continuous.

\begin{theorem} Let $\cM$ and $\cM_m$ be invariant subspaces of
$\otimes_{i=1}^k F^2(H_{n_i})\otimes \cE$ with $\dim \cE<\infty$.
If $\text{\rm WOT-}\lim_{m\to\infty}P_{\cM_m}=P_\cM$, then
$$
\liminf_{m\to\infty}m(\cM_m)\geq m(\cM).
$$
\end{theorem}
\begin{proof} Let  ${\bf M}:=(M_1,\ldots, M_k)$ with $M_i:=(M_{i,1},\ldots, M_{i,n_i})$ and $M_{i,j}:=P_{\cM^\perp}({\bf S}_{i,j}\otimes I_\cE)|_{\cM^\perp}$. Similarly, we define ${\bf M}^{(m)}:=(M_1^{(m)},\ldots, M_k^{(m)})$. Due to Theorem \ref{multiplicity}, we have
\begin{equation}
\label{mdc}
m(\cM)=\dim\cE- \text{\rm curv}({\bf M}) \quad \text{ and } \quad
m(\cM_m)=\dim\cE- \text{\rm curv}({\bf M}^{(m)}).
\end{equation}
As in the proof of the same theorem, we have
\begin{equation*}
\begin{split}
\text{\rm trace\,}[\Phi_{M_1}^{q_1}\circ \cdots \circ \Phi_{M_k}^{q_k} ({\bf \Delta_{M}}(I_{\cM^\perp}))]
&=\text{\rm trace\,}[P_{\cM^\perp}(P_{q_1}^{(1)}\otimes \cdots \otimes P_{q_k}^{(k)}\otimes I_\cE)]\\
&=n_1^{q_1}\cdots n_k^{q_k}\dim\cE- \text{\rm trace\,}[P_{\cM}(P_{q_1}^{(1)}\otimes \cdots \otimes P_{q_k}^{(k)}\otimes I_\cE)]
\end{split}
\end{equation*}
and a similar relation associated with ${\bf  M}^{(m)}$ holds.
Since $\text{\rm WOT-}\lim_{m\to\infty}P_{\cM_m}=P_\cM$, we deduce that
$$
\lim_{m\to\infty} \text{\rm trace\,}[P_{\cM_m}(P_{q_1}^{(1)}\otimes \cdots \otimes P_{q_k}^{(k)}\otimes I_\cE)]=\text{\rm trace\,}[P_{\cM}(P_{q_1}^{(1)}\otimes \cdots \otimes P_{q_k}^{(k)}\otimes I_\cE)].
$$
Consequently, we obtain
$$
\lim_{m\to\infty}\text{\rm trace\,}[\Phi_{M_1^{(m)}}^{q_1}\circ \cdots \circ \Phi_{M_k^{(m)}}^{q_k} ({\bf \Delta_{M^{(m)}}}(I_{\cM_m^\perp}))]=\text{\rm trace\,}[\Phi_{M_1}^{q_1}\circ \cdots \circ \Phi_{M_k}^{q_k} ({\bf \Delta_{M}}(I_{\cM^\perp}))]
$$
As in the proof of Theorem \ref{KK}, one can use  Lemma \ref{semi-cont} to show that
$\limsup_{m\to\infty}\text{\rm curv}({\bf M}^{(m)})\leq \text{\rm curv}({\bf M} )$. Now, relation \eqref{mdc} implies
$\liminf_{m\to\infty}m(\cM_m)\geq m(\cM)$. The proof is complete.
\end{proof}

The next result shows that the curvature is upper semi-continuous.
\begin{theorem} Let ${\bf T}$ and $\{{\bf T}^{(m)}\}_{m\in \NN}$ be elements in the polyball $ {\bf B_n}(\cH)$   such that they have finite ranks which are  uniformly bounded. If ${\bf T}^{(m)}\to {\bf T}$, as $m\to \infty$, in the norm topology,
 then $$
\limsup_{m\to\infty}\text{\rm curv}({\bf T}^{(m)})\leq \text{\rm curv}({\bf T}).
$$
\end{theorem}
\begin{proof}
As in the proof of Theorem \ref{KK}, due to Lemma \ref{semi-cont}, it is enough to show that
$
\lim_{m\to\infty}\text{\rm trace\,}(A_{\bf q}^{(m)})
=
\text{\rm trace\,}(A_{\bf q})
$
for each ${\bf q}=(q_1,\ldots, q_k)\in \ZZ_+^k$,
where
$
A_{\bf q}^{(m)}:=\Phi_{T^{(m)}_1}^{q_1}\circ \cdots \circ \Phi_{T^{(m)}_k}^{q_k} ({\bf \Delta_{T^{(\it m)}}}(I))$ and $ A_{\bf q}:=\Phi_{T_1}^{q_1}\circ \cdots \circ \Phi_{T_k}^{q_k} ({\bf \Delta_{T }}(I)).
$
Taking into account that $A_{\bf q}^{(m)}$ and $A_{\bf q}$ have finite rank, we have
\begin{equation}\label{tr-rank}
\begin{split}
\left|\text{\rm trace\,}[A_{\bf q}^{(m)}-A_{\bf q}]\right|
\leq \|A_{\bf q}^{(m)}-A_{\bf q}\|\rank[A_{\bf q}^{(m)}-A_{\bf q}].
\end{split}
\end{equation}
Since ${\bf T}^{(m)}\to {\bf T}$, as $m\to \infty$, in the norm topology, it is easy to see that $\|A_{\bf q}^{(m)}-A_{\bf q}\|\to 0$ as $m\to\infty$.
On the other hand,   note that there is $M>0$ such that $\rank [{\bf \Delta_{T^{(\it m)} }}(I)]\leq M$ for any $m\in \NN$ and
$$
\rank[A_{\bf q}^{(m)}-A_{\bf q}]\leq n_1^{q_1}\cdots n_k^{q_k} \left(\rank [{\bf \Delta_{T^{(\it m)}}}(I)]+\rank [{\bf \Delta_{T }}(I)]\right).
$$
Consequently, relation \eqref{tr-rank} implies
$\text{\rm trace\,}(A_{\bf q}^{(m)})\to \text{\rm trace\,} (A_{\bf q})$ as $m\to\infty$. The rest of the proof is similar to the one of Theorem
\ref{KK}. The proof is complete.
\end{proof}

Given a function $\kappa :\NN\to \NN$ and ${\bf n}^{(i)}:=(n_1^{(i)},\ldots, n_{\kappa(i)}^{(i)})\in \NN^{\kappa(i)}$ for  $i\in \{1,\ldots,p\}$,
we consider the polyball ${\bf B}_{{\bf n}^{(i)}}(\cH_i)$, where $\cH_i$ is a Hilbert space.
 Let ${\bf X}^{(i)}\in{\bf B}_{{\bf n}^{(i)}}(\cH_i)$ with ${\bf X}^{(i)}:=(X_1^{(i)},\ldots, X_{\kappa(i)}^{(i)})$ and
 $X_r^{(i)}:=(X_{r,1}^{(i)},\ldots, X^{(i)}_{r,n_r^{(i)}})\in B(\cH_i)^{n_r^{(i)}}$  for
 $r\in \{1,\ldots, \kappa(i)\}$.
 If ${\bf X}:=({\bf X}^{(1)},\ldots, {\bf X}^{(p)})\in {\bf B}_{{\bf n}^{(1)}}(\cH_1)\times \cdots \times {\bf B}_{{\bf n}^{(p)}}(\cH_p)$, we define the ampliation
 $\widetilde {\bf X}$ by setting
 $\widetilde {\bf X}:=(\widetilde {\bf X}^{(1)},\ldots, \widetilde {\bf X}^{(p)})$, where
 $\widetilde {\bf X}^{(i)}:=(\widetilde X_1^{(i)},\ldots, \widetilde X_{\kappa(i)}^{(i)})$ and
 $\widetilde X_r^{(i)}:=(\widetilde X_{r,1}^{(i)},\ldots, \widetilde X^{(i)}_{r,n_r^{(i)}})$  for
 $r\in \{1,\ldots, \kappa(i)\}$,  and
 $$
 \widetilde X_{r,s}^{(i)}:=I_{\cH_1}\otimes \cdots \otimes I_{\cH_{i-1}}\otimes X_{r,s}^{(i)}\otimes I_{\cH_{i+1}}\otimes I_{\cH_p}
 $$
 for all $i\in \{1,\ldots,p\}$, $r\in \{1,\ldots, \kappa(i)\}$, and $s\in \{1,\ldots, n_r^{(i)}\}$.

 \begin{theorem} \label{multiplicative-prop} Let  ${\bf X}:=({\bf X}^{(1)},\ldots, {\bf X}^{(p)})\in {\bf B}_{{\bf n}^{(1)}}(\cH_1)\times \cdots \times {\bf B}_{{\bf n}^{(p)}}(\cH_p)$ be such that each ${\bf X}^{(i)}$ has trace class defect. Then the ampliation $\widetilde {\bf X}$ is in the regular polyball ${\bf B}_{({\bf n}^{(1)},\ldots, {\bf n}^{(p)})}(\cH_1\otimes \cdots \otimes \cH_p)$, has trace class defect,
 and
 $$
 \text{\rm curv}\, (\widetilde {\bf X})=\prod_{i=1}^p \text{\rm curv}\,({\bf X}^{(i)}).
 $$
 \end{theorem}
 \begin{proof} Note that, for each  ${\bf m}^{(i)}\in \ZZ_+^{\kappa(i)}$ with $0\leq {\bf m}^{(i)}\leq (1,\ldots, 1)$ and $i\in \{1,\ldots, p\}$, we have
 $$
 {\bf \Delta}_{\widetilde {\bf X}}^{({\bf m}^{(1)},\ldots, {\bf m}^{(p)})}(I_{\cH_1\otimes \cdots \otimes \cH_p})
 ={\bf \Delta}^{{\bf m}^{(1)}}_{{\bf X}^{(1)}}(I_{\cH_1})\otimes \cdots \otimes {\bf \Delta}^{{\bf m}^{(p)}}_{{\bf X}^{(p)}}(I_{\cH_p})\geq 0,
 $$
 which shows that $\widetilde {\bf X}$ is in the regular polyball ${\bf B}_{({\bf n}^{(1)},\ldots, {\bf n}^{(p)})}(\cH_1\otimes \cdots \otimes \cH_p)$ and
 $$
 \text{\rm trace}\,\left[{\bf \Delta}_{\widetilde {\bf X}} (I_{\cH_1\otimes \cdots \otimes \cH_p})\right]
 =\text{\rm trace}\,\left[{\bf \Delta}_{{\bf X}^{(1)}}(I_{\cH_1})\right]\cdots \text{\rm trace}\,\left[{\bf \Delta}_{{\bf X}^{(p)}}(I_{\cH_p})\right]<\infty.
 $$
 Let ${\bf q}^{(i)}:=(q_1^{(i)}, \ldots, q_{\kappa(i)}^{(i)})\in \ZZ_+^{\kappa(i)} $
 for $i\in \{1,\ldots, p\}$. According to Corollary \ref{curva-maps}, we have
\begin{equation*}
\begin{split}
&\text{\rm curv}\,(\widetilde{\bf X})\\
&=
\lim_{{\bf q}^{(1)}\in \ZZ_+^{\kappa(1)},\ldots,{\bf q}^{(p)}\in \ZZ_+^{\kappa(p)}}
\frac{\text{\rm trace}\,\left[ \Phi_{X_1^{(1)}}^{q_1^{(1)}}\circ \cdots \circ
\Phi_{X_{\kappa(1)}^{(1)}}^{q_{\kappa(1)}^{(1)}}\left({\bf \Delta}_{{\bf X}^{(1)}}(I_{\cH_1})\right)
\otimes \cdots \otimes
\Phi_{X_1^{(p)}}^{q_1^{(p)}}\circ \cdots \circ
\Phi_{X_{\kappa(p)}^{(p)}}^{q_{\kappa(p)}^{(p)}}\left({\bf \Delta}_{{\bf X}^{(p)}}(I_{\cH_p})\right)\right]}
{\left[ (n_1^{(1)})^{q_1^{(1)}}\cdots (n_{\kappa({1})}^{(1)})^{q_{\kappa(1)}^{(1)}} \right]\cdots \left[ (n_1^{(p)})^{q_1^{(p)}}\cdots (n_{\kappa(p)}^{(p)})^{q_{\kappa(p)}^{(p)}} \right]}\\
&=
\lim_{{\bf q}^{(1)}\in \ZZ_+^{\kappa(1)}}
\frac{\text{\rm trace}\,\left[ \Phi_{X_1^{(1)}}^{q_1^{(1)}}\circ \cdots \circ
\Phi_{X_{\kappa(1)}^{(1)}}^{q_{\kappa(1)}^{(1)}}\left({\bf \Delta}_{{\bf X}^{(1)}}(I_{\cH_1})\right)
           \right]}
{(n_1^{(1)})^{q_1^{(1)}}\cdots (n_{\kappa(1)}^{(1)})^{q_{\kappa(1)}^{(1)}}}
\cdots
\lim_{{\bf q}^{(p)}\in \ZZ_+^{\kappa(p)}}\frac{\text{\rm trace}\,\left[
\Phi_{X_1^{(p)}}^{q_1^{(p)}}\circ \cdots \circ
\Phi_{X_{\kappa(p)}^{(p)}}^{q_{\kappa(p)}^{(p)}}\left({\bf \Delta}_{{\bf X}^{(p)}}(I_{\cH_p})\right)\right]}
{ (n_1^{(p)})^{q_1^{(p)}}\cdots (n_{\kappa(p)}^{(p)})^{q_{\kappa(p)}^{(p)}} }\\
&=\prod_{i=1}^p \text{\rm curv}\,({\bf X}^{(i)}).
\end{split}
\end{equation*}
The proof is complete.
\end{proof}

\begin{proposition} Let ${\bf n}^{(i)}:=(n_1^{(i)},\ldots, n_{\kappa(i)}^{(i)})\in \NN^{\kappa(i)}$  and  ${\bf X}^{(i)}\in B(\cH_i)^{n_1^{(i)}}\times \cdots  \times B(\cH_i)^{n_{\kappa(i)}^{(i)}}$ for $i\in \{1,\ldots, p\}$.
Then each ${\bf X}^{(i)}$ is unitarily equivalent to
${\bf S}^{({\bf n}^{(i)})}\otimes I_{\cE_i}$  for some  Hilbert space $\cE_i$  with  $\dim \cE_i<\infty$, if and only if the following conditions are satisfied.
 \begin{enumerate}
 \item[(i)]
 The ampliation $\widetilde {\bf X}:=(\widetilde {\bf X}^{(1)},\ldots, \widetilde {\bf X}^{(p)})\in {\bf B}_{({\bf n}^{(1)},\ldots, {\bf n}^{(p)})}(\cH_1\otimes \cdots \otimes \cH_p)$ is  a pure element  with finite rank.
 \item[(ii)]  Either $\widetilde {\bf X}$ or each $\widetilde {\bf X}^{(i)}$, $i\in \{1,\ldots, p\}$,  has characteristic function.
 \item[(iii)]
$ \text{\rm curv}\,(\widetilde{\bf X})=\rank [{\bf \Delta}_{\widetilde{\bf X}}(I)].
$
\end{enumerate}
In this case, the Berezin kernel ${\bf K_{\widetilde{\bf X}}}$ is a unitary operator which is unitarily equivalent to the tensor product  ${\bf K}_{{\bf X}^{(1)}}\otimes \cdots \otimes {\bf K}_{{\bf X}^{(p)}}$, and
$$ X_{r,s}^{(i)}={\bf K}^*_{{\bf X}^{(i)}} ({\bf S}^{({\bf n}^{(i)})}_{r,s}\otimes I_{\overline{{\bf \Delta_{{\bf X}^{(i)}}}(I) (\cH_i)}}){\bf K}_{{\bf X}^{(i)}}
$$
for all $i\in \{1,\ldots,p\}$, $r\in \{1,\ldots, \kappa(i)\}$, and $s\in \{1,\ldots, n_r^{(i)}\}$,
where ${\bf S}^{({\bf n}^{(i)})}$  is  the universal model of the polyball ${\bf B}_{{\bf n}^{(i)}}$ and ${\bf K}_{{\bf X}^{(i)}}$ is the associated Berezin kernel of ${\bf X}^{(i)}$.
\end{proposition}
\begin{proof} The direct implication  follows due to  Theorem \ref{comp-inv}. We prove the converse.
Note that since $\widetilde {\bf X}:=(\widetilde {\bf X}^{(1)},\ldots, \widetilde {\bf X}^{(p)})\in {\bf B}_{({\bf n}^{(1)},\ldots, {\bf n}^{(p)})}(\cH_1\otimes \cdots \otimes \cH_p)$ is  a pure element  with finite rank, each element
${\bf X}^{(i)}$  is a pure element in the polyball ${\bf B}_{{\bf n}^{(i)}}(\cH_i)$ and has finite rank. The Berezin kernel of $\widetilde{\bf X}$,
$$
{\bf K}_{\widetilde {\bf X}}:\cH_1\otimes \cdots \otimes \cH_p\to \left( \otimes_{r=1}^{\kappa(1)}F^2(H_{n^{(1)}_r})\right)\otimes \cdots \otimes \left( \otimes_{r=1}^{\kappa(p)}F^2(H_{n^{(p)}_r})\right)\otimes \overline{{\bf \Delta_{\widetilde {\bf X}}}(I) (\cH_1
\otimes \cdots \otimes \cH_p)},
$$
can be identified,  up to a unitary equivalence,
with
$
{\bf K}_{{\bf X}^{(1)}}\otimes \cdots \otimes {\bf K}_{{\bf X}^{(p)}}$
which is acting from $\cH_1\otimes \cdots \otimes \cH_p$ to
 $$\left[\left( \otimes_{r=1}^{\kappa(1)}F^2(H_{n^{(1)}_r})\right)\otimes \overline{{\bf \Delta}_{{\bf X}^{(1)}}(I) (\cH_1)}\right]\otimes \cdots \otimes \left[\left( \otimes_{r=1}^{\kappa(p)}F^2(H_{n^{(p)}_r})\right)\otimes \overline{{\bf \Delta}_{ {\bf X}^{(p)}}(I) (\cH_p)}\right].
$$
Assume that $\widetilde {\bf X}$ has characteristic function.
Since $\text{\rm curv}\,(\widetilde{\bf X})=\rank {\bf \Delta}_{\widetilde{\bf X}}(I)$, Theorem  \ref{comp-inv} implies that ${\bf K}_{\widetilde {\bf X}}$ is a unitary operator. Consequently,  each Berezin kernel  ${\bf K}_{{\bf X}^{(i)}}$ is a unitary operator  and
$$ X_{r,s}^{(i)}={\bf K}^*_{{\bf X}^{(i)}} ({\bf S}^{({\bf n}^{(i)})}_{r,s}\otimes I_{\overline{{\bf \Delta_{{\bf X}^{(i)}}}(I) (\cH_i)}}){\bf K}_{{\bf X}^{(i)}}
$$
for all $i\in \{1,\ldots,p\}$, $r\in \{1,\ldots, \kappa(i)\}$, and $s\in \{1,\ldots, n_r^{(i)}\}$,
where ${\bf S}^{({\bf n}^{(i)})}$  is  the universal model of the polyball ${\bf B}_{{\bf n}^{(i)}}$.
Now, assume that each $\widetilde {\bf X}^{(i)}$, $i\in \{1,\ldots, p\}$,  has characteristic function. Note that, due to Theorem \ref{multiplicative-prop} relation $\text{\rm curv}\,(\widetilde{\bf X})=\rank [{\bf \Delta}_{\widetilde{\bf X}}(I)]
$
is equivalent
$$
1\leq \prod_{i=1}^p \text{\rm curv}\,({\bf X}^{(i)})=\prod_{i=1}^p \rank [{\bf \Delta}_{{\bf X}^{(i)}}].
$$
Since $1\leq \text{\rm curv}\,({\bf X}^{(i)})\leq \rank [{\bf \Delta}_{{\bf X}^{(i)}}]$, we deduce that $\text{\rm curv}\,({\bf X}^{(i)})= \rank [{\bf \Delta}_{{\bf X}^{(i)}}]$. Applying Theorem  \ref{comp-inv} to ${\bf X}^{(i)}\in {\bf B}_{{\bf n}^{(i)}}(\cH_i)$ for each $i\in \{1,\ldots, p\}$, one can  complete the proof.
\end{proof}

\begin{corollary} Let $T_1,\ldots, T_p$ be contractions on a Hilbert space $\cH$. Then each of them is unitarily equivalent  to a unilateral shift of finite multiplicity if and only if  the ampliation
$$
\widetilde T:=(T_1\otimes I\otimes\cdots \otimes I, I\otimes T_2\otimes I\otimes \cdots \otimes I,\ldots , I\otimes \cdots \otimes I\otimes T_p)
$$
is a pure element of the regular polydisc ${\bf B}_{(1,\ldots, 1)}(\cH\otimes \cdots\otimes \cH)$, has finite rank, and
$\text{\rm curv}\,(\widetilde{ T})=\rank [{\bf \Delta}_{\widetilde{T}}(I)].
$
\end{corollary}

\section{Commutative polyballs and curvature invariant}

In this section, we introduce the  curvature invariant  associated with the elements of the  commutative polyball ${\bf B_n^c}(\cH)$  with the property  that they  have finite rank defects and characteristic functions. There  are commutative analogues of most of the results from  the previous sections.

 Since the case $n_1=\cdots=n_k=1$ was considered in the previous sections, we assume, throughout this section,  that the $k$-tuple ${\bf n}=(n_1,\ldots, n_k)$ has  at least one $n_i\geq 2$. The  commutative polyball ${\bf B_n^c}(\cH)$ is the set of all ${\bf X}=(X_1,\ldots, X_k)\in {\bf B_n}(\cH)$ with $X_i=(X_{i,1},\ldots, X_{i,n_i})$, where  the entries $X_{i,j}$ are commuting operators. According to \cite{Po-Berezin3}, the universal model associated with the abstract commutative polyball ${\bf B_n^c}$ is the $k$-tuple
 ${\bf B}:=({\bf B}_1,\ldots, {\bf B}_k)$   with ${\bf B}_i=({\bf B}_{i,1},\ldots, {\bf B}_{i, n_i})$, where
the operator ${\bf B}_{i,j}$ is acting on the tensor Hilbert space
$F_s^2(H_{n_1})\otimes\cdots\otimes F_s^2(H_{n_k})$  and is defined by setting
$${\bf B}_{i,j}:=\underbrace{I\otimes\cdots\otimes I}_{\text{${i-1}$
times}}\otimes B_{i,j}\otimes \underbrace{I\otimes\cdots\otimes
I}_{\text{${k-i}$ times}},\qquad i\in \{1,\ldots, k\}, j\in \{1,\ldots, n_i\},
$$
where ${B}_{i,j}:=P_{F_s^2(H_{n_i})} {S}_{i,j}|_{F_s^2(H_{n_i})}$ and  $F_s^2(H_{n_i})\subset F^2(H_{n_i})$ is the symmetric Fock space on $n_i$ generators. We recall that $F_s^2(H_{n_i})$ is coinvariant under  each operator ${S}_{i,j}$. For basic results concerning Berezin transforms and   model theory on the commutative polyball ${\bf B_n^c}(\cH)$ we refer the reader to \cite{Po-Berezin3}.

For each ${\bf q}=(q_1,\ldots, q_k)\in \ZZ_+^k$, define  the operator  $\widetilde{{\bf N}}_{\leq {\bf q}}$ on the tensor product $\otimes_{i=1}^k F_s^2(H_{n_i})$ by
 $$\widetilde{{\bf N}}_{\leq {\bf q}}:=
\sum_{{0\leq s_i\leq q_i}\atop{i\in\{1,\ldots, k\}}} \frac{1}{\text{\rm trace\,}\left[Q_{s_1}^{(1)}\otimes \cdots \otimes Q_{s_k}^{(k)}\right]}\,
Q_{s_1}^{(1)}\otimes \cdots \otimes Q_{s_k}^{(k)},
$$
where $Q_{s_i}^{(i)}:=P_{F_s^2(H_{n_i})} P_{s_i}^{(i)}|_{F_s^2(H_{n_i})}$ for  $i\in \{1,\ldots, k\}$ and the orthogonal projection $P_{s_i}^{(i)}$ is defined in Section 1.  Note that $Q_{s_i}^{(i)}$ is the orthogonal projection of the symmetric Fock space ${F_s^2(H_{n_i})}$ onto its subspace of homogeneous polynomials of degree $s_i$ and
$$
\text{\rm trace}\,[Q_{s_i}^{(i)}]=\frac{(s_i+1)\cdots (s_i+n_i-1)}{(n_i-1) !}=\left(\begin{matrix} s_i+n_i-1\\ n_i-1\end{matrix}\right).
$$
\begin{lemma}\label{identity2}
 Let $Y$ be a bounded operator on  $\otimes_{i=1}^k F_s^2(H_{n_i})\otimes \cH$ and  $\dim(\cH)<\infty$. Then
 $$
 \frac{\text{\rm trace\,}\left[(Q_{q_1}^{(1)}\otimes \cdots \otimes Q_{q_k}^{(k)}\otimes I_\cH)Y\right]}{\text{\rm trace\,}\left[Q_{q_1}^{(1)}\otimes \cdots \otimes Q_{q_k}^{(k)}\right]}
 =\text{\rm trace\,}\left[{\bf \Delta}_{{\bf B}\otimes I_\cH}(Y)(\widetilde{{\bf N}}_{\leq {\bf q}}\otimes I_\cH)\right],
$$
where ${\bf B}$ is the  universal model associated
  with the abstract commutative
  polyball ${\bf B_n^c}$.
\end{lemma}
\begin{proof}
Note that     ${\bf B}:=({\bf B}_1,\ldots, {\bf B}_k)$ is  a pure $k$-tuple  in the
 polyball $ {\bf B^c_n}(\otimes_{i=1}^kF_s^2(H_{n_i}))$.
 Due to  Corollary \ref{Taylor-rep},   we have
$$
Y=\sum_{s_1=0}^\infty \Phi_{{\bf B}_1\otimes I_\cH}^{s_1}\left(\sum_{s_2=0}^\infty \Phi_{{\bf B}_2\otimes I_\cH}^{s_2}\left(\cdots \sum_{s_k=0}^\infty \Phi_{{\bf B}_k\otimes I_\cH}^{s_k} \left({\bf \Delta}_{{\bf B}\otimes I_\cH}(Y)\right)\cdots \right)\right),
$$
where the iterated series converge  in the weak operator topology.
 Setting $Q_{\bf q}:=Q_{q_1}^{(1)}\otimes \cdots \otimes Q_{q_k}^{(k)}$ for ${\bf q}=(q_1,\ldots, q_k)\in \ZZ_+^k$, we deduce that
\begin{equation*}
\begin{split}
\left(Q_{\bf q}\otimes I_\cH\right) Y
 =
\left(Q_{\bf q}\otimes I_\cH\right)
\left(\sum_{s_1=0}^{q_1} \Phi_{{\bf B}_1\otimes I_\cH}^{s_1}\left(\sum_{s_2=0}^{q_2} \Phi_{{\bf B}_2\otimes I_\cH}^{s_2}\left(\cdots \sum_{s_k=0}^{q_k} \Phi_{{\bf B}_k\otimes I_\cH}^{s_k} \left({\bf \Delta}_{{\bf B}\otimes I_\cH}(Y)\right)\cdots \right)\right)\right).
\end{split}
\end{equation*}
Hence,  $\text{\rm trace\,}\left[\left(Q_{\bf q}\otimes I_\cH\right) Y\right]$ is equal to
\begin{equation*}
\begin{split}
&=
\sum_{s_1=0}^{q_1}\cdots \sum_{s_k=0}^{q_k} \text{\rm trace\,}
\left[\left(Q_{\bf q}\otimes I_\cH\right)\Phi_{{\bf B}_1\otimes I_\cH}^{s_1}\circ\cdots\circ \Phi_{{\bf B}_k\otimes I_\cH}^{s_k}({\bf \Delta}_{{\bf B}\otimes I_\cH}(Y))\right]\\
&=
\sum_{s_1=0}^{q_1}\cdots \sum_{s_k=0}^{q_k} \text{\rm trace\,}
\left[\left(Q_{\bf q}\otimes I_\cH\right)\sum_{{\alpha_1\in \FF_{n_1}^+, \cdots
\alpha_k\in \FF_{n_k}^+} \atop{|\alpha_1|=s_1\cdots,  |\alpha_k|=s_k}}
({\bf B}_{1,\alpha_1}\cdots {\bf B}_{k,\alpha_k}\otimes I_\cH){\bf \Delta}_{{\bf B}\otimes I_\cH}(Y)
({\bf B}_{k,\alpha_k}^*\cdots {\bf B}_{1,\alpha_1}^*\otimes I_\cH)
\right]\\
&=
\sum_{s_1=0}^{q_1}\cdots \sum_{s_k=0}^{q_k} \text{\rm trace\,}
\left[{\bf \Delta}_{{\bf B}\otimes I_\cH}(Y)\sum_{{\alpha_1\in \FF_{n_1}^+, \cdots
\alpha_k\in \FF_{n_k}^+} \atop{|\alpha_1|=s_1\cdots,  |\alpha_k|=s_k}}
({\bf B}_{k,\alpha_k}^*\cdots {\bf B}_{1,\alpha_1}^* Q_{\bf q} {\bf B}_{1,\alpha_1}\cdots {\bf B}_{k,\alpha_k})\otimes I_\cH
\right].
\end{split}
\end{equation*}
Since  the symmetric Fock space $F_s^2(H_{n_i})$ is coinvariant under  each operator ${S}_{i,j}$, one can see that
$$
{\bf B}_{i,\alpha_i}^*(I\otimes \cdots \otimes Q_{q_i}^{(i)}\otimes\cdots \otimes I)=(I\otimes \cdots \otimes Q_{q_i-s_i}^{(i)}\otimes\cdots\otimes I){\bf B}_{i,\alpha_i}^*
$$
for any  $i\in \{1,\ldots, k\}$ and $\alpha_i\in \FF_{n_i}^+$ with $|\alpha_i|=s_i\leq q_i$.
This can be used to deduce that
\begin{equation*}
\begin{split}
\sum_{{\alpha_1\in \FF_{n_1}^+, \cdots
\alpha_k\in \FF_{n_k}^+} \atop{|\alpha_1|=s_1\cdots,  |\alpha_k|=s_k}}
&{\bf B}_{k,\alpha_k}^*\cdots {\bf B}_{1,\alpha_1}^* Q_{\bf q} {\bf B}_{1,\alpha_1}\cdots {\bf B}_{k,\alpha_k}\\
&=
Q_{q_1-s_1}^{(1)}\otimes \cdots \otimes Q_{q_k-s_k}^{(k)}\sum_{{\alpha_1\in \FF_{n_1}^+, \cdots
\alpha_k\in \FF_{n_k}^+} \atop{|\alpha_1|=s_1\cdots,  |\alpha_k|=s_k}}
{\bf B}_{k,\alpha_k}^*\cdots {\bf B}_{1,\alpha_1}^*  {\bf B}_{1,\alpha_1}\cdots {\bf B}_{k,\alpha_k}\\
&=
\frac{\text{\rm trace\,}[Q_{q_1}^{(1)}]}{\text{\rm trace\,}[Q_{q_1-s_1}^{(1)}]} Q_{q_1-s_1}^{(1)}\otimes \cdots \otimes \frac{\text{\rm trace\,}[Q_{q_k}^{(k)}]}{\text{\rm trace\,}[Q_{q_k-s_k}^{(k)}]} Q_{q_k-s_k}^{(k)}.
\end{split}
\end{equation*}
Putting together the relations above, we obtain
\begin{equation*}
\begin{split}
&\text{\rm trace\,}\left[\left(Q_{\bf q}\otimes I_\cH\right) Y\right]\\
&=
\sum_{s_1=0}^{q_1}\cdots \sum_{s_k=0}^{q_k} \text{\rm trace\,}
\left\{{\bf \Delta}_{{\bf B}\otimes I_\cH}(Y)\left[\frac{\text{\rm trace\,}[Q_{q_1}^{(1)}]}{\text{\rm trace\,}[Q_{q_1-s_1}^{(1)}]} Q_{q_1-s_1}^{(1)}\otimes \cdots \otimes \frac{\text{\rm trace\,}[Q_{q_k}^{(k)}]}{\text{\rm trace\,}[Q_{q_k-s_k}^{(k)}]} Q_{q_k-s_k}^{(k)}\otimes I_\cH\right]\right\}\\
&=
\text{\rm trace\,}
\left\{{\bf \Delta}_{{\bf B}\otimes I_\cH}(Y)\left(\sum_{s_1=0}^{q_1}
\frac{\text{\rm trace\,}[Q_{q_1}^{(1)}]}{\text{\rm trace\,}[Q_{s_1}^{(1)}]} Q_{s_1}^{(1)}\right)\otimes \cdots \otimes \left(\sum_{s_k=0}^{q_k} \frac{\text{\rm trace\,}[Q_{q_k}^{(k)}]}{\text{\rm trace\,}[Q_{s_k}^{(k)}]} Q_{s_k}^{(k)}\right)\otimes I_\cH
\right\}.
\end{split}
\end{equation*}
Consequently, we have
\begin{equation*}
\begin{split}
\frac{\text{\rm trace\,}\left[\left(Q_{\bf q}\otimes I_\cH\right) Y\right]}
{\text{\rm trace\,}\left[Q_{\bf q}\right]}&=
\text{\rm trace\,}
\left\{{\bf \Delta}_{{\bf B}\otimes I_\cH}(Y)\left[\left(\sum_{s_1=0}^{q_1}
\frac{1}{\text{\rm trace\,}[Q_{s_1}^{(1)}]} Q_{s_1}^{(1)}\right)\otimes \cdots \otimes \left(\sum_{s_k=0}^{q_k} \frac{1}{\text{\rm trace\,}[Q_{s_k}^{(k)}]} Q_{s_k}^{(k)}\right)\otimes I_\cH\right]\right\}\\
&=\text{\rm trace\,}\left[{\bf \Delta}_{{\bf B}\otimes I_\cH}(Y)
\sum_{{0\leq s_i\leq q_i}\atop{i\in\{1,\ldots, k\}}} \frac{1}{\text{\rm trace\,}\left[Q_{\bf s}\right]}\,
Q_{\bf s}\otimes I_\cH\right]\\
&=
\text{\rm trace\,}\left[{\bf \Delta}_{{\bf B}\otimes I_\cH}(Y)
(\widetilde{{\bf N}}_{\leq \bf q}\otimes I_\cH)\right],
\end{split}
\end{equation*}
which completes the proof.
\end{proof}

Define the bounded linear operator $\widetilde{{\bf N}}$ on $ F_s^2(H_{n_1})\otimes \cdots\otimes F^2_s(H_{n_k})$ by setting
$$\widetilde{\bf N}:=\sum_{{\bf q}=(q_1,\ldots, q_k)\in \ZZ_+^k} \frac{1}{\text{\rm trace\,}\left[Q_{q_1}^{(1)}\otimes \cdots \otimes Q_{q_k}^{(k)}\right]}\,
Q_{q_1}^{(1)}\otimes \cdots \otimes Q_{q_k}^{(k)}.
$$

\begin{theorem}\label{trace-N2} Let $Y$ be a bounded operator acting on  $\otimes_{i=1}^k F_s^2(H_{n_i})\otimes \cH$ and  $\dim \cH<\infty$. If
$${\bf \Delta_{B\otimes {\it I}_\cH}}(Y)  :=(id-\Phi_{{\bf B}_1\otimes I_\cH})\circ\cdots\circ (id-\Phi_{{\bf B}_k\otimes I_\cH})(Y)\geq 0,
$$
then  ${\bf \Delta_{B\otimes I_\cH}}(Y)(\widetilde{\bf N}\otimes I_\cH)$ is a trace class operator and
$$
\text{\rm trace\,}\left[{\bf \Delta_{B\otimes {\it I}_\cH}}(Y)(\widetilde{\bf N}\otimes I_\cH)\right]
=
 \lim_{{\bf q}=(q_1,\ldots, q_k)\in \ZZ_+^k}
 \frac{\text{\rm trace\,}\left[(Q_{q_1}^{(1)}\otimes \cdots \otimes Q_{q_k}^{(k)}\otimes I_\cH)Y\right]}{\text{\rm trace\,}\left[Q_{q_1}^{(1)}\otimes \cdots \otimes Q_{q_k}^{(k)}\right]}.
 $$
\end{theorem}
\begin{proof}
First, note that  Lemma \ref{identity2} implies
$$
 0\leq \text{\rm trace\,}\left[{\bf \Delta_{B\otimes {\it I}_\cH}}(Y)(\widetilde{\bf N}_{\leq {\bf q}}\otimes I_\cH)\right]=\frac{\text{\rm trace\,}\left[(Q_{q_1}^{(1)}\otimes \cdots \otimes Q_{q_k}^{(k)}\otimes I_\cH)Y\right]}{\text{\rm trace\,}\left[Q_{q_1}^{(1)}\otimes \cdots \otimes Q_{q_k}^{(k)}\right]}
 \leq \|Y\|\dim \cH
$$
for any $q_i\geq 0$ and  $i\in \{1,\ldots, k\}$.
Since $\{\widetilde{\bf N}_{\leq {\bf q}}\}_{{\bf q}\in \ZZ_+^k}$ is an increasing multi-sequence
of positive operators convergent to $\widetilde{\bf N}$ and ${\bf \Delta_{B\otimes{\it I}_\cH}}(Y)\geq 0$
we deduce that
\begin{equation*}
\begin{split}
\text{\rm trace\,}\left[{\bf \Delta}_{{\bf B}\otimes I_\cH}(Y)(\widetilde{\bf N}\otimes I_\cH)\right]&=\lim_{{\bf q}\in \ZZ_+^k}\text{\rm trace\,}\left[{\bf \Delta}_{{\bf B}\otimes I_\cH}(Y)(\widetilde{\bf N}_{\leq {\bf q}}\otimes I_\cH)\right]\\
&=\sup_{{\bf q}\in \ZZ_+^k}\text{\rm trace\,}\left[{\bf \Delta}_{{\bf B}\otimes {\it I}_\cH}(Y)(\widetilde{\bf N}_{\leq {\bf q}}\otimes I_\cH)\right].
\end{split}
\end{equation*}
Therefore, ${\bf \Delta}_{{\bf B}\otimes I_\cH}(Y)(\widetilde{\bf N}\otimes I_\cH)$ is a trace class operator and the equality in the theorem holds. The proof is complete.
\end{proof}

We recall from \cite{Po-Berezin3} that the {\it constrained   Berezin kernel}
associated with  ${\bf T} \in  {\bf B^c_n}(\cH)$  is      the
bounded operator  \ ${\bf \widetilde K}_{\bf T}:\cH\to \otimes_{i=1}^k F_s^2(H_{n_i}) \otimes
\overline{{\bf \Delta_{T}}(I) (\cH)}$ defined by
$${\bf \widetilde K}_{\bf T}:=\left(P_{\otimes_{i=1}^k F_s^2(H_{n_i})  }\otimes I_{\overline{{\bf \Delta_{T}}(I) (\cH)}}\right){\bf K_{T}},
$$
where ${\bf K_{T}}$ is the noncommutative Berezin kernel associated with ${\bf T}\in
{\bf B_n}(\cH)$. One can easily see that the range of ${\bf K_T}$ is in $\otimes_{i=1}^k F_s^2(H_{n_i}) \otimes
\overline{{\bf \Delta_{T}}(I) (\cH)}$  and
  ${\bf \widetilde K_{T}} { T}^*_{i,j}= ({\bf B}_{i,j}^*\otimes I)  {\bf \widetilde K_{T}}.
    $
 for any $i\in \{1,\ldots, k\}$ and $j\in \{1,\ldots, n_i\}$.
 An element ${\bf T} \in  {\bf B^c_n}(\cH)$    is said to have
 {\it constrained characteristic function} if there is a multi-analytic operator ${\bf \widetilde \Theta_T}:\otimes_{i=1}^k F_s^2(H_{n_i})\otimes \cE\to \otimes_{i=1}^k F_s^2(H_{n_i})\otimes \overline{{\bf \Delta_T}(I)(\cH)}$ with respect the universal model ${\bf B}$, i.e.
 ${\bf \widetilde \Theta_T} ({\bf B}_{i,j}\otimes I)=({\bf B}_{i,j}\otimes I){\bf \widetilde \Theta_T} $ for any $i\in\{1,\ldots, k\}$ and $j\in\{1,\ldots, n_i\}$,  such that $\bf {\widetilde K_T \widetilde K_T^*} +{\bf \widetilde \Theta_T}{\bf \widetilde \Theta_T^*}={\it I}$. The characteristic function is essentially unique if we restrict it to its support.
 We proved in \cite{Po-Berezin3} that an element  ${\bf T}$ has constrained characteristic function if and only if
${\bf \Delta}_{{\bf B}\otimes I}(I-\bf {\widetilde K_T \widetilde K_T^*})\geq 0.
$
 We remark that if $k=1$, then  any element in $ {\bf B^c_n}(\cH)$ has characteristic function.

In the commutative setting,  we call  the bounded operator ${\bf \Delta}_{{\bf B}\otimes {\it  I}_\cH}(\bf {\widetilde K_T \widetilde K_T^*})(\widetilde{\bf N}\otimes {\it I}_\cH)$ the {\it curvature operator} associated with ${\bf T}\in {\bf B_n^c}(\cH)$.

\begin{theorem} \label{index3} If  ${\bf T} \in  {\bf B_n^c}(\cH)$   has finite rank and characteristic function, then the curvature operator associated with ${\bf T}$  is trace class and \begin{equation*}\begin{split}
\text{\rm trace\,}\left[{\bf \Delta}_{{\bf B}\otimes I_\cH}(\bf {\widetilde K_T \widetilde K_T^*})(\widetilde{\bf N}\otimes {\it I}_\cH)\right]
&=\lim_{(q_1,\ldots, q_k)\in \ZZ_+^k}  \frac{\text{\rm trace\,}\left[ (Q_{q_1}^{(1)}\otimes \cdots \otimes Q_{q_k}^{(k)}\otimes I_\cH){\bf \widetilde K_{T}} {\bf\widetilde  K_{T}^*}\right]}{\text{\rm trace\,}\left[Q_{q_1}^{(1)}\otimes \cdots \otimes Q_{q_k}^{(k)}\right]}\\
&=\rank [{\bf \Delta_{T}}(I)]-\text{\rm trace\,}\left[{\bf \widetilde\Theta}_{\bf T}({\bf P}_\CC\otimes I){\bf \widetilde \Theta}_{\bf T}^* (\widetilde{\bf N}\otimes {\it I}_\cH)\right],
\end{split}
\end{equation*}
where ${\bf \widetilde \Theta}_{\bf T}$ is the constrained characteristic function of ${\bf T}$.
\end{theorem}
\begin{proof}

 Applying Theorem \ref{trace-N2} when $Y=I-\bf {\widetilde K_T \widetilde K_T^*}={\bf \widetilde \Theta_T}{\bf \widetilde \Theta_T^*}$ and  taking into account that the constrained characteristic function  is a multi-analytic operator with respect to ${\bf B}$,  we obtain
 \begin{equation*}\begin{split}
\lim_{(q_1,\ldots, q_k)\in \ZZ_+^k}  \frac{\text{\rm trace\,}\left[ (Q_{q_1}^{(1)}\otimes \cdots \otimes Q_{q_k}^{(k)}\otimes I_\cH){\bf \widetilde K_{T}} {\bf \widetilde K_{T}^*}\right]}{\text{\rm trace\,}\left[Q_{q_1}^{(1)}\otimes \cdots \otimes Q_{q_k}^{(k)}\right]}
&=\text{\rm trace\,}\left[{\bf \Delta}_{{\bf B}\otimes I}(\bf {\widetilde K_T \widetilde K_T^*})(\widetilde{\bf N}\otimes {\it I}_\cH)\right]\\
&=\rank [{\bf \Delta_{T}}(I)]-\text{\rm trace\,}\left[{\bf \Delta}_{{\bf B}\otimes {\it I}}(\bf {\widetilde \Theta_T \widetilde\Theta_T^*})(\widetilde{\bf N}\otimes {\it I}_\cH)\right]\\
&=\rank [{\bf \Delta_{T}}(I)]-\text{\rm trace\,}\left[{\bf \widetilde \Theta}_{\bf T}({\bf P}_\CC\otimes I){\bf \widetilde\Theta}_{\bf T}^* (\widetilde{\bf N}\otimes I_\cH)\right].
\end{split}
\end{equation*}
The proof is complete.
\end{proof}

Given an element ${\bf T}$ in the  commutative polyball ${\bf B_n^c}(\cH)$  with the property  that  it has finite rank and characteristic function, we introduce  its curvature by setting
$$
\text{\rm curv}_c({\bf T}):=
\lim_{m\to\infty}\frac{1}{\left(\begin{matrix} m+k\\ k\end{matrix}\right)}\sum_{{q_1\geq 0,\ldots, q_k\geq 0}\atop {q_1+\cdots +q_k\leq m}} \frac{\text{\rm trace\,}\left[ {\bf \widetilde K_{T}^*} (Q_{q_1}^{(1)}\otimes \cdots \otimes Q_{q_k}^{(k)}\otimes I_\cH){\bf \widetilde K_{T}}\right]}{\text{\rm trace\,}\left[Q_{q_1}^{(1)}\otimes \cdots \otimes Q_{q_k}^{(k)}\right]}.
$$

\begin{theorem} \label{commutative-curva} If  ${\bf T} \in  {\bf B_n^c}(\cH)$   has finite rank and characteristic function, then the curvature $\text{\rm curv}_c({\bf T})$  exists and satisfies the asymptotic formulas
\begin{equation*}\begin{split}
\text{\rm curv}_c({\bf T})
&=
\lim_{m\to\infty}\frac{1}{\left(\begin{matrix} m+k\\ k\end{matrix}\right)}\sum_{{q_1\geq 0,\ldots, q_k\geq 0}\atop {q_1+\cdots +q_k\leq m}} \frac{\text{\rm trace\,}\left[ \Phi_{T_1}^{q_1}\circ \cdots \circ \Phi_{T_k}^{q_k}({\bf \Delta_{T}}(I))\right]}{\prod_{i=1}^k{\text{\rm trace\,}\left(Q_{q_i}^{(i)}\right)}}\\
&=\lim_{m\to\infty}\frac{1}{\left(\begin{matrix} m+k-1\\ k-1\end{matrix}\right)}\sum_{{q_1\geq 0,\ldots, q_k\geq 0}\atop {q_1+\cdots +q_k=m}} \frac{\text{\rm trace\,}\left[ \Phi_{T_1}^{q_1}\circ \cdots \circ \Phi_{T_k}^{q_k}({\bf \Delta_{T}}(I))\right]}{\prod_{i=1}^k{\text{\rm trace\,}\left(Q_{q_i}^{(i)}\right)}}\\
&=\lim_{(q_1,\ldots, q_k)\in \ZZ_+^k} \frac{\text{\rm trace\,}\left[ \Phi_{T_1}^{q_1}\circ \cdots \circ \Phi_{T_k}^{q_k}({\bf \Delta_{T}}(I))\right]}{\prod_{i=1}^k{\text{\rm trace\,}\left(Q_{q_i}^{(i)}\right)}}\\
&= n_1 !\cdots n_k !
 \lim_{q_1\to\infty}\cdots\lim_{q_k\to\infty}
\frac{\text{\rm trace\,}\left[(id-\Phi_{T_1}^{q_1+1})\circ\cdots\circ (id-\Phi_{T_k}^{q_k+1})(I)\right]}
{q_1^{n_1}\cdots q_k^{n_k}}.
\end{split}
\end{equation*}
\end{theorem}

\begin{proof}  Since the range of the Berezin kernel ${\bf K_T}$ is in the Hilbert space $\otimes_{i=1}^k F_s^2(H_{n_i}) \otimes
\overline{{\bf \Delta_{T}}(I) (\cH)}$  and ${\bf \widetilde K}_{\bf T}:=\left(P_{\otimes_{i=1}^k F_s^2(H_{n_i})  }\otimes I_{\overline{{\bf \Delta_{T}}(I) (\cH)}}\right){\bf K_{T}},
$
relation \eqref{connection} implies
\begin{equation} \label{connection2}
 {\bf \widetilde K_{T}^*} (Q_{q_1}^{(1)}\otimes \cdots \otimes Q_{q_k}^{(k)}\otimes I_\cH){\bf \widetilde K_{T}}=
  \Phi_{T_1}^{q_1}\circ \cdots \circ \Phi_{T_k}^{q_k}({\bf \Delta_{T}}(I))
  \end{equation}
  for any $q_1,\ldots,q_k\in \ZZ^+$. Hence, and due to  Theorem \ref{index3},
  we deduce that
  $$
  \lim_{(q_1,\ldots, q_k)\in \ZZ_+^k}  x_{ \bf q}
 =
  \lim_{(q_1,\ldots, q_k)\in \ZZ_+^k}  \frac{\text{\rm trace\,}\left[ (Q_{q_1}^{(1)}\otimes \cdots \otimes Q_{q_k}^{(k)}\otimes I_\cH){\bf \widetilde K_{T}} {\bf\widetilde  K_{T}^*}\right]}{\text{\rm trace\,}\left[Q_{q_1}^{(1)}\otimes \cdots \otimes Q_{q_k}^{(k)}\right]},
  $$
where  $ x_{ \bf q}:=\frac{\text{\rm trace\,}\left[ \Phi_{T_1}^{q_1}\circ \cdots \circ \Phi_{T_k}^{q_k}({\bf \Delta_{T}}(I))\right]}{\prod_{i=1}^k{\text{\rm trace\,}\left(Q_{q_i}^{(i)}\right)}}$ for ${\bf q}=(q_1,\ldots, q_k)\in \ZZ_+^k$.
Let ${\bf T}:=(T_1,\ldots, T_k)\in  {\bf B_n^c}(\cH)$ and ${\bf q}=(q_1,\ldots, q_k)\in \ZZ_+^k$. If $X\in \cT^+(\cH)$, we have
\begin{equation*}
\begin{split}
\text{\rm trace\,}[\Phi_{T_i}^{q_i}(X)]
&= \text{\rm trace\,}\left[ \sum_{\alpha_i\in \FF_{n_i}^+, |\alpha_i|=q_i} T_{i,\alpha_i} XT_{i,\alpha_i}^*\right]
=\text{\rm trace\,}\left[ \left(\sum_{\alpha_i\in \FF_{n_i}^+, |\alpha_i|=q_i} T_{i,\alpha_i}^* T_{i,\alpha_i}\right) X\right]\\
&\leq \left\|\sum_{\alpha_i\in \FF_{n_i}^+, |\alpha_i|=q_i} T_{i,\alpha_i}^* T_{i,\alpha_i}\right\|\text{\rm trace\,}(X)
\leq
  \text{\rm trace\,}[Q_{q_i}^{(i)}]\, \text{\rm trace\,}(X).
\end{split}
\end{equation*}
The latter equality was proved  by Arveson in \cite{Arv1}.
Applying the inequality above repeatedly, we obtain
\begin{equation} \label{ine-tr}
\begin{split}
 \frac{\text{\rm trace\,}\left[ \Phi_{T_1}^{q_1}\circ \cdots \circ \Phi_{T_k}^{q_k}({\bf \Delta_{T}}(I))\right]}{\prod_{i=1}^k{\text{\rm trace\,}\left(Q_{q_i}^{(i)}\right)}}\leq  \text{\rm trace\,}\left[{\bf \Delta_{T}}(I)\right].
\end{split}
\end{equation}

Due to Theorem \ref{trace-N2} and  Theorem \ref{index3}, the multi-sequence $\{x_{ \bf q}\}_{{\bf q}=(q_1,\ldots, q_k)\in \ZZ_+^k}$ is decreasing with respect to each of the indices $q_1, \ldots, q_k$, and $L:=\lim_{{\bf q} \in \ZZ_+^k}x_{ \bf q} $ exists. Given $\epsilon >0$, let $N_0\in \NN$ be such that
$
\left|x_{ \bf q}-L\right|< \epsilon$ for any $q_1\geq N_0,\ldots q_k\geq N_0$.
Consider the sets $A_1,\ldots, A_k, B_{N_0}$ defined in the proof of Lemma \ref{basic}.  Using relation \eqref{ine-tr}, we have  $x_{\bf q}\leq \text{\rm trace\,}[{\bf \Delta_T(I)}]$ for any ${\bf q}\in \ZZ_+^k$. Hence,  we deduce that
\begin{equation*}
\begin{split}
\sum_{{\bf q}=(q_1,\ldots, q_k)\in \cup_{i=1}^k A_i} x_{\bf q}
&\leq \sum_{i=1}^k \text{\rm card\,} (A_i)  \text{\rm trace\,}[{\bf \Delta_T(I)}]\leq kN_0 \left(\begin{matrix} m+k-2\\ k-2\end{matrix}\right) \text{\rm trace\,}[{\bf \Delta_T(I)}].
\end{split}
\end{equation*}
Now, as in the proof of Lemma \ref{basic}, we obtain that
\begin{equation*}
\left|\frac{1}{\left(\begin{matrix} m+k-1\\ k-1\end{matrix}\right)}\sum_{{q_1\geq 0,\ldots, q_k\geq 0}\atop {q_1+\cdots +q_k=m}} x_{\bf q}  -L\right|<\epsilon
\end{equation*}
for  $m$ big enough, which shows that
$$L:=\lim_{{\bf q} \in \ZZ_+^k}x_{ \bf q}=\lim_{m\to\infty} \frac{1}{\left(\begin{matrix} m+k-1\\ k-1\end{matrix}\right)}\sum_{{q_1\geq 0,\ldots, q_k\geq 0}\atop {q_1+\cdots +q_k=m}} x_{\bf q}.
$$
  Using Stolz-Ces\` aro  convergence theorem, we deduce that
\begin{equation*}
\lim_{m\to\infty}\frac{1}{\left(\begin{matrix} m+k\\ k\end{matrix}\right)}\sum_{{q_1\geq 0,\ldots, q_k\geq 0}\atop {q_1+\cdots +q_k\leq m}}x_{\bf q} =L.
\end{equation*}
Now, using relation \eqref{connection2}, we conclude
$$
\lim_{m\to\infty}\frac{1}{\left(\begin{matrix} m+k\\ k\end{matrix}\right)}\sum_{{q_1\geq 0,\ldots, q_k\geq 0}\atop {q_1+\cdots +q_k\leq m}} \frac{\text{\rm trace\,}\left[ {\bf \widetilde K_{T}^*} (Q_{q_1}^{(1)}\otimes \cdots \otimes Q_{q_k}^{(k)}\otimes I_\cH){\bf \widetilde K_{T}}\right]}{\text{\rm trace\,}\left[Q_{q_1}^{(1)}\otimes \cdots \otimes Q_{q_k}^{(k)}\right]} =L,
$$
and, consequently, the curvature $\text{\rm curv}_c({\bf T})$ exists.
Now, we prove the last equality in the theorem. Since $L=\lim_{q_1\to\infty}\cdots\lim_{q_k\to\infty} x_{\bf q} $
 and setting $y_{q_1}:=\lim_{q_2\to\infty}\cdots\lim_{q_k\to\infty} x_{\bf q}  $, an application of  Stolz-Ces\` aro  convergence theorem to the sequence $\{y_{q_1}\}_{q_1=0}^\infty$ implies
$$
\lim_{q_1\to \infty}\frac{1}{\sum_{s_1=0}^{q_1}\text{\rm trace\,}\left[Q_{s_1}^{(1)}\right]}\sum_{s_1=0}^{q_1}
\lim_{q_2\to\infty}\cdots\lim_{q_k\to\infty} \frac{\text{\rm trace\,}\left[ \Phi_{T_1}^{s_1}\circ \Phi_{T_2}^{q_2}\circ\cdots \circ \Phi_{T_k}^{q_k}({\bf \Delta_{T}}(I))\right]}{\prod_{i=2}^k{\text{\rm trace\,}\left(Q_{q_i}^{(i)}\right)}}=\lim_{q_1\to\infty}y_{q_1}=L.
$$
Similarly, setting $z_{q_2}:=\lim_{q_3\to\infty}\cdots\lim_{q_k\to\infty} \frac{\text{\rm trace\,}\left[ \Phi_{T_1}^{s_1}\circ \Phi_{T_2}^{q_2}\circ\cdots \circ \Phi_{T_k}^{q_k}({\bf \Delta_{T}}(I))\right]}{\prod_{i=2}^k{\text{\rm trace\,}\left(Q_{q_i}^{(i)}\right)}}$, we deduce that
$$
\lim_{q_2\to \infty}\frac{1}{\sum_{s_2=0}^{q_2}\text{\rm trace\,}\left[Q_{s_2}^{(2)}\right]}\sum_{s_2=0}^{q_2}
\lim_{q_3\to\infty}\cdots\lim_{q_k\to\infty} \frac{\text{\rm trace\,}\left[ \Phi_{T_1}^{s_1}\circ \Phi_{T_2}^{s_2}\circ\Phi_{T_3}^{q_3}\circ\cdots \circ \Phi_{T_k}^{q_k}({\bf \Delta_{T}}(I))\right]}{\prod_{i=3}^k{\text{\rm trace\,}\left(Q_{q_i}^{(i)}\right)}}=\lim_{q_2\to\infty}z_{q_2}.
$$
Continuing this process and putting together these results, we obtain
$$
L=\lim_{q_1\to\infty}\cdots\lim_{q_k\to\infty}
\frac{\text{\rm trace\,}\left[(id-\Phi_{T_1}^{q_1+1})\circ\cdots\circ (id-\Phi_{T_k}^{q_k+1})(I)\right]}
{ \prod_{i=1}^k \left(\sum_{s_i=0}^{q_i}\text{\rm trace\,}\left[Q_{s_i}^{(i)}\right]\right)}.
$$
Note that
$
\sum_{s_i=0}^{q_i}\text{\rm trace\,}\left[Q_{s_i}^{(i)}\right]=
\frac{(q_i+1)(q_i+2)\cdots (q_i+n_i)}{n_i !}
$
and, consequently,
$$
L= n_1 !\cdots n_k !
 \lim_{q_1\to\infty}\cdots\lim_{q_k\to\infty}
\frac{\text{\rm trace\,}\left[(id-\Phi_{T_1}^{q_1+1})\circ\cdots\circ (id-\Phi_{T_k}^{q_k+1})(I)\right]}
{q_1^{n_1}\cdots q_k^{n_k}}.
$$
The proof is complete.
\end{proof}
We remark that, due to relation \eqref{connection2}, all the asymptotic  formulas in Theorem \ref{commutative-curva}  can be written in terms of the constrained Berezin kernel ${\bf \widetilde K_T}$.

\begin{theorem} If  ${\bf T}=({ T}_1,\ldots, { T}_k)\in B(\cH)^{n_1}\times \cdots \times B(\cH)^{n_k}$, then the following statements are equivalent:
\begin{enumerate}
\item[(i)] ${\bf T}$ is unitarily equivalent  to ${\bf B}\otimes I_\cK$ for some finite dimensional  Hilbert space   $\cK$;
    \item[(ii)] ${\bf T}$ is a pure finite rank  element in the   polyball  ${\bf B^c_n}(\cH)$ such that ${\bf \Delta}_{{\bf B}\otimes I}(I-{\bf \widetilde K_{T}}{\bf \widetilde  K_{T}^*})\geq 0$, and
    $$\text{\rm curv}_c({\bf T})=\rank [{\bf \Delta_{T}}(I)].
    $$
\end{enumerate}
In this case, the constrained Berezin kernel ${\bf \widetilde  K_T}$ is a unitary operator and
$${\bf T}_{i,j}={\bf \widetilde K_T^*} ({\bf B}_{i,j}\otimes I_{\overline{{\bf \Delta_{T}}(I) (\cH)}}){\bf\widetilde  K_T}, \qquad i\in\{1,\ldots,k\},  j\in\{1,\ldots, n_i\}.
$$
\end{theorem}
\begin{proof} The proof of the implication $(i)\implies (ii)$ is similar to that of Theorem \ref{comp-inv}, but uses Theorem \ref{commutative-curva}.

Assume that item (ii) holds. Since ${\bf \Delta_{B\otimes {\it I}}}(I-{\bf \widetilde  K_{T}}{\bf \widetilde K_{T}^*})\geq 0$,  the tuple ${\bf T}$   has characteristic function ${\bf \widetilde \Theta_T}:\otimes_{i=1}^k F_s^2(H_{n_i})\otimes \cE\to \otimes_{i=1}^k F_s^2(H_{n_i})\otimes \overline{{\bf \Delta_T}(I)(\cH)}$ which is  a multi-analytic operator with respect to the universal model ${\bf B}$ and
$\bf {\widetilde K_T \widetilde K_T^*} +{\bf \widetilde \Theta_T}{\bf \widetilde \Theta_T^*}={\it I}$. Since ${\bf T}$ is pure,  ${\bf \widetilde K}_{\bf T}$  is an isometry and, consequently, ${\bf \widetilde \Theta_T}$ is a partial isometry.
According to \cite{Po-Berezin3}, the  support of ${\bf \widetilde \Theta_T}$  satisfies the relation
 \begin{equation}
 \label{support}
 \supp ({\bf \widetilde \Theta_T})=\bigvee_{(\alpha)\in \FF_{n_1}^+\times\cdots
  \times \FF_{n_k}^+}({\bf B}_{(\alpha)}\otimes I_\cH) (\cM)
 =\otimes_{i=1}^k F_s^2(H_{n_i})\otimes \cL,
 \end{equation}
 where $\cL:=({\bf P}_\CC\otimes I_\cH)\overline{{\bf\widetilde  \Theta^*}
 (\otimes_{i=1}^k F_s^2(H_{n_i})\otimes \overline{{\bf \Delta_T}(I)(\cH)})}$.
Due to Theorem \ref{index3}, we  have
\begin{equation*}
\begin{split}
\text{\rm curv}_c({\bf T})&= \rank [{\bf \Delta_{T}}(I)]-\text{\rm trace\,}\left[{\bf \widetilde \Theta_T}({\bf P}_\CC\otimes I_\cL) {\bf \widetilde  \Theta_T^*}(\widetilde{\bf N}\otimes I_\cH)\right].
\end{split}
\end{equation*}
Since $\text{\rm curv}_c({\bf T})=\rank [{\bf \Delta_{T}}(I)]$, we deduce that
$\text{\rm trace\,}\left[{\bf \widetilde \Theta_T}({\bf P}_\CC\otimes I_\cL) {\bf \widetilde \Theta_T^*}(\widetilde{\bf N}\otimes I_\cH)\right]=0$. Since  the trace is faithful, we obtain
${\bf \widetilde \Theta_T}({\bf P}_\CC\otimes I_\cL) {\bf \widetilde \Theta_T^*}(\widetilde{\bf N}\otimes I_\cH)=0$, which implies
${\bf \widetilde \Theta_T}({\bf P}_\CC\otimes I_\cL) {\bf \widetilde \Theta_T^*}(Q_{q_1}^{(1)}\otimes \cdots \otimes Q_{q_k}^{(k)}\otimes I_\cH)=0$ for any $(q_1,\ldots, q_k)\in \ZZ_+^k$. Therefore,
${\bf \widetilde  \Theta_T}({\bf P}_\CC\otimes I_\cL) {\bf \widetilde \Theta_T^*}=0$.  Due to   relation \ref{support} and the fact that  ${\bf \widetilde \Theta_T}$ is a multi-analytic operator with respect to ${\bf B}$, we deduce that
 ${\bf \widetilde \Theta_T}=0$. Using relation  $\bf {\widetilde K_T\widetilde  K_T^*} +{\bf \widetilde\Theta_T}{\bf \widetilde\Theta_T^*}={\it I}$, we deduce that that ${\bf \widetilde K_T}$ is a co-isometry. Therefore, ${\bf \widetilde K_T}$ is a unitary operator. Since ${\bf T}_{i,j}={\bf\widetilde  K_T^*} ({\bf B}_{i,j}\otimes I_{\overline{{\bf \Delta_{T}}(I) (\cH)}}){\bf  \widetilde K_T}$ for $i\in\{1,\ldots,k\}$ and $j\in\{1,\ldots, n_i\}$,
the proof is complete.
\end{proof}

\begin{theorem} Let  ${\bf T}\in {\bf B_n^c}(\cH)$  have finite rank and  let $\cM$ be an invariant subspace under ${\bf T}$ such that  ${\bf T}|_\cM\in {\bf B_n^c}(\cM)$.  If ${\bf T}$ and ${\bf T}|_\cM$  have characteristic functions and  $\dim \cM^\perp<\infty$, then ${\bf T}|_\cM$ has finite rank and
$
\text{\rm curv}_c({\bf T})=\text{\rm curv}_c({\bf T}|_\cM).
$
\end{theorem}
\begin{proof} Note that $\rank [{\bf T}|_\cM]=\rank [{\bf \Delta_T}(P_\cM)]$. Since
${\bf \Delta_T}(P_\cM)={\bf \Delta_T}(I_\cH)-{\bf \Delta_T}(P_{\cM^\perp})$, we deduce that $\rank [{\bf T}|_\cM]<\infty$.
Since $\cM$ is an invariant subspace under ${\bf T}$, we have
\begin{equation*}
\begin{split}
\text{\rm trace}&\left[\left(id -\Phi_{T_1|_\cM}^{q_1+1}\right)\circ \cdots \circ\left(id -\Phi_{ T_k|_\cM}^{q_k+1}\right)(I_\cM)\right]\\
&= \text{\rm trace}\left[\left(id -\Phi_{T_1}^{q_1+1}\right)\circ \cdots \circ\left(id -\Phi_{ T_k}^{q_k+1}\right)(P_\cM)\right]\\
&\leq \text{\rm trace}\left[\left(id -\Phi_{T_1}^{q_1+1}\right)\circ \cdots \circ\left(id -\Phi_{ T_k}^{q_k+1}\right)(I_\cH)\right]+\text{\rm trace}\left[\left(id -\Phi_{T_1}^{q_1+1}\right)\circ \cdots \circ\left(id -\Phi_{ T_k}^{q_k+1}\right)(P_{\cM^\perp})\right].
\end{split}
\end{equation*}
 Taking into account that $\left\|\Phi_{T_1}^{p_1}\cdots \Phi_{T_k}^{p_k}(P_{\cM^\perp})\right\|\leq 1$ for any $(p_1,\ldots,p_k)\in \ZZ_+^k$, one can easily see that
$$\left\|\left(id +\Phi_{T_1}^{q_1+1}\right)\circ \cdots \circ\left(id +\Phi_{ T_k}^{q_k+1}\right)(P_{\cM^\perp})\right\|\leq 2^k
$$
for any $(q_1,\ldots,q_k)\in \ZZ_+^k$. On the other hand, it is easy to see  that the dimension of the range of the operator $\left(id +\Phi_{T_1}^{q_1+1}\right)\circ \cdots \circ\left(id +\Phi_{ T_k}^{q_k+1}\right)(P_{\cM^\perp})$ is less than or equal to
$$
\left(1+\text{\rm trace\,} (Q_{q_1+1}^{(1)})\right)\cdots \left(1+\text{\rm trace\,} (Q_{q_k+1}^{(k)})\right)
\dim\cM^\perp.
$$
Using the fact that if $A$ is a finite rank positive operator, then
$\text{\rm trace}\, ( A)\leq \|A\|\rank [A]$, we deduce that
\begin{equation*}
\begin{split}
\text{\rm trace\,}&\left[\left(id -\Phi_{T_1}^{q_1+1}\right)\circ \cdots \circ\left(id -\Phi_{ T_k}^{q_k+1}\right)(P_{\cM^\perp})\right]\\
&\leq
\text{\rm trace\,}\left[\left(id +\Phi_{T_1}^{q_1+1}\right)\circ \cdots \circ\left(id +\Phi_{ T_k}^{q_k+1}\right)(P_{\cM^\perp})\right]\\
&\leq 2^k \left(1+\text{\rm trace\,} (Q_{q_1+1}^{(1)})\right)\cdots \left(1+\text{\rm trace\,} (Q_{q_k+1}^{(k)})\right)
\text{\rm trace\,}[P_{\cM^\perp}].
\end{split}
\end{equation*}
Consequently, we have
\begin{equation*}
\begin{split}
&\left|\frac{\text{\rm trace\,}\left[\left(id -\Phi_{T_1}^{q_1+1}\right)\circ \cdots \circ\left(id -\Phi_{ T_k}^{q_k+1}\right)(I_\cH)\right]}{{q_1^{n_1}\cdots q_k^{n_k}}}-\frac{\text{\rm trace\,}\left[\left(id -\Phi_{T_1|_\cM}^{q_1+1}\right)\circ \cdots \circ\left(id -\Phi_{ T_k|_\cM}^{q_k+1}\right)(I_\cM)\right]}{{q_1^{n_1}\cdots q_k^{n_k}}}
 \right|\\
 &\qquad\qquad\leq \frac{2^k \left(1+\text{\rm trace\,} (Q_{q_1+1}^{(1)})\right)\cdots \left(1+\text{\rm trace\,} (Q_{q_k+1}^{(k)})\right)}{{q_1^{n_1}\cdots q_k^{n_k}}}\text{\rm trace}\,[P_{\cM^\perp}].
\end{split}
\end{equation*}
Since
$
\text{\rm trace\,}(Q_{q_i+1}^{(i)})=\frac{(q_i+2)(q_i+3)\cdots (q_i+n_i)}{(n_i-1) !},
$
we have
$\lim_{q_i\to\infty} \frac{\text{\rm trace\,}(Q_{q_i+1}^{(i)})}{q_i^{n_i}}=0$. Passing to the limit as $q_1\to\infty, \ldots, q_k\to\infty$ in the inequality above and using Theorem \ref{commutative-curva}, we deduce
  that
$
\text{\rm curv}_c({\bf T})-\text{\rm curv}_c({\bf T}|_\cM) =0.
$
The proof is complete.
\end{proof}

Combining Corollary 2.6 from \cite{Po-Berezin3} with Lemma \ref{sub-poly}, we deduce  that  if $\cM\subset \otimes_{i=1}^k F_s^2(H_{n_i})\otimes \cK$  is an invariant subspace under the universal model ${\bf B}\otimes I_\cK$, then $({\bf B}\otimes I_\cK)|_\cM$ is in the commutative polyball if and only if $\cM$ is a Beurling type invariant subspace for ${\bf B}\otimes I_\cK$.

The proof of the next lemma is similar to that of Lemma \ref{coincidence}. The only difference is that we need to use Theorem 2.7. We omit the proof.
\begin{lemma}\label{coincidence2} Let $\cM$ be  an invariant subspace of
$\otimes_{i=1}^k F_s^2(H_{n_i})\otimes \cE$ such that it does not contain nontrivial reducing subspaces for the universal model ${\bf B}\otimes I_\cE$, and let
 ${\bf T}:=P_{\cM^\perp} ({\bf B}\otimes I_\cE)|_{\cM^\perp}$. Then there is a unitary operator $Z:\overline{{\bf \Delta_T}(I)(\cH)}\to \cE$ such that
 $(I\otimes Z){\bf\widetilde K_T}=V$, where ${\bf \widetilde K_T}$ is the Berezin kernel associated with ${\bf T}$ and  $V$ is the injection of $\cM^\perp$ into
 $\otimes_{i=1}^k F_s^2(H_{n_i})\otimes \cE$.

 Moreover, ${\bf T}$ has characteristic function if and only if $\cM$ is a Beurling type invariant subspace.
\end{lemma}

Given a  Beurling type invariant subspace  $\cM$   of the  tensor product $F_s^2(H_{n_1})\otimes\cdots\otimes F_s^2(H_{n_k})\otimes \cE$, where $\cE$ is a finite dimensional Hilbert space,
 we introduce  its multiplicity   by setting
$$
m_c(\cM):=
\lim_{m\to\infty}\frac{1}{\left(\begin{matrix} m+k\\ k\end{matrix}\right)}\sum_{{q_1\geq 0,\ldots, q_k\geq 0}\atop {q_1+\cdots +q_k\leq m}} \frac{\text{\rm trace\,}\left[ P_\cM (Q_{q_1}^{(1)}\otimes \cdots \otimes Q_{q_k}^{(k)}\otimes I_\cE) \right]}{\prod_{i=1}^k{\text{\rm trace\,}\left(Q_{q_i}^{(i)}\right)}}.
$$
The multiplicity measures the size of the subspace $\cM$. Note that if $\cM=F_s^2(H_{n_1})\otimes\cdots\otimes F_s^2(H_{n_k})\otimes \cE$, then $m_c(\cM)=\dim \cE$.
In what follows,  we show that the multiplicity invariant exists.
However, it remains an open problem whether  the multiplicity invariant  exists for arbitrary invariant subspaces
of the tensor product $F_s^2(H_{n_1})\otimes\cdots\otimes F_s^2(H_{n_k})\otimes \cE$.
This is true for the polydisc, when $n_1=\cdots n_k=1$, and the symmetric Fock space, when $k=1$.

The proof of the next result is similar to that of Theorem \ref{multiplicity}, but uses Lemma \ref{coincidence2} and Theorem \ref{commutative-curva}. We shall omit it.

\begin{theorem}  Let $\cM$ be a Beurling type invariant subspace of   $F_s^2(H_{n_1})\otimes\cdots\otimes F_s^2(H_{n_k})\otimes \cE$, where $\cE$ is a finite dimensional Hilbert space. Then the  the multiplicity $m_c(\cM)$ exists and
satisfies the equations
 \begin{equation*}
\begin{split}
m_c(\cM)&=\lim_{(q_1,\ldots, q_k)\in \ZZ_+^k}  \frac{\text{\rm trace\,}\left[ P_\cM (Q_{q_1}^{(1)}\otimes \cdots \otimes Q_{q_k}^{(k)}\otimes I_\cE) \right]}{\text{\rm trace\,}\left[Q_{q_1}^{(1)}\otimes \cdots \otimes Q_{q_k}^{(k)}\right]}\\
&=\lim_{q_1\to\infty}\cdots\lim_{q_k\to\infty}
\frac{\text{\rm trace\,}\left[P_{\cM}(Q_{\leq(q_1,\ldots, q_k)}\otimes I_\cE)  \right]}{\text{\rm trace\,}\left[Q_{\leq(q_1,\ldots, q_k)}\right]}\\
&=\lim_{m\to\infty}\frac{1}{\left(\begin{matrix} m+k-1\\ k-1\end{matrix}\right)}\sum_{{q_1\geq 0,\ldots, q_k\geq 0}\atop {q_1+\cdots +q_k= m}} \frac{\text{\rm trace\,}\left[ P_\cM (Q_{q_1}^{(1)}\otimes \cdots \otimes Q_{q_k}^{(k)}\otimes I_\cE)\right]}{\text{\rm trace\,}\left[Q_{q_1}^{(1)}\otimes \cdots \otimes Q_{q_k}^{(k)}\right]}\\
&=\dim \cE -\text{\rm curv}_c({\bf M}),
\end{split}
\end{equation*}
where  ${\bf M}:=(M_1,\ldots, M_k)$ with $M_i:=(M_{i,1},\ldots, M_{i,n_i})$ and $M_{i,j}:=P_{\cM^\perp}({\bf B}_{i,j}\otimes I_\cE)|_{\cM^\perp}$.
\end{theorem}

We remark that there are commutative analogues of all the results from Section 4, concerning the continuity and multiplicative properties of the curvature and multiplicity invariants. Since the proofs are very similar we will omit them. We only mention the following result concerning the lower semi-continuity of the multiplicity invariant.
\begin{theorem}
Let $\cM$ and $\cM_m$ be Beurling type invariant subspaces in
$\otimes_{i=1}^k F_s^2(H_{n_i})\otimes \cE$ with $\dim \cE<\infty$.
If $\text{\rm WOT-}\lim_{m\to\infty}P_{\cM_m}=P_\cM$, then
$$
\liminf_{m\to\infty}m_c(\cM_m)\geq m_c(\cM).
$$
\end{theorem}

       %

      \end{document}